\documentclass[a4paper,12pt]{book}
\usepackage[T1]{fontenc}
\usepackage[utf8]{inputenc}
\usepackage{geometry}
\usepackage{amsmath,amssymb,graphicx,stmaryrd,latexsym}
\usepackage{amsthm}
\usepackage{dsfont}\let\mathbb\mathds
\usepackage[frenchb,english]{babel}
\usepackage[all]{xy}
\usepackage{amscd}
\usepackage{makeidx}

\newtheorem{theoreme}{Theorem}[section]
\newtheorem{definition}[theoreme]{Definition}
\newtheorem{lem}[theoreme]{Lemma}
\newtheorem{ex}[theoreme]{Example}
\newtheorem{prop}[theoreme]{Proposition}
\newtheorem{cor}[theoreme]{Corollary}

\theoremstyle{remark}
\newtheorem{rem}{Remark} 

\theoremstyle{Thmbis}
\newtheorem*{Thm}{Theorem}

\newcounter{diagram}[chapter]
\renewcommand\thediagram{\thechapter.\arabic{diagram}}
\newenvironment{diagram}[1]{
\refstepcounter{diagram}
\def\titredudiagramme{#1}
\vspace{4mm}
}{
\begin{center}
  {\rm Diagram \thediagram. \emph{\titredudiagramme}}
\end{center}
\vspace{8mm}
}

\date{}

\makeindex

\begin{document}

\begin{center}
{\Large UNIVERSIT\'E DE NICE SOPHIA ANTIPOLIS   --   UFR Sciences}\\
\vspace*{0.5cm}
\'Ecole Doctorale Sciences Fondamentales et Appliqu\'ees\\
\vspace*{2.5cm}
{\Large\bf TH\`ESE}\\
\vspace*{0.3cm}
pour obtenir le titre de\\
\vspace*{0.1cm}
{\Large\bf Docteur en Sciences}\\
\vspace*{0.1cm}
Sp\'ecialit\'e 
{\Large {\sc Math\'ematiques}}\\
\vspace*{0.8cm}
pr\'esent\'ee et soutenue par\\
\vspace*{0.1cm}
{\large\bf Kruna SEGRT}\\
\vspace*{1.5cm}
{\LARGE\bf Morita theory in enriched context}\\
\vspace*{1.5cm}
Th\`ese dirig\'ee par {\bf Clemens BERGER}\\
\vspace*{0.2cm}
soutenue le 24 février 2012\\
\vspace*{1.5 cm}
Membres du jury :\\
\vspace*{0.4cm}
\begin{tabular}{lllll}
M. & Clemens BERGER & & Directeur de th\`ese\\
M. & Denis-Charles CISINSKI & & Examinateur\\
M. & Paul-André MELLIES & & Rapporteur et Examinateur\\
M. & Stefan SCHWEDE & & Rapporteur\\
M. & Carlos SIMPSON & & Examinateur\\
Mme & Svjetlana TERZIC & & Co-directrice de th\`ese\\
M. & Rainer VOGT & & Rapporteur et Examinateur\\
\end{tabular}

\vspace*{2.5 cm}
Laboratoire Jean-Alexandre Dieudonn\'e, Universit\'e de Nice, Parc Valrose, 06108 Nice Cedex 2
\end{center}

\tableofcontents 

\chapter*{Introduction}\label{intro}

The aim of this thesis is to formulate and prove a homotopy-theoretical generalization of classical Morita theory. More precisely, we indicate sufficient conditions for a strong monad $T$ on a monoidal Quillen model category $\mathcal E$ in order that the homotopy category of $T$-algebras be equivalent to the homotopy category of modules of a certain monoid in $\mathcal E$. 


In order to state our homotopical Morita theorem we rely on the notion of tensorial strength which has been introduced by Anders Kock \cite{KockSF,KockMS} a long time ago. A strong functor $T:\mathcal A \rightarrow \mathcal B$ between categories $\mathcal A$ and $\mathcal B$, tensored over a closed symmetric monoidal category $\mathcal E$, is a functor equipped with a tensorial strength $$\sigma_{X,A}: X\otimes TA \rightarrow T(X\otimes A)$$ 
for any objects $X$ of $\mathcal E$ and $A$ of $\mathcal A$ satisfying some natural unit and associativity axioms. There is a similar notion of strong natural transformation.
If $\mathcal A$ and $\mathcal B$ are enriched and tensored over $\mathcal E$, then giving a strength for $T$ amounts to giving an enrichment of $T$ over $\mathcal E$. 
In particular, a given functor extends to a strong functor if and only if it extends to an enriched functor.
We rephrase these constructions of Kock in a $2$-categorical framework. This emphasizes the relative character of the notion of strength while Kock's original treatment was concentrated on monads. The $2$-categorical view point clarifies in particular the four axioms appearing in Kock's definition of a strong monad. We establish in Chapter 3 
\begin{Thm}\label{th2-cat}
The following 2-categories of tensored $\mathcal E$-categories are 2-isomorphic:
\begin{enumerate}
 \item The 2-category of strong functors and strong natural transformations of tensored $\mathcal E$-categories, $\mathbf{StrongCat}$; 
 \item The 2-category of $\mathcal E$-functors and $\mathcal E$-natural transformations of tensored $\mathcal E$-categories, $\mathcal E\mathbf{-Cat}$.
\end{enumerate}
\end{Thm}

The consequence of this 2-isomorphism is an equivalence between the notions of strong and enriched monads. 
In particular, under mild conditions the category of algebras over a strong monad $T$ on $\mathcal E$ is canonically enriched, tensored and cotensored over $\mathcal E$. 
This implies that the image $T(I)$ of the unit $I$ of $\mathcal E$ gets the structure of a monoid, through its identification with the endomorphism monoid $\underline{\mathcal Alg}_T(T(I),T(I))$ of the free $T$-algebra on $I$. More precisely, the strength of $T$ induces a morphism of monads $\lambda: -\otimes T(I)\to T$ which relates the categories of $T(I)$-modules and of $T$-algebras by a canonical adjunction. 
This can be considered as an embryonic form of the Morita theorem.
And indeed, our homotopical Morita theorem consists essentially in finding the right homotopical hypotheses in order to transform this formal adjunction into a Quillen equivalence. We prove in Chapter \ref{ChMthEcon} 

\begin{Thm}{(Homotopical Morita theorem)}
Let $\mathcal E$ be a cofibrantly generated monoidal model category with cofibrant unit $I$ and with generating cofibrations having cofibrant domain. 
Assume given a strong monad $\left(T,\mu,\eta, \sigma \right)$ on $\mathcal E$ such that
\begin{enumerate}
 \item The category of T-algebras $\mathcal Alg_{T}$ admits a transferred model structure;
 \item The unit $\eta_{X}: X\rightarrow T\left( X\right) $ is a cofibration at each cofibrant object X in $\mathcal E$;
 \item The tensorial strength $$\sigma_{X,Y}: X\otimes TY \xrightarrow{\sim} T\left(X \otimes Y \right)$$ is a weak equivalence for all cofibrant objects X,Y in $\mathcal E$;
 \item The forgetful functor takes free cell attachments in $\mathcal Alg_T$ to homotopical cell attachments in $\mathcal E$ (cf. Definition \ref{defcellatt}).
\end{enumerate}
Then the monad morphism $\lambda: -\otimes T\left(I\right)\rightarrow T $ induces a Quillen equivalence between the category of $T\left( I\right)$-modules and the category of T-algebras:  
$$ Ho\left( Mod_{T\left( I\right) }\right) \simeq Ho\left( \mathcal Alg_{T}\right) $$
\end{Thm}
In the special case where $\mathcal E$ is the category of $\Gamma$-spaces equipped with Bousfield-Friedlander's stable model structure 
\cite{BouHT} and $T$ is the strong monad associated to a well-pointed $\Gamma$-theory, this recovers a theorem proved by Schwede \cite{SchSHAT}. It was one of the main motivations of this thesis to understand Schwede's theorem as an instance of a general homotopical Morita theorem.

Hypothesis $(d)$ is difficult to check in practice since it involves an analysis of certain pushouts in the category of $T$-algebras and is not directly expressed by properties of the monad $T$.
We therefore establish in the last section of Chapter \ref{chModCT} a more accessible form of hypothesis $(d)$ provided that the model category $\mathcal E$ is pointed, has a good realization functor for simplicial objects and 
satisfies a suitable version of ``excision'' (see Section \ref{SRealMModCat} for precise definitions).
If this is the case, hypothesis $(d)$ may be reformulated as follows (cf. Proposition \ref{propForgfuncRE}):
\begin{enumerate}
 \item[\textit{(d')}] \textit{The forgetful functor takes free cell extensions in $\mathcal Alg_T$ to cofibrations in $\mathcal E$};
 \item[\textit{(d'')}] \textit{The monad $T$ takes any cofibration $X\to Y$ between cofibrant objects to a cofibration $T(X)\to T(Y)$ between cofibrant objects and the induced map $T(Y)/T(X)\to T(Y/X)$ is a weak equivalence.}
\end{enumerate}

If $\mathcal E$ is the stable model category of $\Gamma$-spaces almost all hypotheses of our homotopical Morita theorem follow from an important result of Lydakis \cite{LydSG} concerning the homotopical properties of the so-called assembly map.
In order to keep this thesis as self-contained as possible we include a proof of Lydakis' theorem. 
This proof is similar but slightly more conceptual than Lydakis' original proof and applies to $\Gamma$-spaces with values in essentially any cartesian Quillen model category for the homotopy theory of topological spaces.

\chapter{Preliminaries}

This chapter is devoted to the language of category theory. 
In Section \ref{SSymMC} we provide definitions of monoidal categories, symmetric monoidal categories together with some examples. In Section \ref{SEnrCat}, once the notion of closed symmetric monoidal category is fixed, we define an enriched category. Then all the basic theory of categories is translated to the enriched context. In Section \ref{STensCat} we define tensored and cotensored enriched categories. In Section \ref{SMMonAlg} we provide definitions for monoid, monads and their algebras. 

For more detailed informations see \cite{BorHCA, MaclCWM}.

\section{Symmetric monoidal categories}\label{SSymMC}


\begin{definition}\label{defcatmon}
A monoidal category $\left( \mathcal E, \otimes, I\right) $ is a category $\mathcal E$ equipped with:
\begin{enumerate}
\item A bifunctor $\otimes: \mathcal E\times\mathcal E\rightarrow\mathcal E$ called the tensor product;
\item An object $I$ in $\mathcal E$, called the unit;
\item For every triple of objects $\left(X,Y,Z \right) $ in $\mathcal E$, an associativity isomorphism which is natural in $X,Y,Z$ and given by 
\begin{center}
 $a_{XYZ}:\left( X\otimes Y\right)\otimes Z\longrightarrow X\otimes\left( Y\otimes Z\right) $;
\end{center}   
\item For every object $X$ in $\mathcal E$, a left unit isomorphism natural in $X$ and given by
\begin{center}
 $l_{X}: I\otimes X\longrightarrow X $;
\end{center}
\item For every object $X$ in $\mathcal E$, a right unit isomorphism natural in $X$ and given by 
\begin{center}
 $r_{X}: X\otimes I\longrightarrow X $;
\end{center}
such that the following two diagrams commute:
\begin{diagram}{Associativity axiom}
\label{diagass}
$$
\xymatrix@C=2cm @R=1cm{\relax 
\left( \left(X\otimes Y\right) \otimes Z\right) \otimes T  \ar[r]^{a_{X\otimes Y,Z,T}} \ar[d]_{a_{X,Y,Z}\otimes 1} & \left(X\otimes Y\right) \otimes \left( Z\otimes T\right) \ar[dd]^{a_{X,Y,Z\otimes T}}\\
    \left( X\otimes \left(Y\otimes Z\right) \right) \otimes T \ar[d]_{a_{X,Y\otimes Z,T}} \\
    X \otimes \left( \left(Y\otimes Z\right) \otimes T\right) \ar[r]_{1\otimes a_{Y,Z,T}} & X\otimes\left( Y\otimes\left( Z\otimes T\right)\right) 
}
$$
\end{diagram}
\begin{diagram}{Unit axiom}
$$
\xymatrix{\relax 
\left( X\otimes I\right) \otimes Y\ar[rr]^{a_{XIY}} \ar[rd]_{r_{X}\otimes1} && X\otimes\left( I\otimes Y\right) \ar[ld]^{1\otimes l_{Y}} \\ & X\otimes Y
}
$$
\end{diagram}
\end{enumerate}
\end{definition}

\begin{definition}\label{defcatmonsym}
A symmetric monoidal category $\left( \mathcal E, \otimes, I, s\right) $  is a monoidal category $\left( \mathcal E, \otimes, I\right) $  equipped with, 
for every couple $\left(X,Y \right) $ of objects in $\mathcal E$, a symmetry isomorphism, natural in $X,Y$ and given by 
\begin{center}
 $s_{XY}:X\otimes Y\longrightarrow Y\otimes X $;
\end{center}   
such that the following diagrams commute
\begin{diagram}{}
$$
\xymatrix@C=3cm @R=1.5cm {\relax 
\left(X\otimes Y\right) \otimes Z \ar[r]^{s_{XY}\otimes 1} \ar[d]_{a_{XYZ}} & \left(Y\otimes X\right) \otimes Z \ar[d]^{a_{YXZ}}\\
    X\otimes \left(Y\otimes Z\right) \ar[d]_{s_{X,Y\otimes Z}} & Y\otimes \left( X\otimes Z\right) \ar[d]^{1 \otimes s_{YZ}} \\
    \left(Y\otimes Z\right) \otimes X \ar[r]_{a_{YZX}} & Y\otimes \left( Z\otimes X\right) 
}
$$
\end{diagram}
\begin{diagram}{}
$$
\xymatrix{\relax 
 X\otimes I \ar[rr]^{s_{XI}} \ar[rd]_{r_X} && I\otimes X \ar[ld]^{l_X} \\ 
& X
}
$$
\end{diagram}
\begin{diagram}{}
$$
\xymatrix@C=0.3cm{\relax 
 X\otimes Y \ar@{=}[rr] \ar[rd]_{s_{XY}} && X\otimes Y \\ 
& Y\otimes X \ar[ru]_{s_{YX}}
}
$$
\end{diagram}
\end{definition}

\begin{ex}\label{exmoncat}
We give some basic examples of symmetric monoidal categories:
\begin{itemize}
 \item $\left(\mathbf{Set},\times, 1 \right) $, category of sets with the cartesian product;
 \item $\left(\mathbf{Top},\times, 1 \right) $, category of topological spaces with the product;
 \item $\left(\mathbf{CGTop},\times, 1 \right) $, category of compactly generated topological spaces with the product;
 \item $\left(\mathbf{Cat},\times, 1 \right) $, category of categories with the cartesian product;
 \item $\left(\mathbf{Ab},\otimes, \mathbb Z \right) $, category of abelian groups with the tensor product;
 \item $\left(\mathbf{Mod_{R}},\otimes, R \right) $, category of R-modules, where R is a commutative ring, with the tensor product;
 \item $\left(\mathbf{G-Mod_{R}},\otimes, R \right) $, category of graded R-modules, where R is a commutative ring, with its usual tensor product;
\item $\left(\mathbf{DG-Mod_{R}},\otimes, R \right) $, category of differential graded R-modules, where R is a commutative ring, with its usual tensor product.
\end{itemize}
A non-symmetric example is the category of R-bimodules over a non commutative ring R, with the tensor product $\otimes_{R}$.
\end{ex}

\begin{definition}\label{defcatmonsymcl}
A symmetric monoidal category $\mathcal E$ is closed if for each object $X$ in $\mathcal E$, the functor $-\otimes Y: \mathcal E\to \mathcal E$ has a right adjoint $ \underline{\mathcal E}\left( Y,-\right) $.

In particular, we have a bijection
$$\mathcal E\left( X\otimes Y,Z\right)\longrightarrow \mathcal E\left( X,\underline{\mathcal E}\left( Y,Z\right)\right) $$
\end{definition}

\begin{ex}
All symmetric monoidal categories given in Example \ref{exmoncat} are closed, except $\left(Top,\times, 1 \right) $. 

In fact, for a topological space Y, the functor $-\otimes Y$ cannot have a right adjoint since it does not preserve regular epimorphisms.
\end{ex}


\begin{definition}\label{defcatmonsymbicl}
A monoidal category $\mathcal E$ is biclosed when, for each object $X$ of $\mathcal E$, both functors
\begin{center}
$-\otimes X: \mathcal E\to \mathcal E$ \hspace{.3in} and  \hspace{.3in} $X\otimes -: \mathcal E\to \mathcal E$
\end{center}
have a right adjoint.
\end{definition}

\begin{ex}
The non-symmetric monoidal category of R-bimodules over a non commutative ring R is biclosed.
\end{ex}

In a symmetric monoidal category, a consequence of the symmetry is that both functors $-\otimes X$ and $X\otimes -$ are naturally isomorphic.
Therefore, one obtains:

\begin{lem}\cite{BorHCA}
A closed symmetric monoidal category $\mathcal E$ is biclosed.
\end{lem}

\begin{rem}
A left unit morphism $l_{X}: I\otimes X\rightarrow X $ is dual to the morphism
 $$j_{X}: I\rightarrow \underline{\mathcal E}\left( X,X\right) $$
In particular, there are isomorphisms $$\mathcal E\left( X,X\right)\cong\mathcal E\left( I\otimes X,X\right)\cong\mathcal E\left( I,\underline{\mathcal E}\left( X,X\right)\right) $$
Similarly, the right unit isomorphism $r_{X}: X\otimes I\rightarrow X $ is dual to the morphism
 $$i_{X}: X\rightarrow \underline{\mathcal E}\left( I,X\right) $$ 
In particular, there are isomorphisms 
 $$\mathcal E\left( X,X\right)\cong\mathcal E\left( X\otimes I,X\right)\cong\mathcal E\left( X,\underline{\mathcal E}\left( I,X\right)\right) $$
\end{rem}

We provide definitions of the evaluation and the composition morphisms, which are closely related. 
\begin{definition}\label{defev}
Let $\mathcal E $ be a closed symmetric monoidal category.

An evaluation morphism, called ev consists in giving, for every pair of objects $\left(X,Y \right) $ in $\mathcal E$, a morphism 
$$ ev_{X}: \underline{\mathcal E}\left( X,Y\right)\otimes X \rightarrow Y$$ 
in $\mathcal E $.

By adjunction, the morphism ev is dual to the identity morphism
$$ Id_{\underline{\mathcal E}\left( X,Y\right)}: \underline{\mathcal E}\left( X,Y\right)\rightarrow \underline{\mathcal E}\left( X,Y\right)$$
\end{definition}
\begin{rem}
We will need in Chapter 3 the morphism
$$ \gamma_{Y}: X\rightarrow \underline{\mathcal E}\left( Y,X\otimes Y\right)$$
which is by adjunction dual to the identity morphism $ Id_{X\otimes Y}: X\otimes Y \rightarrow X\otimes Y$.
\end{rem}




\begin{rem}
We exhibit the universal property of the tensor-cotensor adjunction 
$$-\otimes Y:\mathcal C \leftrightarrows \mathcal D:Hom(Y,-)$$

The component $\eta_{X}: X\rightarrow Hom\left(Y,X\otimes Y\right) $ of the unit has the universal property that every morphism $f:X\rightarrow Hom\left(Y,Z \right)$ is given in one unique way as a morphism $$  X \xrightarrow{\eta_{X}} Hom\left(Y,X\otimes Y\right)\xrightarrow{Hom\left(Y,g\right)} Hom\left(Y,Z \right)$$ for a morphism $g:X\otimes Y\rightarrow Z$.

Hence, one has $f=\hat{g}=Hom\left(Y,g\right)\circ \eta_{X}$. Note that $\eta$ is the dual of the co-evaluation morphism, precisely $\gamma$.

Similarly, the component $\varepsilon_{Z}:Hom\left(Y,Z \right)\otimes Y\rightarrow Z$ of the counit has the universal property that every morphism $g:X\otimes Y\rightarrow Z$ is given in one unique way as a morphism $$  X\otimes Y \xrightarrow{f\otimes Y} Hom\left(Y,Z \right)\otimes Y\xrightarrow{\varepsilon_{Z}} Z$$ for a morphism $f:X\rightarrow Hom\left(Y,Z \right)$.

Hence, one has $g=\hat{f}=\varepsilon_{Y}\circ f\otimes Y$. Note that  $\varepsilon$ is the evaluation morphism, precisely $ev$.
\end{rem}

\begin{definition}\label{defcomp}
Let $\mathcal E $ be a closed symmetric monoidal category.

A composition morphism, called c consists in giving, for every triple of objects $\left(X,Y,Z \right) $ in $\mathcal E$, a morphism
$$ c_{XYZ}: \mathcal E\left( Y,Z\right)\otimes\mathcal E\left( X,Y\right)\longrightarrow \mathcal E\left( X,Z\right)$$
in $\mathcal E$ such that the following diagram commutes 
\begin{diagram}{}
$$
\xymatrix@C=2cm @R=1.8cm{\relax 
    \mathcal E\left( Y,Z\right)\otimes\mathcal E\left( X,Y\right)\otimes X \ar[r]^{1\otimes ev} \ar[d]_{c \otimes 1} & \mathcal E\left( Y,Z\right)\otimes Y \ar[d]^{ev} \\
    \mathcal E\left( X,Z\right)\otimes X \ar[r]^{ev} & Z
  }
$$
\end{diagram}
for all objects X,Y,Z in $\mathcal E$.

The adjoint of the composition morphism is twice the evaluation morphism $$\hat{c}=ev\circ 1\otimes ev$$
\end{definition}
\begin{rem}
 Similarly, a composition morphism is determined by an evaluation morphism, such that the following diagram commutes: 
$$
\xymatrix@C=2.5cm @R=1.8cm{\relax 
   \mathcal E\left( X,Y\right)\otimes\mathcal E\left( I,X\right) \ar[r]^{c_{IXY}}  & \mathcal E\left( I,Y\right)  \\
    \mathcal E\left( X,Y\right)\otimes X \ar[u]^{1 \otimes i_X}_\cong \ar[r]^{ev} & Y \ar[u]_{i_Y}^\cong
  }
$$
This is a consequence of the compatibility of the unit and the associativity morphisms.

In fact, by adjunction we obtain
$$
\shorthandoff{;:!?}
\xymatrix @!0 @C=3.3cm @R=1.5cm{\relax 
 \mathcal E\left( X,Y\right) \ar[rr]^(.4){c\hat{}} \ar[rd]_{\mathcal E\left(1,i_{Y}\right)}  && \mathcal E\left( \mathcal E\left( I,X\right),\mathcal E\left( I,Y\right)\right) \ar[ld]^{\mathcal E\left( i_{X},1\right)} \\ 
& \mathcal E\left( X, \mathcal E\left( I,Y\right)\right) 
}
$$
it is equivalent to 
$$
\shorthandoff{;:!?}
\xymatrix @!0 @C=3.5cm @R=1.5cm{\relax 
 \mathcal E\left( X,Y\right) \ar[rr]^{\mathcal E\left( r_{X},1\right)} \ar[rd]_{\mathcal E\left( 1,i_{X}\right)}&& \mathcal E\left(X\otimes 1,Y\right) \ar[ld]^{\cong}  \\ 
& \mathcal E\left( X, \mathcal E\left( I,Y\right)\right) 
}
$$
which is equivalent to the commutativity of the following diagram
$$
\shorthandoff{;:!?}
\xymatrix @!0 @C=3cm @R=1.5cm{\relax 
 \left( X\otimes Y\right) \otimes I \ar[rr]^{a_{XYI}} \ar[rd]_{r_{X\otimes Y}} && X\otimes\left(Y\otimes I\right) \ar[ld]^{1\otimes r_{Y}} \\ 
& X\otimes Y
}
$$
The commutativity of the last diagram is a consequence of the unit axiom of a monoidal category.
For more details, see \cite{EilCC}.
Associativity of the composition corresponds to the associativity of the monoidal structure.
\end{rem}


\section{Enriched categories}\label{SEnrCat}


In this section, we introduce the notion of enriched categories. Almost all concepts and results of ordinary category theory can be extended to an $\mathcal E$-enriched context (cf. \cite{BorHCA,KelBCECT}).

\begin{definition} \label{defcatenr}
Let $\mathcal E$ be a monoidal category. An enriched category $\mathcal C$ over $\mathcal E$ consists in giving:   
\begin{enumerate}
\item A class $Ob\left(\mathcal C\right) $ of objects;
\item For every pair of objects $\left(X,Y\right) $ in $\mathcal C$, an internal object $\underline{\mathcal C}\left(X,Y\right)$ in $\mathcal E$;
\item For every triple of objects $\left(X,Y,Z \right) $ in $\mathcal C$, a composition morphism in $\mathcal E$
 $$c_{XYZ}:\underline{\mathcal C}\left(Y,Z\right)\otimes \underline{\mathcal C}\left(X,Y\right)\longrightarrow \underline{\mathcal C}\left(X,Z\right);  $$
\item For every object $X$ in $\mathcal C$, a unit morphism in $\mathcal E$
 $$j_{X}: I\longrightarrow \underline{\mathcal C}\left(X,X\right); $$
such that the following coherence diagrams commute:
\begin{diagram}{}
$$
\xymatrix {\relax
\left( \underline{\mathcal C}\left(Z,T\right)\otimes \underline{\mathcal C}\left(Y,Z\right)\right)\otimes \underline{\mathcal C}\left(X,Y\right) \ar[rr]^{a} \ar[d]_{c\otimes 1} && \underline{\mathcal C}\left(Z,T\right)\otimes\left( \underline{\mathcal C}\left(Y,Z\right)\otimes \underline{\mathcal C}\left(X,Y\right)\right) \ar[d]^{1\otimes c}\\
\underline{\mathcal C}\left(Y,T\right)\otimes\underline{\mathcal C}\left(X,Y\right) \ar[rd]_{c} && \underline{\mathcal C}\left(Z,T\right)\otimes\underline{\mathcal C}\left(X,Z\right) \ar[ld]^{c} \\
& \underline{\mathcal C}\left(X,T\right) 
}
$$
\end{diagram}
\begin{diagram}{}
$$
\xymatrix @C=2cm @R=1.5cm{\relax  
I\otimes \underline{\mathcal C}\left(X,Y\right) \ar[r]^l \ar[d]_{j_{Y}\otimes 1} & \underline{\mathcal C}\left(X,Y\right) \ar[d]_{Id} & \underline{\mathcal C}\left(X,Y\right)\otimes I \ar[l]_r \ar[d]^{1\otimes j_{X}}\\
\underline{\mathcal C}\left(Y,Y\right)\otimes \underline{\mathcal C}\left(X,Y\right) \ar[r]^c & \underline{\mathcal C}\left(X,Y\right) & \underline{\mathcal C}\left(X,Y\right)\otimes \underline{\mathcal C}\left(X,X\right) \ar[l]_c 
}
$$
\end{diagram}

An enriched category $\mathcal C$ over $\mathcal E$ is also called a $\mathcal E$-category.
\end{enumerate}
\end{definition}

\begin{ex}
Taking enriched categories over different examples of monoidal categories $\mathcal E$ given in Example \ref{exmoncat}, we recover some familiar categories.  
\begin{itemize}
 \item $\mathbf{Set}-$category is an ordinary (locally small) category;  
 \item $\mathbf{Ab}-$category is a linear category;  
 \item $\mathbf{Cat}-$category is a 2-category;  
 \item $\mathbf{DG-Mod_{R}}$-category is a differential graded category.
\end{itemize}
\end{ex}

We further generalize other basic concepts and results of ordinary category theory to the enriched context.

\begin{definition} \label{deffonenr}
Let $\mathcal A$ and $\mathcal B$ be two categories enriched over a monoidal category $\mathcal E$.

A $\mathcal E$-functor $\left(F,\varphi_{F}\right)$ consists in giving:   
\begin{enumerate}
\item A functor $F:\mathcal A\longrightarrow\mathcal B$;
\item For every object $X$ in $\mathcal A$, an object $FX$ in $\mathcal B$;
\item For every pair of objects $\left(X,Y\right) $ in $\mathcal A$, a morphism in $\mathcal E$
 $$\varphi_{F}:\underline{\mathcal A}\left(X,Y\right)\longrightarrow \underline{\mathcal B}\left(FX,FY\right) $$ called the enrichment morphism
such that the following diagrams commute:
\begin{diagram}{Unit axiom}
$$
\xymatrix @C=1.5cm @R=1.2cm{\relax 
\underline{\mathcal A}\left(X,X\right) \ar[rr]^{\varphi_{F}}  && \underline{\mathcal B}\left(FX,FX\right)  \\ 
& I \ar[lu]^{j_X} \ar[ru]_{j_{FX}}
}
$$
\end{diagram}

\begin{diagram}{Composition axiom}
$$
\xymatrix @C=2.7cm @R=1.5cm {\relax
\underline{\mathcal A}\left(Y,Z\right)\otimes \underline{\mathcal A}\left(X,Y\right) \ar[r]^{c_{XYZ}} \ar[d]_{\varphi_{F}\otimes \varphi_{F}} & \underline{\mathcal A}\left(X,Z\right) \ar[d]^{\varphi_{F}} \\
\underline{\mathcal B}\left(FY,FZ\right)\otimes \underline{\mathcal B}\left(FX,FY\right) \ar[r]^{c_{FXFYFZ}} & \underline{\mathcal B}\left(FX,FZ\right)
}
$$
\end{diagram}

\end{enumerate}
\end{definition}

\begin{definition} \label{deftnatenr}
Let $\mathcal A$ and $\mathcal B$ be two categories enriched over a monoidal category $\mathcal E$ and $F,G:\mathcal A\rightarrow\mathcal B$ two $\mathcal E$-functors. 

A $\mathcal E$-natural transformation $\alpha: F\longrightarrow G$ consists in giving, for every object $X$ in $\mathcal A$, a morphism
 $$\alpha_{X}:I\longrightarrow \underline{\mathcal B}\left(FX,GX\right)  $$     
in $\mathcal E$ such that the following diagram commutes 
\begin{diagram}{}
$$
\shorthandoff{;:!?}
\xymatrix @!0 @C=4.7cm @R=2.3cm{\relax
I\otimes \underline{\mathcal A}\left(X,Y\right) \ar[rr]^{\alpha_{Y}\otimes \varphi_{F}} &&  \underline{\mathcal B}\left(FY,GY\right)\otimes \underline{\mathcal B}\left(FX,FY\right) \ar[d]^{c}\\
\underline{\mathcal A}\left(X,Y\right) \ar[u]^{l^{-1}} \ar[d]_{r^{-1}} && \underline{\mathcal B}\left(FX,GY\right)\\
\underline{\mathcal A}\left(X,Y\right)\otimes I \ar[rr]^{\varphi_{G}\otimes \alpha_{X}} && \underline{\mathcal B}\left(GX,GY\right)\otimes \underline{\mathcal B}\left(FX,GX\right) \ar[u]_{c}
}
$$
\end{diagram}
for all objects X,Y in $\mathcal E$. 
\end{definition}

Kelly \cite{KelBCECT} observed that one can write the ordinary $\mathcal E$-naturality condition (Definition \ref{deftnatenr}) in the more compact form. We have the following definition: 
\begin{definition}\label{deftnatcomp}
Let $\mathcal E$ be a closed symmetric monoidal category. Consider two $\mathcal E$-categories $\mathcal C$ and $\mathcal D$ and two $\mathcal E$-functors $F,G:\mathcal C\rightarrow\mathcal D$.

A $\mathcal E$-natural transformation $\alpha:F \longrightarrow G$ consists in giving a family of morphisms $\alpha_{A}:FA \longrightarrow GA$ in $\mathcal D$, indexed by the objects in $\mathcal C$ and such that the following diagram
\begin{diagram}{}
$$
\shorthandoff{;:!?}
\xymatrix @C=3.5cm @R=2cm{\relax
\underline{\mathcal C}\left(A,B\right) \ar[r]^{\varphi_{F}} \ar[d]_{\varphi_{G}} & \underline{\mathcal D}\left(FA,FB\right) \ar[d]^{\underline{\mathcal D}\left(1,\alpha_{B}\right)} \\
\underline{\mathcal D}\left(GA,GB\right) \ar[r]^{\underline{\mathcal D}\left(\alpha_{A},1\right)} & \underline{\mathcal D}\left(FA,GB\right) 
}
$$
\end{diagram}
commutes in $\mathcal E$. 
\end{definition}
The composition of two enriched natural transformations is an enriched natural transformation using Definition \ref{deftnatcomp}.

\begin{lem}\cite{BorHCA}
If $\mathcal E$ is a closed symmetric monoidal category, then the category $\mathcal E$ is itself a $\mathcal E$-category. 
\end{lem}
\begin{rem}~\\
A $\mathcal E$-category $\mathcal C$ admits itself an underlying category $\mathcal C_{0}$ in the ordinary sense such that
\begin{enumerate}
\item $Ob\left(\mathcal C_{0}\right)=Ob\left(\mathcal C\right)$
\item $Mor_{\mathcal C_{0}}\left( X,Y\right)=\mathcal E\left( I,\underline{\mathcal C}\left( X,Y\right)\right)$
\end{enumerate}
\end{rem}

\begin{rem}~\\
\begin{enumerate}
 \item Every enriched natural transformation induces a natural transformation between underlying functors (Definition \ref{deftnatcomp} with the functor $\left(- \right)_{o} $ everwhere). 
 \item Notation: we say that the natural transformation $\left( \alpha_{X}\right)_{0}:F_{0}X \Rightarrow G_{0}X$ extends to an enriched natural transformation $\alpha_{X}:FX \Rightarrow GX$  
\end{enumerate}
\end{rem}

We generalize the case of adjoint functors to the enriched context.

\begin{definition}
Let $\mathcal E$ be a closed symmetric monoidal category, $\mathcal A$ and $\mathcal B$ two $\mathcal E$-categories.

A couple of $\mathcal E$-functors $F:\mathcal A\rightarrow \mathcal B$, $G:\mathcal B\rightarrow \mathcal A$ defines a $\mathcal E$-adjunction, with F left adjoint to G and G right adjoint to F, when for every pair of objects $\left( A,B\right) \in Ob\left(\mathcal A\ \right)\times Ob\left(\mathcal B\ \right) $ there are isomorphisms in $\mathcal E$:  
\begin{center}
$\mathcal B\left(F\left( A\right),B \right)\cong \mathcal A\left(A,G\left( B\right)\right)$; 
\end{center}
which are $\mathcal E$-natural in A and B.
\end{definition}

\begin{prop}\cite{BorHCA}
Let $\mathcal E$ be a closed symmetric monoidal category. Let $\mathcal A$ and $\mathcal B$ be two $\mathcal E$-categories and $G:\mathcal B\rightarrow \mathcal A$ a $\mathcal E$-functor.
The following are equivalent:

\begin{enumerate}
 \item Functor G has a left $\mathcal E$-adjoint $F:\mathcal A\rightarrow \mathcal B$;
 \item For every object A in $\mathcal A$, there is an object $F\left( A\right)$ in $\mathcal B$ with isomorphisms
\begin{center}
$\mathcal B\left(F\left( A\right),B \right)\cong \mathcal A\left(A,G\left( B\right)\right)$ 
\end{center}
which are $\mathcal E$-natural in B $\in \mathcal B$.

\end{enumerate}
\end{prop}

\section{Tensored and cotensored enriched categories}\label{STensCat}

In this section, we provide definitions of tensored and cotensored $\mathcal E$-categories.
 
\begin{definition}\label{defcatenrtens}
Let $\mathcal E$ be a closed symmetric monoidal category and $\mathcal C$ a $\mathcal E$-category.

The category $\mathcal C$ is called $\mathcal E$-tensored if, for every object A in $\mathcal C$, the functor $$\underline{\mathcal C}\left(A,-\right):\mathcal C \rightarrow \mathcal E$$ admits a left adjoint $$-\otimes A:\mathcal E \rightarrow \mathcal C$$ such that:
\begin{enumerate}
\item For every object A in $\mathcal C$, there is an isomorphism
$$ I\otimes A\cong A, $$ natural in A; 
\item For every pair of objects $\left(X,Y \right)$ in $\mathcal E$ and every object A in $\mathcal C$, there is an isomorphism
$$ \left(X\otimes_{\mathcal E} Y \right)\otimes A\cong X\otimes\left(Y\otimes A \right) $$
which is natural in X, Y, A.
\end{enumerate}
\end{definition}

\begin{definition}\label{defcatenrcot}
Let $\mathcal E$ be a closed symmetric monoidal category and $\mathcal C$ a $\mathcal E$-category.

The category $\mathcal C$ is called $\mathcal E$-cotensored if the functor $$\underline{\mathcal C}\left(A,-\right):\mathcal C \rightarrow \mathcal E$$ admits a right adjoint:$$A^{\left( -\right) }:\mathcal E \rightarrow \mathcal C$$ such that:
\begin{enumerate}
\item For every object A in $\mathcal C$, there is an isomorphism
$$ A^{I}\cong A,$$ natural in A;  
\item For every pair of objects $\left(X,Y \right)$ in $\mathcal E$ and every object A in $\mathcal C$, there is an isomorphism:
$$ A^{X\otimes_{\mathcal E}Y}\cong \left(A^{X} \right)^{Y} $$ natural in X, Y, A.

\end{enumerate}
\end{definition}

\begin{rem}
Putting $\mathcal C=\mathcal E$, it follows from Definition \ref{defcatenrtens} and \ref{defcatenrcot} that $\mathcal E$ is tensored and cotensored over itself.
\end{rem}



\begin{prop}
Let $\mathcal E$ be a closed symmetric monoidal category and $\mathcal C$ a $\mathcal E$-category. Then
\begin{enumerate}
\item $\mathcal C$ is tensored if and only if every $\mathcal E$-functor $\underline{\mathcal C}\left(X,-\right):\mathcal C\rightarrow \mathcal E $, for an object X in $\mathcal C$, has a left $\mathcal E$-adjoint $-\otimes X:\mathcal E\rightarrow \mathcal C$;
\item $\mathcal C$ is cotensored if and only if every $\mathcal E$-functor $\underline{\mathcal C}\left(X,-\right):\mathcal C\rightarrow \mathcal E $, for an object X in $\mathcal C$, has a right $\mathcal E$-adjoint $X^{-}:\mathcal E\rightarrow \mathcal C$.
\end{enumerate}
\end{prop}

\begin{prop}\label{propprescot}\cite{BorHCA}
Let $\mathcal E$ be a closed symmetric monoidal category. Let $\mathcal A$ and $\mathcal B$ be cotensored $\mathcal E$-categories and $G:\mathcal B\rightarrow \mathcal A$ a $\mathcal E$-functor. 

Then G has a left $\mathcal E$-adjoint functor if and only if
\begin{enumerate}
 \item The functor G preserves cotensors;
 \item The underlying functor $G_{0}: \mathcal B_{0} \rightarrow \mathcal A_{0} $ has a left adjoint.
\end{enumerate}
\end{prop}

\section{Monoids, monads and their algebras}\label{SMMonAlg}

Monads are important in the theory of adjoint functors and they generalize closure operators on partially ordered sets to arbitrary categories.
The notion of algebras over a monad generalizes classical notions from universal algebra, and in this sense, monads can be thought of as "theories".

In this section, we provide definitions of monoids, monads and their algebras and indicate some elementary properties, with special emphasis on the case where the base category is regular and the monad preserves reflexive coequalizers.

\begin{definition}
A monoid $\left(M,m,n\right)$ in a monoidal category $\mathcal C$ consists in giving:
\begin{enumerate}
\item An object M in $\mathcal C$;
\item Unit and multiplication morphisms $n:I \longrightarrow M$ and $m:M\otimes M \longrightarrow M$
such that the following diagrams commute: 
\end{enumerate}
\begin{diagram}{}
$$
\shorthandoff{;:!?}
\xymatrix @!0 @C=3cm @R=2.7cm{\relax
I\otimes M\ar[r]^{n\otimes M} \ar[rd]_{l} & M\otimes M \ar[d]^{m} & M\otimes I \ar[l]_{M\otimes n} \ar[ld]^{r} \\
& M
}
$$
\end{diagram} 
\begin{diagram}{}
$$
\shorthandoff{;:!?}
\xymatrix @!0 @C=4.7cm @R=2.7cm{\relax
\left( M\otimes M\right) \otimes M \ar[r]^{a} \ar[d]_{m\otimes M} & M\otimes \left( M\otimes M\right) \ar[r]^{M\otimes m} & M\otimes M \ar[d]^{m} \\
M\otimes M \ar[rr]^{m} && M
}
$$
\end{diagram}
\end{definition}

Some basic examples of monoids are:
\begin{ex}~\\
\begin{itemize}
 \item A monoid in $\left(\mathbf{Set}, \times, 1 \right) $ is just a monoid in the ordinary sense; 
 \item A monoid in $\left(\mathbf{Top}, \times, * \right) $ is a topological monoid; 
 \item A monoid in $\left(\mathbf{Ab},\otimes, \mathbb Z \right)$ is a ring;
 \item A monoid in $\left(\mathbf{Mod_{R}},\otimes, R \right) $ is a R-algebra.
\end{itemize}
\end{ex}

\begin{definition}
Let $\left(M,m,n\right)$ and $\left(M',m',n'\right)$ be two monoids in a monoidal category $\mathcal C$. 

A morphism of monoids $f:M\rightarrow M'$ is such that the following diagrams commute:
\begin{diagram}{}
$$
\shorthandoff{;:!?}
\xymatrix @!0 @C=2cm @R=2.7cm{\relax
M \ar[rr]^{f}  && M' \\
& I \ar[ru]_{n'} \ar[lu]^{n}
}
$$
\end{diagram} 
\begin{diagram}{}
$$
\shorthandoff{;:!?}
\xymatrix @!0 @C=4cm @R=2.7cm{\relax
M \otimes M \ar[r]^{f\otimes f} \ar[d]_{m} & M'\otimes M' \ar[d]^{m'} \\
M \ar[r]^{f} & M'
}
$$
\end{diagram}
\end{definition}
The monoids and the morphisms of monoids in a monoidal category $\mathcal C$ constitute a category, written $Monoids\left(\mathcal C\right) $.

\begin{definition}
A monad $\left(T,\mu,\eta\right)$ in a category $\mathcal C$ consists in giving:
\begin{enumerate}
\item A functor $T:\mathcal C\rightarrow \mathcal C$;
\item Natural transformations $\eta:Id_{\mathcal C}\longrightarrow T$ and $\mu:TT\longrightarrow T$ called the unit and the multiplication of the monad, 
such that the following diagrams commute: 
\end{enumerate}
\begin{diagram}{}
$$
\shorthandoff{;:!?}
\xymatrix @!0 @C=2.5cm @R=2.3cm{\relax
T\ar[r]^{\eta T} \ar@{=}[rd] & TT \ar[d]^{\mu} & T \ar[l]_{T\eta} \ar@{=}[ld] \\
& T
}
$$
\end{diagram} 
\begin{diagram}{}
$$
\shorthandoff{;:!?}
\xymatrix @!0 @C=3cm @R=2.3cm{\relax
TTT \ar[r]^{\mu T} \ar[d]_{T\mu} & TT \ar[d]^{\mu} \\
TT \ar[r]^{\mu} & T
}
$$
\end{diagram}
\end{definition}

\begin{rem}
A monad $\left(T,\mu,\eta\right)$ in a category $\mathcal C$ is a monoid in the category of endofunctors of $\mathcal C$, where the monoidal structure is given by composition of endofunctors.   
\end{rem}

\begin{definition}
Let $\left(T,\mu,\eta \right) $ and $\left(S,\xi,\zeta \right) $ be two monads in a category $\mathcal C$. 

A morphism of monads $\lambda:S\rightarrow T$ consists in giving a natural transformation $\lambda:S\rightarrow T$ such that the following diagrams commute:
\begin{diagram}{} 
$$
\shorthandoff{;:!?}
\xymatrix  @!0 @C=2cm @R=2cm{\relax 
 S \ar[rr]^{\lambda}  && T \\ 
& I \ar[ru]_{\eta} \ar[lu]^{\zeta} 
}
$$
\end{diagram} 

\begin{diagram}{}
$$
\shorthandoff{;:!?}
\xymatrix @!0 @C=3.5cm @R=2.5cm{\relax
SS \ar[r]^{\lambda\circ \lambda} \ar[d]_{\xi} & TT \ar[d]^{\mu} \\
S \ar[r]^{\lambda} & T
}
$$
\end{diagram} 
\end{definition}
The monads and the morphisms of monads in a monoidal category $\mathcal C$ constitute a category, written $Monads\left(\mathcal C\right) $.

We can translate the notion of a monad to the enriched context.
\begin{definition}

A $\mathcal E$-monad $\left(T,\mu,\eta,\varphi\right)$ in a $\mathcal E$-category $\mathcal C$ consists in giving:
\begin{enumerate}
\item A $\mathcal E$-functor $\left( T, \varphi \right)$, where $T:\mathcal C\rightarrow\mathcal C$ and $\varphi_{T}:\underline{\mathcal E}(A,B)\rightarrow\underline{\mathcal E}(TA,TB) $ denotes the enrichment morphism;
\item $\mathcal E$-natural transformations $\eta:Id_{\mathcal C}\longrightarrow T$ and $\mu:TT\longrightarrow T$,
such that the following diagrams commute:
\end{enumerate}
\begin{diagram}{}
\label{diageta}
$$
\shorthandoff{;:!?}
\xymatrix @!0 @C=2.5cm @R=2.3cm{\relax
T\ar[r]^{\eta T} \ar@{=}[rd] & TT \ar[d]^{\mu} & T \ar[l]_{T\eta} \ar@{=}[ld] \\
& T
}
$$
\end{diagram} 
 
\begin{diagram}{}
\label{diagmu}
$$
\shorthandoff{;:!?}
\xymatrix @!0 @C=3cm @R=2.3cm{\relax
TTT \ar[r]^{\mu T} \ar[d]_{T\mu} & TT \ar[d]^{\mu} \\
TT \ar[r]^{\mu} & T
}
$$
\end{diagram}
\end{definition}

Every monoid admits an induced monad. We have the following lemma:
\begin{lem}\label{lememon}
Let $\mathcal E$ be a symmetric monoidal category and suppose that $\left(M,m,n \right) $ is a monoid in $\mathcal E$.

Then we can construct a $\mathcal E$-monad $\left( -\otimes M,\eta,\mu\right) $ with $\eta$ and $\mu$ given by:
\begin{diagram}{}
$$
\shorthandoff{;:!?}
\xymatrix  @!0 @C=2cm @R=2cm{\relax 
X  \ar[rr]^{\eta_{X}} \ar[rd]_{r^{-1}} && X\otimes M \\ 
& X\otimes I \ar[ru]_{X\otimes n} 
}
$$
\end{diagram}
\begin{diagram}{}
$$
\shorthandoff{;:!?}
\xymatrix  @!0 @C=2cm @R=2cm{\relax 
\left( X\otimes M\right) \otimes M  \ar[rr]^{\mu_{X}} \ar[rd]_{a} && X\otimes M \\ 
& X\otimes \left( M\otimes M\right)  \ar[ru]_{X\otimes m} 
}
$$
\end{diagram}
\end{lem}

\begin{definition}
Let $\left(T,\mu,\eta\right)$ be a monad on a category $\mathcal C$.

An algebra on a monad $\left(T,\mu,\eta\right)$ is a pair $\left(C,\xi_{C}\right)$ consisting of an object C of $\mathcal C$ together with a morphism $\xi_{C}:TC\rightarrow C$ such that the following diagrams commute:
\begin{diagram}{}
$$
\shorthandoff{;:!?}
\xymatrix @!0 @C=3cm @R=2.3cm{\relax
TTC \ar[r]^{\mu_{C}} \ar[d]_{T\left(\xi_{C}\right)} & TC \ar[d]^{\xi_{C}} \\
TC \ar[r]^{\xi_{C}} & C
}
$$
\end{diagram}
\begin{diagram}{}
$$
\shorthandoff{;:!?}
\xymatrix @!0 @C=2.5cm @R=2.3cm{\relax
C\ar[r]^{\eta_{C}} \ar@{=}[rd] & TC \ar[d]^{\xi_{C}} \\
& C
}
$$
\end{diagram} 
An algebra on a monad $\left(T,\mu,\eta\right)$ is also called a T-algebra.
\end{definition}

\begin{definition}
Let $\left(T,\mu,\eta\right)$ be a monad on a category $\mathcal C$.
Given two T-algebras $\left(C,\xi_{C}\right)$ and $\left(D,\xi_{D}\right)$ on $\mathcal C$, a morphism $$f:\left(C,\xi_{C}\right)\rightarrow \left(D,\xi_{D}\right)$$ of T-algebras is a morphism $f:C\rightarrow D$ in $\mathcal C$ such that the following diagram commutes 
\begin{diagram}{}
$$
\shorthandoff{;:!?}
\xymatrix @!0 @C=3cm @R=2.3cm{\relax
TC \ar[r]^{T\left(f \right)} \ar[d]_{\xi_{C}} & TD \ar[d]^{\xi_{D}} \\
C \ar[r]^{f} & D
}
$$
\end{diagram}
\end{definition}

$T$-algebras and morhisms of $T$-algebras constitute a category $\mathcal Alg_{T}$ (i.e. $\mathcal C^{T}$), also called the Eilenberg-Moore category of the monad.  

The following proposition characterizes the forgetful functor $U_T$ from the category of $T$-algebras to the underlying category.

\begin{prop}\cite{BorHCA}
Let $\left(T,\mu,\eta\right)$ be a monad on a category $\mathcal C$.
Consider the forgetful functor
\begin{eqnarray*}
U_{T}:\mathcal Alg_{T} & \longrightarrow & \mathcal C \\ 
 \left(C,\xi_{C}\right)& \longrightarrow & C \\
 \left(\left(C,\xi_{C}\right)\xrightarrow{f} \left(D,\xi_{D}\right)\right) & \longrightarrow & \left(C \xrightarrow{f} D \right) 
\end{eqnarray*}
Then:
\begin{enumerate}
 \item $U_{T}$ is faithful;
 \item $U_{T}$ reflects isomorphisms;
 \item $U_{T}$ has a left adjoint $F_{T}$ given by:
\end{enumerate}
\begin{eqnarray*}
F_{T}: \mathcal C & \longrightarrow & \mathcal Alg_{T} \\ 
 C & \longrightarrow & \left( TC, \mu_{C}\right)\\ 
\left(C \xrightarrow{f} C' \right) & \longrightarrow & \left( \left( TC, \mu_{C}\right)\xrightarrow{T(f)} \left( TC',\mu_{C'}\right) \right)
\end{eqnarray*}
\end{prop}
Moreover, the unit of the adjunction  $\eta:I_{\mathcal C}\rightarrow U_{T}F_{T}=T$ and the counit $\varepsilon:F_{T}U_{T}\rightarrow I_{\mathcal Alg_{T}}$ is given by $\varepsilon_{\left(C,\xi_{C} \right) }=\xi_{C} $.
 
\begin{lem}\label{lemabscoeq}(\cite{BorHCA}) 
Let $\left(T,\mu,\eta\right)$ be a monad on a category $\mathcal C$. 

For every T-algebra X, the following diagram is a coequalizer in $\mathcal Alg_{T}$:
\begin{diagram}{}\label{diagcoeqtalg}
$$
\shorthandoff{;:!?}
\xymatrix @!0 @C=3cm {\relax 
TTX \ar@<2pt>[r]^{\mu_{X}} \ar@<-2pt>[r]_{T\left(\xi_{X} \right) } & TX \ar[r]^{\xi_{X}} & X 
}
$$
\end{diagram}
Moreover, the forgetful functor $U_{T}:\mathcal Alg_{T}\rightarrow \mathcal C $ takes this coequalizer to a split coequalizer in $\mathcal C$
\begin{diagram}{}
$$
\shorthandoff{;:!?}
\xymatrix @!0 @C=3cm {\relax 
TTX \ar@<2pt>[r]^{\mu_{X}} \ar@<-2pt>[r]_{T\left(\xi_{X} \right) } & TX \ar@/_2pc/@{.>}[l]_{\eta_{TX}} \ar[r]^{\xi_{X}} & X \ar@/_2pc/@{.>}[l]_{\eta_{X}}
}
$$
\end{diagram}
\end{lem} 

\begin{definition}
A coequalizer diagram
\begin{diagram}{}
$$
\shorthandoff{;:!?}
\xymatrix @!0 @C=3cm {\relax 
A \ar@<2pt>[r]^{f} \ar@<-2pt>[r]_{g} & B \ar@/_2pc/@{.>}[l]_{h} \ar[r]^{e} & C  
}
$$
\end{diagram}
in a category $\mathcal C$ is said to be reflexive if there is a map $h:B\rightarrow A$ such that $g\circ h= Id_{B}$ and $f\circ h=Id_{B}$.
\end{definition}
\begin{prop}\label{propalgcocom}\cite{LCCA}
Let $\left(T,\mu,\eta\right)$ be a monad on a cocomplete category $\mathcal C$.
Then the following are equivalent:
\begin{enumerate}
 \item The category of T-algebras has reflexive coequalizers; 
 \item The category of T-algebras is cocomplete.
\end{enumerate}
\end{prop}


\begin{prop}\label{propleftadj}\cite{LCCA}
Let $\left(S,\tilde{\mu},\tilde{\eta}\right)$ and $\left(T,\mu,\eta\right)$ be monads on a category $\mathcal C$.
Suppose that the category of $T$-algebras $\mathcal Alg_{T}$ has reflexive coequalizers.

Given a monad morphism $\varphi:S\rightarrow T$, the induced functor $G:\mathcal Alg_{T}\rightarrow \mathcal Alg_{S}$ has a left adjoint 
$$F:\mathcal Alg_{S}\rightarrow \mathcal Alg_{T}$$
\end{prop}

\begin{proof}For any $S$-algebra $(X,\xi_X:SX\to X)$, the $T$-algebra $F(X,\xi_X)=(FX,\xi_{FX})$ is given by the following reflexive coequalizer in $\mathcal Alg_T$:
$$
\shorthandoff{;:!?}
\xymatrix @!0 @C=2.5cm @R=1.8cm{\relax
TSX \ar[rr]^{T\xi_X} \ar[rd]_{T\varphi_X} && TX \ar[r]^{} & FX\\
& TTX \ar[ru]_{\mu_X}  
}
$$
where the common section is given by $T\tilde{\eta}_X:TX\to TSX$. \end{proof} 
\chapter{Model category theory}\label{chModCT}

In this chapter we recall the basic theory of model categories. 
In Section \ref{SModCat}, we give the basic definitions and examples of model categories. The following Section \ref{SCgMc} is devoted to a standard method of constructing a model category, called Quillen's small object argument \cite{QuiHA}. 
It leads to the theory of cofibrantly generated model categories.
Quillen functors, their derived functors and homotopy category are studied in Section \ref{SQfHc}.
In Section \ref{SMmc} we review the basic notions and results on monoidal model categories. 
In the last section, we provide notions of realisation functor and excision in pointed model category in order to reformulate hypothesis $(d)$ of the main theorem, Theorem \ref{thMain}.  
For more detailed informations on model category theory see \cite{QuiHA, HovMC,GoeMCSM, DwSpHtMc}. 

\section{Model categories}\label{SModCat} 
Quillen was the first to introduce model categories in \cite{QuiHA} and, with slightly modified axioms, in \cite{QRatHTh}.
The terminology has changed over the years, especially after publication of the influential books of Hovey \cite{HovMC} and Hirschhorn \cite{DwyHLFMCHC}.


In this section we give some preliminary definitions and the definition of a model category with some basic examples.
\begin{definition}
A morphism $f:X\rightarrow X'$ is a retract of $g:Y\rightarrow Y'$ if there is a commutative diagram
\begin{diagram}{}
$$
\shorthandoff{;:!?}
\xymatrix @!0 @C=2.7cm @R=2.4cm{\relax
X \ar[r]^{i} \ar[d]_{f} & Y \ar[r]^{r} \ar[d]_{g} & X \ar[d]^{f} \\
X' \ar[r]^{i'} & Y' \ar[r]^{r'} & X'
}
$$
\end{diagram} 
in which we have $ri=Id_{X}$ and $r'i'=Id_{X'}$. 
\end{definition}

\begin{definition}
A morphism $i:A\rightarrow B$ has the left lifting property with respect to $p:X\rightarrow Y$  (resp. p has the right lifting property with respect to i) if in any commutative diagram of unbroken arrows
\begin{diagram}{}
$$
\shorthandoff{;:!?}
\xymatrix @!0 @C=2.3cm @R=2cm{\relax
A \ar[r]^{f} \ar[d]_{i} & X \ar[d]^{p} \\
B \ar[r]^{g} \ar@{-->}[ru]^{h} & Y 
}
$$
\end{diagram} 
there is a diagonal filler $h:B \rightarrow X$ such that $hi=f$ and $ph=g$.
\end{definition}

\begin{definition}
A (Quillen) model category consists of a category $\mathcal E$ equipped with three subcategories $cof_\mathcal E$, $we_\mathcal E$, $fib_\mathcal E$, containing all objects of $\mathcal E$ and whose morphisms are called respectively, cofibrations, weak equivalences, fibrations, such that the following five axioms are satisfied:


\begin{enumerate}
 \item[QM1] $\mathcal E$ has finite limits and colimits;
 \item[QM2] (2 out of 3) For composable maps f and g, if two among f, g and fg are in $we_\mathcal E$ then so is the third;
 \item[QM3] (Retracts) Given maps $f$ and $g$ in $\mathcal E$ such that $f$ is a retract of $g$; if $g$ is a fibration, a cofibration or a weak equivalence, then so is $f$.
 \item[QM4] (Lifting) The maps in $cof_\mathcal E \cap we_\mathcal E$ have the left lifting property with respect to the maps in $fib_\mathcal E$; the maps in $cof_\mathcal E$ have the left lifting property with respect to the maps in $we_\mathcal E\cap fib_\mathcal E$;
 \item[QM5] (Factorization) Any map in $\mathcal E$ factors as a map in $cof_\mathcal E\cap we_\mathcal E$ followed by a map in $fib_\mathcal E$, as well as a map in $cof_\mathcal E$ followed by a map in $we_\mathcal E\cap fib_\mathcal E$. 

\end{enumerate} 
The morphisms in $cof_\mathcal E \cap we_\mathcal E$ are called acyclic cofibrations and morphisms in $we_\mathcal E \cap fib_\mathcal E$ are called acyclic fibrations.
\end{definition}
\begin{rem}
Quillen \cite{QuiHA} makes a difference between model categories and closed model categories. 
Nowadays, a Quillen model category is understood to fulfill the axioms \textit{QM1-QM5} of \cite{QRatHTh} which implies closedness in the sense of \cite{QuiHA}. Moreover it is often the case that in \textit{QM1}, existence of all colimits and limits is required and in QM5, 
the factorizations are supposed to be functorial. 
\end{rem}

\begin{ex}
Here are some standard examples of model categories. For more details see \cite{HovMC,DwSpHtMc,GoeMCSM}.
\begin{itemize}
 \item $\mathbf {Mod_{R}}$ category of R-modules, where R is a Frobenius ring;
 \item $\mathbf {Ch(R)}$ category of chain complexes of modules over a ring;
 \item $\mathbf {Top}$ category of topological spaces;
 \item $\mathbf {SSet}$ category of simplicial sets (cf. Section \ref{GspGrGth}).
\end{itemize}
\end{ex}

\section{Cofibrantly generated model categories}\label{SCgMc}

Proving that a particular category has a model structure is always difficult. 
There is, however, a standard method, initiated by Quillen himself \cite{QuiHA} and elaborated by Bousfield, Smith and others. 
This method is based on Quillen's small object argument and leads to the concept of a \emph{cofibrantly generated} model category.
If a model structure is cofibrantly generated, the fibrations (resp. acyclic fibrations) are completely determined by the right lifting property with respect to a \emph{set} of so-called generating acyclic cofibrations (resp. generating cofibrations).
Moreover, the factorizations can be made functorial.
Most of the model categories occuring in literature are cofibrantly generated.
\begin{definition}
Let $\mathcal C$ be a cocomplete category and I a class of maps in $\mathcal C$.
\begin{enumerate}
 \item A map is $I$-injective if it has the right lifting property with respect to the maps in $I$. The class of $I$-injective maps is denoted $I$-inj.
 \item A map is an $I$-cofibration if it has the left lifting property with respect to $I$-injective maps. The class of $I$-cofibrations is denoted $I$-cof.
  \item $I$-cell is the subcategory of $I$-cof containing those morphisms that can be obtained as (possibly tranfinite) composition of pushouts of maps in $I$.
\end{enumerate}
\end{definition}

\begin{rem}
For more details on the concept of transfinite composition, as well as on the concept of relative smallness 
(sometimes also called sequential smallness) cf. \cite{HovMC,DwyHLFMCHC, SchAMMMC}.
\end{rem}


The reason for considering the theory of transfinite compositions and relative $I$-cell complexes is \emph{Quillen's small object argument}.
\begin{theoreme}{Small object argument}(\cite{HovMC,DwyHLFMCHC, SchAMMMC})

Let $\mathcal C$ be a cocomplete category and $I$ a set of maps in $\mathcal C$ whose domains are small relative to $I$-cell. 
Then there is a functorial factorization of any map f in $\mathcal C$ as $f=gh$ where g is in $I$-inj and h is in $I$-cell.  
\end{theoreme}

\begin{definition}
A Quillen model category $\mathcal E$ is cofibrantly generated if $\mathcal E$ is cocomplete and if there exists sets $I$ (resp. $J$) of cofibrations 
(resp. acyclic cofibrations) whose domains are relatively small with respect to $I$-cell (resp. $J$-cell), such that $fib_\mathcal E=J$-inj and $we_\mathcal E\cap fib_\mathcal E=I$-inj.
\end{definition}


\begin{rem}
For a specific choice of $I$ and $J$ as in the definition of a cofibrantly generated model category, the maps in $I$ will be referred to as generating cofibrations and those in $J$ as generating acyclic cofibrations. 
In cofibrantly generated model categories, a map may be functorially factored as an acyclic cofibration followed by a fibration and as a cofibration followed by an acyclic fibration.
Moreover, any cofibration (resp. acyclic cofibration) is a retract of a morphism in $I$-cell (resp. $J$-cell).
\end{rem}



The following theorem gives two different cases where one can lift a model category on $\mathcal E$ to one on $\mathcal Alg_{T}$.
\begin{theoreme}
\cite{SchAMMMC}\label{thcalgT} 

Let $\mathcal E$ be a cofibrantly generated model category with generating set I (resp. J) of cofibrations (resp. acyclic cofibrations). Let T be a monad on $\mathcal E$ such that $\mathcal Alg_T$ is cocomplete. Let $I_T$ (resp. $J_T$) denote the image under the free functor 
$F_T:\mathcal E\to \mathcal Alg_T$ of $I$ (resp. $J$). Assume that the domains of the morphisms in $I_T$ (resp. $J_T$) are relatively small with respect to $I_{T}$-cell (resp. $J_{T}$-cell). Then $\mathcal Alg_{T}$ is a cofibrantly generated model category with generating set $I_T$ (resp. $J_T$) of cofibrations (resp. acyclic cofibrations) provided one of the following two conditions is satisfied:
\begin{enumerate}
 \item $J_T$-cell $\subset we_{\mathcal Alg_T}$ (where weak equivalences in $\mathcal Alg_T$ are those whose underlying map is a weak equivalence in $\mathcal E$);
 \item Every T-algebra admits a fibrant replacement; every fibrant T-algebra admits a path-object (cf.Definition \ref{defobjcof}).
\end{enumerate}
\end{theoreme}

\begin{rem}
The proof of Theorem \ref{thcalgT} is given in \cite{HovMC, DwyHLFMCHC, SchAMMMC}. 
In particular, the condition $(b)$ implies $(a)$ by an idea contained in \cite{QuiHA}. 
One way to obtain the relative smallness of the domains of the morphisms in $I_T$ (resp. $J_T$)  is to ask that $\mathcal E$ is locally finitely presentable, that the domains of the morphisms in $I$ (resp. $J$) are finitely presentable 
(this applies in particular to the example of simplicial sets) and that $T$ preserves filtered colimits (this is sometimes expressed in litterature by saying that $T$ has finite rank).
\end{rem}

We recall here the notions of fibrant and cofibrant replacement in a model category.

\begin{definition}\label{defobjcof}
Let $\varnothing$ and $\ast$ be initial and terminal objects of a model category $\mathcal E$. An object X is cofibrant if the unique map $\varnothing \rightarrow X$ is a cofibration and dually, an object X is fibrant if the unique map $X\rightarrow \ast$ is a fibration. 

More generally, a cofibrant replacement for X consists of a weak equivalence $X_c\rightarrow X$ with $X_c$ cofibrant and dually a fibrant replacement for X consists of a weak equivalence $X\rightarrow X_f$ with $X_f$ fibrant.
Moreover, such replacements always exist by the axiom QM5.
\end{definition}

\begin{lem}(\cite{HovMC}, \cite{DwyHLFMCHC}){Patching lemma of Reedy}\label{lemPatch}

Consider the following commutative cube
\begin{diagram}{}
$$
\shorthandoff{;:!?}
 \xymatrix {
    & Y'  \ar[rr] \ar[dd] && Z' \ar[dd] \\
    X' \ar[ru] \ar[rr] \ar[dd] && T' \ar[ru] \ar[dd] \\
    & Y \ar[rr]  && Z \\
     X \ar[ru] \ar[rr] && T \ar[ru] \\
  }
$$
\end{diagram}
in a model category $\mathcal E$ such that top and bottom squares are pushouts, $X \rightarrow Y$ and $X'\rightarrow Y'$ are cofibrations and the three vertical arrows $X'\rightarrow X$, $Y'\rightarrow Y$, $T'\rightarrow T$ are weak equivalences between cofibrant objects. Then the 
fourth vertical arrow $Z'\rightarrow Z$, is also a weak equivalence between cofibrant objects.
\end{lem}



We shall need in Chapter 4 a slightly more general form of Reedy's patching Lemma where the pushouts are replaced by certain special homotopy pushouts which we call \emph{homotopical cell attachments}.

\begin{definition}\label{defcellatt}
A commutative square 
$$
\shorthandoff{;:!?}
 \xymatrix @!0 @C=2.4cm @R=2.4cm{\relax
    X \ar[r] \ar[d] & T \ar[d] \\
    Y \ar[r]  & Z \\
  }
$$
is called a homotopical cell attachment if $X,Y,T,Z$ are cofibrant,
the vertical maps $X\to Y$ and $T \to Z$ are cofibrations and the comparison map $T\cup_X Y\to Z$ is a weak equivalence.
\end{definition}

\begin{rem}
If the comparison map is an isomorphism, i.e. the square is a pushout, we simply say that it is a \emph{cell attachment}. 
In other words, top and bottom squares in Reedy's patching lemma are supposed to be cell attachements. 
Observe that in a cell attachment the cofibrancy of $T \to Z$ is automatic (i.e. a property) while in a homotopical cell attachment the cofibrancy of $T \to Z$ is a requirement (i.e. part of the structure). 
\end{rem}

\begin{lem}\label{lemReedGen}
The patching lemma of Reedy remains true if top and bottom square are just supposed to be homotopical cell attachments.
\end{lem}

\begin{proof}
This follows from Lemma \ref{lemPatch} and the 2 out of 3 property of weak equivalences and the fact that pushouts are functorial.
\end{proof}

\begin{lem}(\cite{HovMC}, \cite{DwyHLFMCHC}){Telescope lemma of Reedy}\label{lemTeles}

Consider the following commuting diagram of (possibly transfinite) sequences of composable maps
\begin{diagram}{}
$$
\shorthandoff{;:!?}
\xymatrix @!0 @C=2.7cm @R=2.4cm{\relax
A_{0} \ar[r]^{i_{0}} \ar[d]_{f_{0}} & A_{1} \ar[r]^{i_{1}} \ar[d]_{f_{1}} & A_{2} \ar[r]^{i_{2}} \ar[d]_{f_{2}} &.... \\
B_{0} \ar[r]^{j_{0}} & B_{1} \ar[r]^{j_{1}} & B_{2} \ar[r]^{j_{2}} &....
}
$$
\end{diagram} 
where each $f_n$ is a weak equivalence, each $i_n$ and $j_n$ is a cofibration and each $A_i$ and
$B_i$ are cofibrant. Then the colimit of this diagram is a weak equivalence between cofibrant objects.

\end{lem}

\section{Quillen functors and homotopy category}\label{SQfHc}

In this section we recall some basics about Quillen adjunctions, Quillen equivalences and the homotopy category of a model category.

\begin{definition}
Let $\mathcal C$ and $\mathcal D$ be two model categories and $F:\mathcal C \rightleftarrows \mathcal D:G$ an adjoint pair, with $F$ the left adjoint and $G$ the right adjoint. We say that: 
\begin{enumerate}
\item A functor $F:\mathcal C\rightarrow \mathcal D$ is a left Quillen functor if $F$ preserves cofibrations and acyclic cofibrations;
\item A functor $G:\mathcal D\rightarrow \mathcal C$ is a right Quillen functor if $G$ preserves fibrations and acyclic fibrations.
\end{enumerate}
\end{definition}

\begin{definition}
We say that $\left( F,G\right) $ is a Quillen adjunction if $F$ is a left Quillen functor or, equivalently, if $G$ is a right Quillen functor.
\end{definition}

\begin{lem}\cite{HovMC}{Brown's Lemma}\label{lemBrowns}

In any model category the following properties hold:
\begin{enumerate}
 \item Any morphism between cofibrant objects factors as a cofibration followed by a retraction of an acyclic cofibration;
 \item Any morphism between fibrant objects factors as a section of an acyclic fibration followed by a fibration.
\end{enumerate}
\end{lem}

\begin{cor}\cite{HovMC}\label{corQfWE}
\begin{enumerate}
 \item  Any functor between model categories that takes acyclic cofibrations to weak equivalences (e.g. a left Quillen functor) takes weak equivalences between cofibrant objects to weak equivalences;
 \item  Any functor between model categories that takes acyclic fibrations to weak equivalences (e.g. a right Quillen functor) takes weak equivalences between fibrant objects to weak equivalences.
\end{enumerate}
\end{cor}


\begin{definition}
A Quillen adjunction is a Quillen equivalence if for all cofibrant objects $X$ in $\mathcal C$ and all fibrant objects Y in $\mathcal D$, a morphism $X\rightarrow GY$ is a weak equivalence in $\mathcal C$ if and only if the adjoint morphism $FX\rightarrow Y$ is a weak equivalence in $\mathcal D$. 
\end{definition}
\begin{definition}\cite{GoeMCSM, DwSpHtMc}
The homotopy category of $\mathcal E$ is a category $Ho\left( \mathcal E\right)$ with same objects as $\mathcal E$, and with Hom-sets given by
$$Ho(\mathcal E)(X,Y)=\mathcal E(X_c,Y_f)/ \sim $$
where $\sim$ denotes Quillen's left, resp. right homotopy (which coincide here).
\end{definition}

Any Quillen adjunction $F:\mathcal C \rightleftarrows \mathcal D:G$ induces a derived adjunction 
$$LF:Ho\left( \mathcal C\right)\rightleftarrows Ho\left( \mathcal D\right):RG$$ between the homotopy categories (once a cofibrant replacement functor for $\mathcal C$ and a fibrant replacement functor for $\mathcal D$ have been chosen). It can be shown that a Quillen adjunction $(F,G)$ is a Quillen equivalence if and only if the derived adjunction $(LF,RG)$ is an ordinary equivalence of categories.

The following theorem gives an interpretation of the homotopy category $Ho(\mathcal E)$ of a Quillen model category $\mathcal E$ in terms of a universal property.

\begin{theoreme}\cite{GoeMCSM, DwSpHtMc}\label{thHcl}

The homotopy category $Ho(\mathcal E)$ of a Quillen model category $\mathcal E$ is the localization $\gamma:\mathcal E \rightarrow Ho(\mathcal E)$ of $\mathcal E$ with respect to $we_\mathcal E$. Moreover, a morphism of $\mathcal E$ belongs to $we_\mathcal E$ if and only if $\gamma(f)$ is an isomorphism.
\end{theoreme}


\section{Monoidal model categories}\label{SMmc}

We review in this section the basic notions and results on monoidal model categories.
The definition of a monoidal model category involves constraints on the compatibility of the model structure with the closed symmetric monoidal structure.
The compatibility is expressed by pushout and unit axioms given below.
These conditions suffice to ensure that the homotopy category inherits a closed symmetric monoidal structure compatible with the localization functor. 

\begin{definition}\label{defmonmodc}\cite{HovMC,SchAMMMC}
A monoidal model category $\mathcal E$ is a category fullfilling:
\begin{enumerate}
 \item $\mathcal E$ is a closed symmetric monoidal category;
 \item $\mathcal E$ is a model category;
 \item For any pair of cofibrations $f:X\rightarrow Y $ and $g:X'\rightarrow Y'$, the induced map $$f\Box g: \left( X\otimes Y'\right) \bigsqcup_{X\otimes X'} \left( Y\otimes X'\right)\rightarrow Y\otimes Y' $$ given by the pushout diagram:
\begin{diagram}{}\label{diagppa}
$$
\shorthandoff{;:!?}
\xymatrix @C=3cm @R=2cm {\relax
X\otimes X' \ar[r]^{f\otimes X'} \ar[d]_{X\otimes g} & Y\otimes X' \ar[d]^{j_{0}} \ar@/^1pc/[rdd]^{Y\otimes g}\\
X\otimes Y' \ar[r]^{j_{1}} \ar@/_1pc/[rrd]_{f\otimes Y'} & X\otimes Y'\bigsqcup_{X\otimes X'} Y\otimes X' \ar@{.>}[rd]^{f\Box g}\\
&& Y\otimes Y'
}
$$
\end{diagram}
is also a cofibration. If in addition one of the maps f or g is a weak equivalence, then so is the map $f\Box g$;
 \item There exists a cofibrant replacement of the unit $I_{c}\rightarrow I$ such that for any cofibrant object X, the map $I_{c}\otimes X\rightarrow I\otimes X\simeq X$ is a weak equivalence.
\end{enumerate}
\end{definition}

\begin{rem}
The condition $(c)$ is called the pushout-product axiom of Hovey. The condition $(d)$ is automatically satisfied if the unit $I$ is cofibrant.
\end{rem}

Some of the examples of monoidal model categories are:
\begin{ex}
~\\
\begin{enumerate}
 \item[(i)] The category of unbounded chain complexes of R-modules, for a commutative ring R, $Ch\left(R \right) $ is a monoidal model category;
 \item[(ii)] The model category of simplicial sets $\mathbf{SSet}$ forms a monoidal model category;
 \item[(iii)] The model category of pointed simplicial sets $\mathbf{SSet}_{*}$ forms a monoidal model category;
 \item[(iv)] The model categories of k-spaces and compactly generated spaces \textbf{K} and \textbf{T} are monoidal model categories.   
\end{enumerate}
\end{ex}
 The model category of topological spaces $\mathbf{Top}$ is not a monoidal model category since it is not closed, i.e. there are no internal hom's without some extra-conditions.

The following proposition provides a cofibrantly generated model structure on the category of modules over a monoid.

\begin{prop}\label{propmodmod}(\cite{BMDCAO}, Proposition 2.7 (a))
Let $\mathcal E$ be a cofibrantly generated monoidal model category. 

Let $M$ be a well-pointed monoid in $\mathcal E$ i.e. a monoid whose unit: $I \rightarrow M$ is a cofibration in $\mathcal E$.

Then there is a cofibrantly generated model structure on the category of left (right) $M$-modules $Mod_{M}$, where a map is a weak equivalence or a fibration if and only if it is a weak equivalence or a fibration in $\mathcal E$.

\end{prop}

We recall the definition of the monoid axiom which has an important role in lifting the model category structure to monoids and modules.
\begin{definition}\cite{SchAMMMC}\label{defmonax}
A monoidal model category $\mathcal E$ satisfies the monoid axiom if every map in $$\left(\left\lbrace cof_{\mathcal E} \cap we_\mathcal E \right\rbrace \otimes \mathcal E \right)-cell$$ is a weak equivalence. 
\end{definition}

\begin{rem}
Schwede-Shipley \cite{SchAMMMC} showed that the monoid axiom implies the existence of a transferred model structure on the category of monoids in $\mathcal E$, as well as a transferred model structure on the category of modules over a general (non necessarily well-pointed) monoid. 
\end{rem}

\begin{definition}
A $\mathcal E$-model category $\mathcal C$ over a monoidal model category $\mathcal E$ is a category fullfilling:
\begin{enumerate}
 \item The category $\mathcal C$ is enriched, tensored and cotensored over $\mathcal E$; 
 \item $\mathcal C$ is a model category;
 \item For any pair of cofibrations $f:X\rightarrow Y $ in $\mathcal E$ and $g:X'\rightarrow Y'$ in $\mathcal C$, the induced map $$f\Box g: \left( X\otimes Y'\right) \bigsqcup_{X\otimes X'} \left( Y\otimes X'\right)\rightarrow Y\otimes Y' $$ given by the pushout diagram
\begin{diagram}{}\label{diagppa}
$$
\shorthandoff{;:!?}
\xymatrix @C=3cm @R=2cm {\relax
X\otimes X' \ar[r]^{f\otimes X'} \ar[d]_{X\otimes g} & Y\otimes X' \ar[d]^{j_{0}} \ar@/^1pc/[rdd]^{Y\otimes g}\\
X\otimes Y' \ar[r]^{j_{1}} \ar@/_1pc/[rrd]_{f\otimes Y'} & X\otimes Y'\bigsqcup_{X\otimes X'} Y\otimes X' \ar@{.>}[rd]^{f\Box g}\\
&& Y\otimes Y'
}
$$
\end{diagram}
is a cofibration in $\mathcal C$. If in addition one of the maps f or g is a weak equivalence, then so is the map $f\Box g$;
 \item There exists a cofibrant replacement of the unit $I_{c}\rightarrow I$ in $\mathcal E$ such that for any cofibrant object X of $\mathcal C$, the map $I_{c}\otimes X\rightarrow I\otimes X\simeq X$ is a weak equivalence in $\mathcal C$.
\end{enumerate}
\end{definition}

For $\mathcal E=SSet$ these are precisely Quillen's simplicial model categories. 

\section{Realisation and excision in model categories}\label{SRealMModCat}

This final section of Chapter \ref{chModCT} aims to reformulate hypothesis $(d)$ of our main theorem (cf. Intorduction) so as to make it easier to check. 
In order to do so we need quite a bit of additional material on model 
categories. The leading idea is well known in category theory: extend properties of free $T$-algebras to all $T$-algebras by means of their canonical presentation, cf. Lemma \ref{lemabscoeq}. In the model-theoretical 
context, we focus on $T$-algebras freely generated by cofibrant objects of $\mathcal E$ and want to extend their properties to all cofibrant $T$-algebras, provided a transferred model structure on $\mathcal Alg_T$ exists.
This can be done by prolonging the canonical presentation of a $T$-algebra to a simplicial ''resolution``, often called bar resolution.

\begin{definition}
The bar resolution $\mathcal B.(A)$ of a $T$-algebra $A$ is a simplicial object in $\mathcal Alg_T$ which in degree $n$ is defined by the formula $\mathcal B_n(A)=(F_TU_T)^{n+1}(A)$. 

The simplicial face operators are defined by
$$\partial_i=(F_TU_T)^{n-i} \varepsilon_{(F_TU_T)^iA}:\mathcal B_n(A)\to \mathcal B_{n-1}(A) \text{ (for } 
0\leq i\leq n \text{)}$$ and the simplicial degeneracy operators are defined by
$$s_i=(F_TU_T)^{n-1}F_T\eta_{U_T(F_TU_T)^iA}:\mathcal B_n(A)\to\mathcal B_{n+1}(A) \text{ (for } 0\leq i\leq n\text{)}$$
\end{definition}
Here $\varepsilon_A:F_TU_T(A)\to A$ denotes the counit of the adjunction $F_T:\mathcal E \leftrightarrows \mathcal Alg_T:U_T$ which coincides on the underlying object $U_T(A)$ with the map $\xi_A:T(A)\to A$ defining the 
$T$-algebra structure of $A$.

Lemma \ref{lemabscoeq} implies that the underlying simplicial object $U_T \mathcal B.(A)$ in $\mathcal E$ is split over $U_T(A)$, i.e. admits an extra simplicial 
degeneracy in each degree prolonging the split coequalizer of Diagram \ref{diagcoeqtalg} to the left. This implies (cf. e.g. \cite{MayGeoIterLoopSp}) that 
$U_T\mathcal B.(A)$ contains $U_T(A)$ as a simplicial deformation retract. We now assume that the monoidal model category $\mathcal E$ has a \emph{standard 
system of simplices} $C:\Delta\to \mathcal E$ in the sense of Berger-Moerdijk, cf. the appendix of \cite{BerBoardVogtRes}. This allows us to realise the simplicial object $U_T\mathcal B.(A)$ in $\mathcal E$. We denote its 
realization by $\mathcal B(A)=|U_T(\mathcal B.(A))|_C$. It follows from Lemma A.7 in \cite{BerBoardVogtRes} that $\mathcal B(A)$ contains $A$ as a deformation retract with respect to the 
interval given by the $1$-truncation $C_0 \rightrightarrows C_1\rightarrow C_0$ of the cosimplicial object $C$. In particular, the canonical map $\mathcal B(A)\to U_T(A)$ is a weak equivalence in $\mathcal E$.

We now axiomatise Segal's \cite{SECCT} notion of "good simplicial space" as follows.

\begin{definition}\label{defsimpobj}
A simplicial object $X_.$ in a model category $\mathcal E$ is good if all objects $X_n$ are cofibrant in 
$\mathcal E$ and all degeneracy operators $s_i:X_n\to X_{n+1},\,0\leq i\leq n,$ are cofibrations in $\mathcal E$.

A (colimit-preserving) realisation functor for simplicial objects is good if any degree-wise weak equivalence between good simplicial objects in $\mathcal E$ 
realises to a weak equivalence in $\mathcal E$.
\end{definition}

Observe that in a model category in which the cofibrations are precisely the monomorphisms (like for instance in the Quillen model category of simplicial sets)
all simplicial objects are good, since all objects are cofibrant and degeneracy operators always act as split monomorphisms. It is well known that the canonical realisation functor for simplicial objects in simplicial sets is good. 
The canonical realisation functor for simplicial objects in topological spaces is also good, cf. \cite{MayGeoIterLoopSp}.

The main technical result of this section reads then as follows, where for simplicity we call a cofibration \emph{strong} 
if domain and codomain are cofibrant:

\begin{prop}
Let $T$ be a monad on a pointed model category $\mathcal E$ such that the unit $\eta_X:X\to T(X)$ is a cofibration for each cofibrant object 
$X$ and such that $T$ preserves the zero-object. 
Assume furthermore that $\mathcal E$ has a good realisation functor and that for each strong cofibration $X\to Y$ in $\mathcal E$, the image $T(X)\to T(Y)$ is a strong cofibration and the induced map $T(Y)/T(X)\to T(Y/X)$ is weak equivalence.

Then, for any map of $T$-algebras $f:W\to W'$ for which $U_T(f)$ is a strong cofibration and $U_T(W'/W)$ is cofibrant, the induced map $U_T(W')/U_T(W)\to U_T(W'/W)$ is a weak equivalence in $\mathcal E$.
\end{prop}

\begin{proof}
The bar resolution applied to the sequence 
$W\to W'\to W'/W$ defines a sequence $\mathcal B_.(W)\to \mathcal B_.(W')\to\mathcal B_.(W'/W)$ whose underlying sequence is a sequence of good simplicial objects in $\mathcal E$,
as follows from the assumptions made on $T$ and on $f$.
The quotient $U_T\mathcal B_.(W')/U_T\mathcal B_.(W)$ in $\mathcal E$ is also good and contains the constant simplicial object $U_T(W')/U_T(W)$ as a simplicial deformation retract
\footnote{This is actually not true in general, but it holds in the case we are interested in e.g. in the category of $\Gamma$-spaces (with values in simplicial sets).}.
Since the realisation functor commutes with quotients we thus get a weak equivalence $\mathcal B(W')/\mathcal B(W)\to U_T(W')/U_T(W)$.

The canonical map $U_T\mathcal B_.(W')/U_T\mathcal B_.(W)\to U_T\mathcal B_.(W'/W)$ is degree-wise of the form $T^{n+1}(Y)/T^{n+1}(X)\to T^{n+1}(Y/X)$ for some strong cofibration $X \to Y$ and hence a degree-wise weak equivalence by an easy induction.
Since both simplicial objects are good and the realisation functor is also good, we get by realisation a weak equivalence $\mathcal B(W')/\mathcal B(W)\to\mathcal B(W'/W)$. 

The weak equivalences $\mathcal B(W')/\mathcal B(W)\to\mathcal B(W'/W), \mathcal B(W'/W)\to U_T(W'/W)$ and $\mathcal B(W')/\mathcal B(W)\to U_T(W')/U_T(W)$ together with the 2 out of 3 property of weak equivalences finally give the required weak equivalence $U_T(W')/U_T(W)\to U_T(W'/W).$
\end{proof}

In order to extend the result of the preceding proposition to general free cell attachments in $\mathcal Alg_T$ we need to impose a further condition on the pointed model category $\mathcal E$.

\begin{definition}\label{defexc}
A pointed model category $\mathcal E$ satisfies \emph{excision} if for any map $f:Y_1\to Y_2$ of strong cofibrations $X\to Y_1$ and $X\to Y_2$
$$
\xymatrix @C=1.5cm @R=1.2cm{\relax 
Y_1 \ar[rr]^{f}  && Y_2  \\ 
& X \ar[lu] \ar[ru]
}
$$
$f:Y_1\to Y_2$ is a weak equivalence if and only if $f/X:Y_1/X\to Y_2/X$ is a weak equivalence.
\end{definition}

Observe that the "only if" part is true in any pointed model category. "Excision" holds typically in $Top_*$ \emph{after} Bousfield localisation with respect to a generalised homology theory.

\begin{prop}\label{propForgfuncRE}
Let $\mathcal E$ be the pointed model category with good realisation functor and with excision. Let $T$ be a monad which preserves the zero-object and whose unit is a cofibration at each cofibrant object. Then the following are equivalent:
\begin{enumerate}
 \item The forgetful functor functor takes free cell attachments in $\mathcal Alg_T$ (as described below) to homotopical cell attachments in $\mathcal E$ (cf. Definition \ref{defcellatt}).
 \item The forgetful functor takes free cell extensions in $\mathcal Alg_T$ to cofibrations in $\mathcal E$ and for any strong cofibration $X\to Y$ in $\mathcal E$, the image $T(X)\to T(Y)$ is a strong cofibration and the induced map $T(Y)/T(X)\to T(Y/X)$ is a weak equivalence.
\end{enumerate}
\end{prop}

\begin{proof}
A \textit{free cell attachment} is by definition a pushout diagram in $\mathcal Alg_T$ of the form
$$
\shorthandoff{;:!?}
 \xymatrix @!0 @C=2.4cm @R=2.4cm{\relax
    F_T(X) \ar[r] \ar[d] & W \ar[d] \\
    F_T(Y) \ar[r]  & W' \\
  }
$$
for any strong cofibration $X\to Y$ in $\mathcal E$ where we assume furthermore that $U_T(W)$ is cofibrant. We call the induced map $W\to W'$ a \textit{free cell extension}.

If the forgetful functor $U_T$ takes this pushout to a homotopical cell attachment in $\mathcal E$, then by definition (cf. Definition \ref{defcellatt}) the 
underlying map $U_T(W)\to U_T(W')$ is a strong cofibration in $\mathcal E$, hence $(a)$ implies the first half of $(b)$. If we take $W$ to be the zero-object $*$ of $\mathcal Alg_T$, then $W'$ may be 
identified with $F_T(Y)/F_T(X)=F_T(Y/X)$. Moreover, the pushout of $U_TF_T(X)\to U_TF_T(Y)$ along $U_TF_T(X)\to U_T(*)=*$ yields the quotient 
$U_TF_T(Y)/U_TF_T(X)$, whence (since $T=U_TF_T$) a weak equivalence $T(Y)/T(X)\to T(Y/X)$ so that $(a)$ implies also the second half of $(b)$.

Now assume that $(b)$ holds and consider the free cell attachment above. We have to show that its image under $U_T$ is a homotopical cell attachment in $\mathcal E$, 
i.e. that the comparison map $U_T(W)\cup_{TX}T(Y)\to U_T(W')$ is a weak equivalence in $\mathcal E$. This comparison map is a map of cofibrations under $U_T(W)$. 
By excision, it is thus equivalent to show that the quotient map $(U_T(W)\cup_{TX}T(Y))/U_T(W)\to U_T(W')/U_T(W)$ is a weak equivalence.

Since the free cell attachment $W\to W'$ has an underlying cofibration, and since the quotient $W'/W$ in $\mathcal Alg_T$ is isomorphic to $F_T(Y)/F_T(X)=F_T(Y/X)$ and hence has an underlying cofibrant object $T(Y/X)$, the 
preceding proposition gives a canonical weak equivalence $U_T(W')/U_T(W)\to U_T(W'/W)$.
By the 2 out of 3 property of weak equivalences, it suffices thus to show that the aforementioned quotient map $(U_T(W)\cup_{TX}T(Y))/U_T(W)\to U_T(W')/U_T(W)$ composed with $U_T(W')/U_T(W)\to U_T(W'/W)$ is a weak equivalence. This composite map may be identified with the canonical map $T(Y)/T(X)\to T(Y/X)$, which is a weak 
equivalence by assumption.
\end{proof}  



\chapter{Tensorial strength}

This chapter is devoted to some ``strong'' constructions obtained by using the notion of tensorial strength. This concept has been introduced by Anders Kock \cite{KockSF,KockMS}. Our main contribution consists in rephrasing the main constructions of Kock from a 2-categorical view point. This has the advantage to emphasize the relative character of Kock's construction which originally has only been applied to enriched monads and not to enriched functors. The 2-categorical view point illustrates very clearly the correspondence betwen strength and enrichment.

In Section \ref{SStr}, we define the notions of a strong functor, strong natural transformation and strong monad. In the following section, Section \ref{SSaEnr}, we study the correspondence between the tensorial strength and the enrichment. This allows us to closely relate strong and enriched functors as well as strong and enriched natural transformations. In Section \ref{S2CTh} we obtain a 2-isomorphism between the 2-category of strong functors and natural transformations and the 2-category of $\mathcal E$-functors and $\mathcal E$-natural transformations. The consequence is an equivalence between strong and enriched monads. In Section \ref{SSmDayC} we use the context of Day convolution to construct strong monads on the category $\mathcal E^{\mathcal A}$ of functors from $\mathcal A$ to $\mathcal E$.

\section{Strong...}\label{SStr}

In this section, using the concept of a tensorial strength, we introduce the notion of a strong functor, followed by the notions of strong natural transformation and strong monad \cite{KockSF, GoubLRMT}.

\subsection{Strong functors}

We start by transposing the notion of a functor to the strong context.
\begin{definition}
Let $\mathcal E$ be a closed symmetric monoidal category. Let $\mathcal A$ and $\mathcal B$ be two $\mathcal E$-categories tensored over $\mathcal E$. 

A strong functor $\left(T,\sigma\right)$ consists in giving:
\begin{enumerate}
\item A functor $T:\mathcal A\rightarrow\mathcal B$;
\item For every object X in $\mathcal E$ and A in $\mathcal A$, a tensorial strength $$\sigma_{X,A}: X\otimes TA\rightarrow T\left(X\otimes A\right)$$
natural in both variables, such that the following diagrams commute: 
\end{enumerate}
\begin{diagram}{Unit axiom}
\label{uniff}
$$
\shorthandoff{;:!?}
\xymatrix @!0 @C=3cm @R=1.5cm{\relax 
 I\otimes TA \ar[rr]^{\sigma_{I,A}} \ar[rd]_{l_{TA}} && T\left(I\otimes A\right) \ar[ld]^{T\left( l_{A}\right) } \\ 
& TA 
}
$$ 
\end{diagram}
\begin{diagram}{Associativity axiom}
\label{diagassff}
$$
\shorthandoff{;:!?}
\xymatrix @!0 @C=5cm @R=2.5cm{\relax 
 (X\otimes Y)\otimes TA \ar[r]^{a_{X,Y,A}} \ar[d]_{\sigma_{X\otimes Y,A}} & X\otimes(Y\otimes TA) \ar[r]^{X\otimes\sigma_{Y,A}}  & X\otimes T\left(Y\otimes A\right) \ar[d]^{\sigma_{X,Y\otimes A}} \\ 
T\left((X\otimes Y)\otimes A\right)\ar[rr]^{Ta_{X,Y,A}} && T\left(X\otimes (Y\otimes A)\right) 
}
$$
\end{diagram}
\end{definition}
\begin{rem}~\\
\begin{enumerate}
 \item One can assume that the tensor is strictly associative. 
Indeed, by the coherence theorem of Maclane (cf. \cite{MaclCWM}), Diagram \ref{diagassff} can be replaced by the following one: 
\begin{diagram}{}\label{diagasssimplff}
$$
\shorthandoff{;:!?}
\xymatrix @!0 @C=3cm @R=1.5cm{\relax 
 X\otimes Y\otimes TA \ar[rr]^{X\otimes\sigma_{Y,A}} \ar[rd]_{\sigma_{X\otimes Y,A}} && X\otimes T\left(Y\otimes A\right) \ar[ld]^{\sigma_{X,Y\otimes A}} \\ 
& T\left(X\otimes Y\otimes A\right) 
}
$$
\end{diagram}
 \item A strong functor T consists in giving a functor $T_{0}:\mathcal A_{0}\rightarrow \mathcal B_{0}$
on the underlying categories equipped with a tensorial strength and satisfying corresponding axioms.
\end{enumerate}
\end{rem}

There is a dual notion of tensorial strength, where T ``acts'' on the left. Moreover, we could define a dual strong functor.
\begin{definition}
Let $\mathcal E$ be a closed symmetric monoidal category. Let $\mathcal A$ and $\mathcal B$ be two $\mathcal E$-categories tensored over $\mathcal E$ and a functor $T:\mathcal A\rightarrow\mathcal B$. 

A dual tensorial strength $$\sigma'_{A,X}: TA\otimes X \rightarrow T\left(A\otimes X\right)$$ is given by the following commutative diagram
\begin{diagram}{}
$$
\shorthandoff{;:!?}
\xymatrix @C=3cm @R=2.5cm{\relax
TA\otimes X \ar[r]^{\sigma'_{A,X}} \ar[d]_{a_{TA,X}} & T\left(A\otimes X\right)  \\
X\otimes TA \ar[r]^{\sigma_{X,A}} & T\left(X\otimes A\right) \ar[u]_{T(a_{X,A})} 
}
$$
\end{diagram}
\end{definition}

By composing two strong functors, we acquire another strong functor.
\begin{definition}
\label{deffcomp}
Let $\mathcal E$ be a closed symmetric monoidal category and $\mathcal A$, $\mathcal B$ and $\mathcal C$ three $\mathcal E$-categories tensored over $\mathcal E$.
Let $\left(T_{1},\sigma_{1}\right)$ and $\left(T_{2},\sigma_{2}\right)$ be two strong functors such that $T_{1}:\mathcal A\rightarrow \mathcal B$ and $T_{2}:\mathcal B\rightarrow \mathcal C$.

Composition of two strong functors is a strong functor $\left(T_{2}T_{1},\sigma_{2,1}\right)$ where the tensorial strength is given by the following commutative diagram
\begin{diagram}{}
$$
\shorthandoff{;:!?}
\xymatrix @!0 @C=3cm @R=1.5cm{\relax 
 X\otimes T_{2}T_{1}A \ar[rr]^{\sigma_{2,1}} \ar[rd]_{\sigma_{2}} && T_{2}T_{1}\left(X\otimes A\right)  \\ 
&  T_{2}\left(X\otimes T_{1}A\right) \ar[ru]_{T_{2}\left(\sigma_{1}\right) }  
}
$$ 
\end{diagram}
\end{definition}

\subsection{Strong natural transformations}

Analogously, we generalize the notion of a natural transformation to the strong context. 

\begin{definition}\label{defstnatt}
Let $\mathcal E$ be a closed symmetric monoidal category and $\mathcal A$ and $\mathcal B$ two $\mathcal E$-categories tensored over $\mathcal E$. Let $\left(T_{1},\sigma_{1}\right)$,$\left(T_{2},\sigma_{2}\right)$ be two strong functors such that $T_{1},T_{2}:\mathcal A\rightarrow \mathcal B$.

A strong natural transformation $\Psi:T_{1}\Rightarrow T_{2}$ is given by the following commutative diagram: 
\begin{diagram}{}
$$
\shorthandoff{;:!?}
\xymatrix @C=3cm @R=2.5cm{\relax
X\otimes T_{1}A \ar[r]^{\sigma_{1}} \ar[d]_{X\otimes \Psi_{A}} & T_{1}\left(X\otimes A\right) \ar[d]^{\Psi_{X\otimes A}}\\
X\otimes T_{2}A \ar[r]^{\sigma_{2}} & T_{2}\left(X\otimes A\right) 
}
$$
\end{diagram}
\end{definition}

\begin{rem}
A strong natural transformation consists in giving an ordinary natural transformation $\Psi:T_{1}\Rightarrow T_{2}$ satisfying a property of compatibility with the tensorial strength.
\end{rem}

\subsection{Strong monads}

Similarly, we generalize the notion of a monad.

\begin{definition}\label{defStMon}
Let $\mathcal E$ be a closed symmetric monoidal category. 
A strong monad $\left(T,\mu,\eta, \sigma\right)$ in a category $\mathcal E$ consists in giving:
\begin{enumerate}
\item A monad $\left(T,\mu,\eta\right)$ in a category $\mathcal E$;
\item A tensorial strength $\sigma_{A,B}: A\otimes TB\rightarrow T\left(A\otimes B\right)$ 
natural in both variables, such that the following four diagrams commute:
\end{enumerate}
\begin{diagram}{Unit condition for $\sigma$}
$$
\shorthandoff{;:!?}
\xymatrix @!0 @C=3cm @R=1.5cm{\relax 
 I\otimes TA \ar[rr]^{\sigma_{I,A}} \ar[rd]_{l_{TA}} && T\left(I\otimes A\right) \ar[ld]^{T\left( l_{A}\right) } \\ 
& TA 
}
$$ 
\end{diagram}
\begin{diagram}{Associativity condition for $\sigma$}\label{diagassbeta}
$$
\shorthandoff{;:!?}
\xymatrix @!0 @C=5cm @R=2.5cm{\relax 
 (A\otimes B)\otimes TC \ar[r]^{a_{A,B,C}} \ar[d]_{\sigma_{A\otimes B,C}} & A\otimes(B\otimes TC) \ar[r]^{A\otimes\sigma_{B,C}}  & A\otimes T\left(B\otimes C\right) \ar[d]^{\sigma_{A,B\otimes C}} \\ 
T\left((A\otimes B)\otimes C\right)\ar[rr]^{Ta_{A,B,C}} && T\left(A\otimes (B\otimes C)\right) 
}
$$
\end{diagram}
\begin{diagram}{Strong naturality condition for $\eta$}\label{diagsteta}
$$
\shorthandoff{;:!?}
\xymatrix @!0 @C=3cm @R=1.5cm{\relax 
 A\otimes TB \ar[rr]^{\sigma_{A,B}} && T\left(A\otimes B\right)  \\ 
& A\otimes B \ar[lu]^{A\otimes \eta_{B}} \ar[ru]_{\eta_{A\otimes B}} 
}
$$ 
\end{diagram}\label{diagstmu}
\begin{diagram}{Strong naturality condition for $\mu$}
$$
\shorthandoff{;:!?}
\xymatrix @!0 @C=4.7cm @R=2.3cm{\relax
A\otimes T^{2}B \ar[r]^{\sigma_{A,TB}} \ar[d]_{A\otimes \mu_{B}} & T\left(A\otimes TB\right) \ar[r]^{T\left(\sigma_{A,B}\right)} & T^{2}\left(A\otimes B\right) \ar[d]^{\mu_{A\otimes B}}\\
A\otimes TB \ar[rr]^{\sigma_{A,B}} && T\left(A\otimes B\right) 
}
$$
\end{diagram}
\end{definition}

\begin{rem}~\\
\begin{enumerate}
 \item Tabareau \cite{TabMRCLT} defines left and right strong monads. Indeed, in the context of Tabareau, Definition \ref{defStMon} corresponds to the right strong monad and a left strong monad corresponds to a monad with a dual tensorial strength $\sigma'_{A,B}:TA\otimes B\rightarrow T(A\otimes B)$ and satisfying the commutativity conditions of dual diagrams.
Furthermore, when the monoidal category $\mathcal E$ is symmetric, a right strong monad admits automatically a dual tensorial strength, which makes the monad strong at right and left: $$\sigma'_{A,B}=T\left(a_{B,A}\right)\circ \sigma_{B,A}\circ a_{TA,B}$$ 
 \item One can assume that the tensor is strictly associatif. Indeed, by the coherence theorem of Maclane (cf. \cite{MaclCWM}), one can replace Diagram \ref{diagassbeta} by the following one: 
\begin{diagram}{}\label{diagasssimpl}
$$
\shorthandoff{;:!?}
\xymatrix @!0 @C=3cm @R=1.5cm{\relax 
 A\otimes B\otimes TC \ar[rr]^{A\otimes\sigma_{B,C}} \ar[rd]_{\sigma_{A\otimes B,C}} && A\otimes T\left(B\otimes C\right) \ar[ld]^{\sigma_{A,B\otimes C}} \\ 
& T\left(A\otimes B\otimes C\right) 
}
$$
\end{diagram}
We will use this simplification, i.e. Diagram \ref{diagasssimpl} instead of Diagram \ref{diagassbeta} for the rest of the thesis.
 \item The unit and the associativity axioms simply translate the fact that we require for T to be a strong functor.
On the other hand, the strong naturality conditions for $\eta$ and $\mu$ translate the fact that we require for $\eta$ and $\mu$ to be strong natural transformations.

\end{enumerate}
\end{rem}

We can require in addition for the tensorial strength to be an isomorphism. Then we have the following definition of a very strong monad.
 
\begin{definition}
A very strong monad $\left(T,\mu,\eta,\sigma \right)$ in the category $\mathcal E$ consists in giving a strong monad $\left(T,\mu,\eta,\sigma\right)$ such that the tensorial strength $$\sigma_{X,Y}: X\otimes TY \xrightarrow{\cong} T\left(X\otimes Y\right)$$ is an isomorphism for X, Y in $\mathcal E$. 
\end{definition}

\begin{rem}
Strong monads and morphisms of strong monads in a monoidal category $\mathcal E$ constitute a category, written $StMonads\left(\mathcal E\right) $.

Similarly, very strong monads and morphisms of very strong monads in a monoidal category $\mathcal E$ constitute a category, written $VStMonads\left(\mathcal E\right) $.
\end{rem}

\begin{rem}
We have the following inclusion of categories:
$$VStMonads\left(\mathcal E\right)\subset StMonads\left(\mathcal E\right) \subset Monads\left(\mathcal E\right) $$ 
\end{rem}

\begin{prop}(cf. \cite{BMDCAO}, Proposition1.9)
Let $\mathcal E$ be a monoidal category.
There is a correspondence between:
\begin{enumerate}
 \item The category of very strong monads $VStMonads\left(\mathcal E\right)$; 
 \item The category of monoids $Monoids\left(\mathcal E\right)$.
\end{enumerate}
More precisely, the functor which associates to a monoid its induced strong monad is fully faithful and its essential image consists of the very strong monads. 
\end{prop}

\begin{proof}~\\
In the category of very strong monads $VStMonads\left(\mathcal E\right)$ the tensorial strength is an isomorphism.
In particular we have: $$X\otimes T\left(I \right)\cong T\left(X\otimes I \right) \cong T\left( X\right)$$
Hence for every monad T, we obtain an object $T\left(I \right)$ in $\mathcal E$, which has a structure of a monoid (see Proposition \ref{propmonTI}). 

On the other hand, since every monoid M in $\mathcal E$ induces a monad $-\otimes M$ in $\mathcal E$, the category of monoids $Monoids\left(\mathcal E\right)$ induces the category of very strong monads $VStMonads\left(\mathcal E\right)$. 
\end{proof}




\section{Strength and enrichment}\label{SSaEnr}

In this section, we establish a correspondence between the tensorial strength and the enrichment. This leads us naturally to study the relations between strong functors and enriched functors, as well as between strong natural transformations and enriched natural transformations. 

\subsection{Correspondence between strength and enrichment}

In order to relate the tensorial strength and the enrichment, we define a tensorial strength associated to an enrichment and vice-versa.

\begin{definition}
Let $\mathcal E$ be a closed symmetric monoidal category. Let $\mathcal A$ and $\mathcal B$ be two categories tensored over $\mathcal E$ and let $\left( T,\varphi\right) :\mathcal A \rightarrow \mathcal B$ be a $\mathcal E$-functor where $\varphi_{A,B}:\underline{\mathcal A}\left(A,B\right)\rightarrow\underline{\mathcal B}\left(TA,TB\right)$ denotes the enrichment. 

We define a tensorial strength $\sigma_{X,A}: X\otimes TA\rightarrow T\left(X\otimes A\right)$ by the following commutative diagram:
\begin{diagram}{}
\label{famillesigma}
$$
\shorthandoff{;:!?}
\xymatrix @!0 @C=7cm @R=2.8cm{\relax
X\otimes TA \ar[r]^{\sigma_{X,A}} \ar[d]_{\gamma_{A}\otimes TA} & T\left(X\otimes A\right)\\
\underline{\mathcal A}\left(A,X\otimes A \right)\otimes TA \ar[r]^{\varphi_{A,X\otimes A}\otimes TA} & \underline{\mathcal B}\left(TA,T\left( X\otimes A\right)\right)\otimes TA\ar[u]_{ev_{TA}}\\
}
$$
\end{diagram}   
\end{definition}
\begin{rem}
By adjunction, the diagram \ref{famillesigma} is equivalent to:
\begin{diagram}{}
$$
\shorthandoff{;:!?}
\xymatrix @!0 @C=6cm @R=2.3cm{\relax
X \ar[r]^{\widehat{\sigma}_{X,A}} \ar[d]_{\gamma_{A}} & \underline{\mathcal B}\left(TA,T\left(X\otimes A\right)\right) \\
\underline{\mathcal A}\left(A,X\otimes A \right) \ar[r]^{\varphi_{A,X\otimes A}} & \underline{\mathcal B}\left(TA,T\left( X\otimes A\right)\right)\ar@{=}[u]\\
}
$$
\end{diagram}
Hence, we have $\widehat{\sigma}=\varphi\circ\gamma$.  
\end{rem}

\begin{definition}
Let $\mathcal E$ be a closed symmetric monoidal category. Let $\mathcal A$ and $\mathcal B$ be two categories tensored over $\mathcal E$ and let $T:\mathcal A \rightarrow \mathcal B$ be a functor.

To a tensorial strength $\sigma_{X,A}: X\otimes TA\rightarrow T\left(X\otimes A\right)$ we associate an enrichment $\varphi_{A,B}:\underline{\mathcal A}\left(A,B\right)\rightarrow\underline{\mathcal B}\left(TA,TB\right)$ by the following commutative diagram:
\begin{diagram}{}
\label{famillefi}
$$
\shorthandoff{;:!?}
\xymatrix @!0 @C=7cm @R=2.8cm{\relax
\underline{\mathcal A}\left(A,B\right) \ar[r]^{\varphi_{A,B}} \ar[d]_{\gamma_{TA}} & \underline{\mathcal B}\left(TA,TB\right)\\
\underline{\mathcal B}\left(TA,\underline{\mathcal A}\left(A,B\right)\otimes TA\right) \ar[r]^{\underline{\mathcal B}\left(TA,\sigma_{\underline{\mathcal A}\left(A,B\right),A}\right) } & \underline{\mathcal B}\left(TA,T\left(\underline{\mathcal A}\left(A,B\right)\otimes A\right)\right)\ar[u]_{\underline{\mathcal B}\left(TA,T\left(ev\right)\right)}\\
}
$$
\end{diagram}   
\end{definition}
\begin{rem}
By adjunction, the diagram \ref{famillefi} is equivalent to:
\begin{diagram}{}
$$
\shorthandoff{;:!?}
\xymatrix @!0 @C=6cm @R=2.3cm{\relax
\underline{\mathcal A}\left(A,B\right)\otimes TA \ar[r]^{\widehat{\varphi}_{A,B}} \ar@{=}[d] & TB\\
\underline{\mathcal A}\left(A,B\right)\otimes TA \ar[r]^{\sigma_{\underline{\mathcal A}\left(A,B\right),A}} & T\left(\underline{\mathcal A}\left(A,B\right)\otimes A\right) \ar[u]_{T\left(ev_{A}\right)}\\
}
$$
\end{diagram}
Hence we have $\widehat{\varphi}=T\left(ev\right)\circ\sigma$.  
\end{rem}

The following lemma provides a correspondence between an enrichment and a tensorial strength.
   
\begin{lem}\label{lemCStEnr}
There is a canonical correspondence between:
\begin{enumerate}
\item An enrichment of the functor T: $$\varphi_{A,B}:\underline{\mathcal A}\left(A,B\right)\rightarrow\underline{\mathcal B}\left(TA,TB\right)$$
\item A tensorial strength for the functor T: $$\sigma_{X,A}: X\otimes TA\rightarrow T\left(X\otimes A\right)$$;
\end{enumerate}
i.e. the two constructions are mutually inverse.
\end{lem}

\begin{proof}

First, we prove that enrichment of the functor T determins the tensorial strength.

We need to prove the commutativity of the following diagram

$$
\shorthandoff{;:!?}
\xymatrix @R=2cm {\relax
X\otimes TA \ar[rr]^{\sigma_{X,A}} \ar[d]_{\gamma_{A}\otimes 1} && T\left(X\otimes A\right)\\
\underline{\mathcal A}\left(A,X\otimes A \right)\otimes TA \ar[d]_{\gamma_{TA}\otimes 1} && \underline{\mathcal B}\left(TA,T\left(X\otimes A\right)\right)\otimes TA \ar[u]_{ev} \\
\underline{\mathcal B}\left(TA,\underline{\mathcal A}\left(A,X\otimes A \right)\otimes TA \right)\otimes TA \ar[rr]^{\underline{\mathcal B}\left(TA,\sigma\right)\otimes 1} && \underline{\mathcal B}\left(TA,T\left(\underline{\mathcal A}\left(A,X\otimes A \right)\otimes A\right)\right)\otimes TA \ar[u]_{\underline{\mathcal B}\left(TA,T(ev)\right)\otimes1}\\
}
$$
By adjunction, this diagram is equivalent to the following one
$$
\shorthandoff{;:!?}
\xymatrix @R=1.8cm {\relax
X \ar[rr]^{\widehat{\sigma}_{X,A}} \ar[d]_{\gamma_{A}} && \underline{\mathcal B}\left(TA, T\left(X\otimes A\right)\right) \\
\underline{\mathcal A}\left(A,X\otimes A \right)\ar[d]_{\gamma_{TA}} && \underline{\mathcal B}\left(TA,T\left(X\otimes A\right)\right) \ar@{=}[u] \\
\underline{\mathcal B}\left(TA,\underline{\mathcal A}\left(A,X\otimes A \right)\otimes TA \right) \ar[rr]^{\underline{\mathcal B}\left(TA,\sigma\right)} && \underline{\mathcal B}\left(TA,T\left(\underline{\mathcal A}\left(A,X\otimes A \right)\otimes A\right)\right) \ar[u]_{\underline{\mathcal B}\left(TA,T(ev)\right)}\\
}
$$ 
Once again, by adjunction and using the fact that the adjoint of the morphism $\gamma_{TA}\circ\gamma_{A}$ is $\gamma_{A}\otimes TA$, we obtain
$$
\shorthandoff{;:!?}
\xymatrix @R=2cm @C=2cm {\relax
X\otimes TA \ar[rr]^{\sigma_{X,A}} \ar[d]_{\gamma_{A}\otimes 1} && T\left(X\otimes A\right)\\
\underline{\mathcal A}\left(A,X\otimes A \right)\otimes TA \ar[rr]^{\sigma_{\underline{\mathcal A}\left(A,X\otimes A \right),A}} && T\left(\underline{\mathcal A}\left(A,X\otimes A \right)\otimes A\right) \ar[u]_{T\left(ev\right)} \\
}
$$
Using the naturality of $\sigma$, this diagram is equivalent to
$$
\shorthandoff{;:!?}
\xymatrix @R=2cm @C=2cm {\relax
X\otimes TA \ar[rr]^{\sigma_{X,A}} \ar[d]_{\sigma_{X,A}} && T\left(X\otimes A\right)\\
T\left(X\otimes A\right) \ar[rr]^{T\left(\gamma_{A}\otimes A\right)} && T\left(\underline{\mathcal A}\left(A,X\otimes A \right)\otimes A\right) \ar[u]_{T\left(ev\right)} \\
}
$$
But $\hat{\gamma}_{A}=1_{X\otimes A}$, hence $ T\left(ev\right)\circ T\left(\gamma_{A}\otimes A\right)= T\left( \widehat{\gamma}_{A}\right) =1_{T\left(X\otimes A\right)}$.

Therefore, this diagram commutes and the family $\textit{(a)}$ determins $\textit{(b)}$.

It remains to prove that the tensorial strength determins the enrichment of the functor T.

Extending definitions
$$
\shorthandoff{;:!?}
\xymatrix @R=2cm @C=2cm {\relax
\underline{\mathcal A}\left(A,B\right) \ar[r]^{\varphi_{A,B}} \ar[d]_{\gamma_{TA}} & \underline{\mathcal B}\left(TA,TB\right)\\
\underline{\mathcal B}\left(TA,\underline{\mathcal A}\left(A,B\right)\otimes TA\right) \ar[d]_{\underline{\mathcal B}\left(TA,\gamma_{A}\otimes TA\right)} & \underline{\mathcal B}\left(TA,T\left(\underline{\mathcal A}\left(A,B\right)\otimes A\right)\right)\ar[u]_{\underline{\mathcal B}\left(TA,T\left(ev\right)\right)}\\
\underline{\mathcal B}\left(TA,\underline{\mathcal A}\left(A,\underline{\mathcal A}\left(A,B\right)\otimes A\right)\otimes TA\right)\ar[r]^{\underline{\mathcal B}\left(TA,\varphi\otimes TA\right)} & \underline{\mathcal B}\left(TA,\underline{\mathcal B}\left(TA,T\left(\underline{\mathcal A}\left(A,B\right)\otimes A\right)\right) \otimes TA\right) \ar[u]_{\underline{\mathcal B}\left(TA,ev\right)} \\ 
}
$$
By adjunction, this diagram is equivalent to
$$
\shorthandoff{;:!?}
\xymatrix @C=3cm @R=1.8cm {\relax
\underline{\mathcal A}\left(A,B\right)\otimes TA \ar[r]^{\widehat{\varphi}_{A,B}} \ar@{=}[d] & TB\\
\underline{\mathcal A}\left(A,B\right)\otimes TA \ar[d]_{\gamma_{A}\otimes TA} & T\left(\underline{\mathcal A}\left(A,B\right)\otimes A\right) \ar[u]_{T\left(ev\right)}\\
\underline{\mathcal A}\left(A,\underline{\mathcal A}\left(A,B\right)\otimes A\right)\otimes TA\ar[r]^{\varphi\otimes TA} & \underline{\mathcal B}\left(TA,T\left(\underline{\mathcal A}\left(A,B\right)\otimes A\right)\right) \otimes TA \ar[u]_{ev} \\ 
}
$$
Once again, by adjunction and using the fact that the adjoint of the morphism $T\left(ev\right)\circ ev$  is $\underline{\mathcal B}\left(1,T\left(ev\right)\right)$, we have
$$
\shorthandoff{;:!?}
\xymatrix @C=2.8cm @R=2cm{\relax
\underline{\mathcal A}\left(A,B\right) \ar[r]^{\varphi_{A,B}} \ar[d]_{\gamma_{A}} & \underline{\mathcal B}\left(TA,TB\right)\\
\underline{\mathcal A}\left(A,\underline{\mathcal A}\left(A,B\right)\otimes A\right) \ar[r]^{\varphi_{A,\underline{\mathcal A}\left(A,B\right)\otimes A}} & \underline{\mathcal B}\left(TA,T\left(\underline{\mathcal A}\left(A,B\right)\otimes A\right)\right) \ar[u]_{\underline{\mathcal B}\left(1,T\left(ev\right)\right)}\\
}
$$
By naturality of $\varphi$, this diagram is equivalent to
$$
\shorthandoff{;:!?}
\xymatrix @C=3.5cm @R=2cm{\relax
\underline{\mathcal A}\left(A,B\right) \ar[r]^{\varphi_{A,B}} \ar[d]_{\gamma_{A}} & \underline{\mathcal B}\left(TA,TB\right)\\
\underline{\mathcal A}\left(A,\underline{\mathcal A}\left(A,B\right)\otimes A\right) \ar[r]^{\underline{\mathcal A}\left(1,ev\right)} & \underline{\mathcal A}\left(A,B\right) \ar[u]_{\varphi_{A,B}}\\
}
$$
But one has $\widehat{ev}=\underline{\mathcal A}\left(A,ev\right)\circ\gamma_{A}=1_{\underline{\mathcal A}\left(A,B\right)}$.

Therefore, this diagram commutes and the family $\textit{(b)}$ determins $\textit{(a)}$ i.e. the two constructions are mutually inverse.
\end{proof}

\subsection{Strong and enriched functors} 

Once we have the correspondence between a tensorial strength and an enrichment (Lemma \ref{lemCStEnr}), we can closely relate strong and enriched functors.

\begin{prop}\label{propfonf-e}
Let $\mathcal E$ be a closed symmetric monoidal category. Given two categories $\mathcal A$ and $\mathcal B$ tensored over $\mathcal E$ and a functor $T:\mathcal A\rightarrow\mathcal B$, the following conditions are equivalent:
\begin{enumerate}
 \item A functor T extends to a strong functor $\left(T,\sigma \right)$;
 \item A functor T extends to a $\mathcal E$-functor $\left(T,\varphi \right)$.  
\end{enumerate}
\end{prop}

\begin{proof}
First, we prove that $\textit{(a)}$ implies $\textit{(b)}$.

More precisely, if the tensorial strength $\sigma_{X,A}$ satisfies the unit and the associativity axioms, then the enrichment $\varphi_{A,B}$ satisfies the unit and the composition axioms. First, we prove that $\varphi_{A,B}$ satisfies the unit axiom
$$
\shorthandoff{;:!?}
\xymatrix @!0 @C=3cm @R=1.5cm{\relax 
I \ar[rd]_{j_A} \ar[rr]^{j_{TA}}  && \underline{\mathcal B}\left(TA,TA\right)  \\ 
& \underline{\mathcal A}\left(A,A\right)  \ar[ru]_{\varphi_{AA}} 
}
$$
By adjunction, this diagram is equivalent to the following one
$$
\shorthandoff{;:!?}
\xymatrix @C=3cm @R=1.8cm{\relax 
I\otimes TA \ar[r]^{l_{TA}} \ar[d]_{j_{A}\otimes TA}  & TA   \\ 
\underline{\mathcal A}\left(A,A\right)\otimes TA  \ar[r]^{\varphi_{A,A}\otimes TA} & \underline{\mathcal B}\left(TA,TA\right)\otimes TA \ar[u]_{ev} 
}
$$
But, we know that $ev \circ \varphi\otimes 1=\hat{\varphi}=T(ev)\circ \sigma$.

Hence, this diagram is equivalent to
$$
\shorthandoff{;:!?}
\xymatrix @C=4cm @R=2.8cm{\relax 
I\otimes TA \ar[r]^{l_{TA}} \ar[d]_{j_{A}\otimes TA} & TA \\  
\underline{\mathcal A}\left(A,A\right)\otimes TA \ar[r]^{\sigma_{\underline{\mathcal A}\left(A,A\right),TA}} & T\left(\underline{\mathcal A}\left(A,A\right)\otimes A\right) \ar[u]_{T(ev)}
}
$$
By naturality of $\sigma$
$$
\shorthandoff{;:!?}
\xymatrix @C=4cm @R=2.8cm{\relax 
I\otimes TA \ar[r]^{l_{TA}} \ar[d]_{\sigma_{I,A}}  & TA \\ 
T\left(I\otimes A \right) \ar[r]^{T\left(j_{A}\otimes A\right)} & T\left(\underline{\mathcal A}\left(A,A\right)\otimes A\right) \ar[u]_{T(ev)}
}
$$
But $\hat{j}_{A}=l_{A}$, hence $T\left(ev\right)\circ T\left(j_{A}\otimes A\right)=T\left(l_{A}\right)$.

Therefore, we have the following diagram
$$
\shorthandoff{;:!?}
\xymatrix @!0 @C=3cm @R=1.5cm{\relax 
 I\otimes TA \ar[rr]^{\sigma_{I,A}} \ar[rd]_{l_{TA}} && T\left(I\otimes A\right) \ar[ld]^{T\left( l_{A}\right) } \\ 
& TA 
}
$$ 
which clearly commutes by the unit axiom of $\sigma$.

Secondly, we prove that $\varphi_{A,B}$ satisfies the composition axiom
$$
\xymatrix @C=3cm @R=3cm {\relax
\underline{\mathcal A}\left(B,C\right)\otimes \underline{\mathcal A}\left(A,B\right) \ar[r]^{c_{ABC}} \ar[d]_{\varphi_{B,C}\otimes \varphi_{A,B}} & \underline{\mathcal A}\left(A,C\right) \ar[d]^{\varphi_{A,C}} \\
\underline{\mathcal B}\left(TB,TC\right)\otimes \underline{\mathcal B}\left(TA,TB\right) \ar[r]^{c_{TATBTC}} & \underline{\mathcal B}\left(TA,TC\right)
}
$$
By adjunction, this diagram is equivalent to the following one
$$
\xymatrix @C=2.7cm @R=2.2cm {\relax
\underline{\mathcal A}\left(B,C\right)\otimes \underline{\mathcal A}\left(A,B\right)\otimes TA \ar[r]^{c\otimes TA} \ar[d]_{\varphi_{B,C}\otimes \varphi_{A,B}\otimes TA} & \underline{\mathcal A}\left(A,C\right)\otimes TA \ar[d]^{\varphi_{A,C}\otimes TA} \\
\underline{\mathcal B}\left(TB,TC\right)\otimes \underline{\mathcal B}\left(TA,TB\right)\otimes TA \ar[d]_{c\otimes TA} & \underline{\mathcal B}\left(TA,TC\right)\otimes TA \ar[d]^{ev} \\
\underline{\mathcal B}\left(TA,TC\right)\otimes TA \ar[r]^{ev} & TC
}
$$
But $\hat{\varphi}=T(ev)\circ \sigma$, hence $ev \circ \varphi\otimes 1=T(ev) \circ \sigma$.

Therefore, the previous diagram is equivalent to the following one
$$
\xymatrix @C=2.7cm @R=2.2cm {\relax
\underline{\mathcal A}\left(B,C\right)\otimes \underline{\mathcal A}\left(A,B\right)\otimes TA \ar[r]^{c\otimes TA} \ar[d]_{\varphi_{B,C}\otimes \varphi_{A,B}\otimes TA} & \underline{\mathcal A}\left(A,C\right)\otimes TA \ar[d]^{\sigma_{\underline{\mathcal A}\left(A,C\right),A}} \\
\underline{\mathcal B}\left(TB,TC\right)\otimes \underline{\mathcal B}\left(TA,TB\right)\otimes TA \ar[d]_{c\otimes TA} & T\left( \underline{\mathcal A}\left(A,C\right)\otimes A\right)  \ar[d]^{T(ev)} \\
\underline{\mathcal B}\left(TA,TC\right)\otimes TA \ar[r]^{ev} & TC
}
$$ 
By naturality of $\sigma$
$$
\xymatrix @C=2.7cm @R=2cm {\relax
\underline{\mathcal A}\left(B,C\right)\otimes \underline{\mathcal A}\left(A,B\right)\otimes TA \ar[r]^{\sigma_{\underline{\mathcal A}\left(B,C\right)\otimes \underline{\mathcal A}\left(A,B\right),A}} \ar[d]_{\varphi_{B,C}\otimes \varphi_{A,B}\otimes TA} & T\left( \underline{\mathcal A}\left(B,C\right)\otimes \underline{\mathcal A}\left(A,B\right)\otimes A\right)  \ar[d]^{T\left( c\otimes A\right) } \\
\underline{\mathcal B}\left(TB,TC\right)\otimes \underline{\mathcal B}\left(TA,TB\right)\otimes TA \ar[d]_{c\otimes TA} & T\left(\underline{\mathcal A}\left(A,C\right)\otimes A\right) \ar[d]^{T(ev)} \\
\underline{\mathcal B}\left(TA,TC\right)\otimes TA \ar[r]^{ev} & TC
}
$$
By definition of composition $ev\circ c\otimes 1=ev\circ 1\otimes ev$, we have 

$$
\xymatrix @C=2.7cm @R=2cm {\relax
\underline{\mathcal A}\left(B,C\right)\otimes \underline{\mathcal A}\left(A,B\right)\otimes TA \ar[r]^{\sigma_{\underline{\mathcal A}\left(B,C\right)\otimes \underline{\mathcal A}\left(A,B\right),A}} \ar[d]_{\varphi_{B,C}\otimes \varphi_{A,B}\otimes TA} & T\left( \underline{\mathcal A}\left(B,C\right)\otimes \underline{\mathcal A}\left(A,B\right)\otimes A\right)  \ar[d]^{T\left( \underline{\mathcal A}\left(B,C\right) \otimes ev\right) } \\
\underline{\mathcal B}\left(TB,TC\right)\otimes \underline{\mathcal B}\left(TA,TB\right)\otimes TA \ar[d]_{\underline{\mathcal B}\left(TB,TC\right)\otimes ev} & T\left(\underline{\mathcal A}\left(B,C\right)\otimes B\right) \ar[d]^{T(ev)} \\
\underline{\mathcal B}\left(TB,TC\right)\otimes TB \ar[r]^{ev} & TC
}
$$
Once again, using the equality $\hat{\varphi}=T(ev)\circ \sigma$
$$
\xymatrix @C=2.7cm @R=2cm {\relax
\underline{\mathcal A}\left(B,C\right)\otimes \underline{\mathcal A}\left(A,B\right)\otimes TA \ar[r]^{\sigma_{\underline{\mathcal A}\left(B,C\right)\otimes \underline{\mathcal A}\left(A,B\right),A}} \ar[d]_{\varphi_{B,C}\otimes \underline{\mathcal A}\left(A,B\right)\otimes TA} & T\left( \underline{\mathcal A}\left(B,C\right)\otimes \underline{\mathcal A}\left(A,B\right)\otimes A\right)  \ar[d]^{T\left( \underline{\mathcal A}\left(B,C\right) \otimes ev\right) } \\
\underline{\mathcal B}\left(TB,TC\right)\otimes \underline{\mathcal A}\left(A,B\right)\otimes TA \ar[d]_{\underline{\mathcal B}\left(TB,TC\right)\otimes \sigma} & T\left(\underline{\mathcal A}\left(B,C\right)\otimes B\right) \ar[dd]^{T(ev)} \\
\underline{\mathcal B}\left(TB,TC\right)\otimes T\left( \underline{\mathcal A}\left(A,B\right)\otimes A\right) \ar[d]_{\underline{\mathcal B}\left(TB,TC\right)\otimes T(ev)} \\ 
\underline{\mathcal B}\left(TB,TC\right)\otimes TB \ar[r]^{ev} & TC
}
$$
Then, by bifunctoriality of tensor product
$$
\xymatrix @C=2.7cm @R=2cm {\relax
\underline{\mathcal A}\left(B,C\right)\otimes \underline{\mathcal A}\left(A,B\right)\otimes TA \ar[r]^{\sigma_{\underline{\mathcal A}\left(B,C\right)\otimes \underline{\mathcal A}\left(A,B\right),A}} \ar[d]_{\underline{\mathcal A}\left(B,C\right)\otimes \sigma} & T\left( \underline{\mathcal A}\left(B,C\right)\otimes \underline{\mathcal A}\left(A,B\right)\otimes A\right)  \ar[d]^{T\left( \underline{\mathcal A}\left(B,C\right) \otimes ev\right) } \\
\underline{\mathcal A}\left(B,C\right)\otimes T\left( \underline{\mathcal A}\left(A,B\right)\otimes A\right) \ar[d]_{\underline{\mathcal A}\left(B,C\right)\otimes T(ev)} & T\left(\underline{\mathcal A}\left(B,C\right)\otimes B\right) \ar[dd]^{T(ev)} \\
\underline{\mathcal A}\left(B,C\right)\otimes TB \ar[d]_{\varphi_{B,C} \otimes TB} \\ 
\underline{\mathcal B}\left(TB,TC\right)\otimes TB \ar[r]^{ev} & TC
}
$$
Using the equality $ev \circ \varphi\otimes 1=T(ev) \circ \sigma$ once again
$$
\xymatrix @C=2.7cm @R=2cm {\relax
\underline{\mathcal A}\left(B,C\right)\otimes \underline{\mathcal A}\left(A,B\right)\otimes TA \ar[r]^{\sigma_{\underline{\mathcal A}\left(B,C\right)\otimes \underline{\mathcal A}\left(A,B\right),A}} \ar[d]_{\underline{\mathcal A}\left(B,C\right)\otimes \sigma} & T\left( \underline{\mathcal A}\left(B,C\right)\otimes \underline{\mathcal A}\left(A,B\right)\otimes A\right)  \ar[d]^{T\left( \underline{\mathcal A}\left(B,C\right) \otimes ev\right) } \\
\underline{\mathcal A}\left(B,C\right)\otimes T\left( \underline{\mathcal A}\left(A,B\right)\otimes A\right) \ar[d]_{\underline{\mathcal A}\left(B,C\right)\otimes T(ev)} & T\left(\underline{\mathcal A}\left(B,C\right)\otimes B\right) \ar[dd]^{T(ev)} \\
\underline{\mathcal A}\left(B,C\right)\otimes TB \ar[d]_{\sigma_{\underline{\mathcal A}\left(B,C\right),B}} \\ 
T\left( \underline{\mathcal A}\left(B,C\right)\otimes B\right) \ar[r]^{T(ev)} & TC
}
$$
Finally, by naturality of $\sigma$
$$
\xymatrix @C=2.7cm @R=2cm {\relax
\underline{\mathcal A}\left(B,C\right)\otimes \underline{\mathcal A}\left(A,B\right)\otimes TA \ar[r]^{\sigma_{\underline{\mathcal A}\left(B,C\right)\otimes \underline{\mathcal A}\left(A,B\right),A}} \ar[d]_{\underline{\mathcal A}\left(B,C\right)\otimes \sigma} & T\left( \underline{\mathcal A}\left(B,C\right)\otimes \underline{\mathcal A}\left(A,B\right)\otimes A\right)  \ar[dd]^{T\left( \underline{\mathcal A}\left(B,C\right) \otimes ev\right) } \\
\underline{\mathcal A}\left(B,C\right)\otimes T\left( \underline{\mathcal A}\left(A,B\right)\otimes A\right) \ar[d]_{\sigma} \\
T\left( \underline{\mathcal A}\left(B,C\right)\otimes \underline{\mathcal A}\left(A,B\right)\otimes A\right) \ar[r]^{T\left( \underline{\mathcal A}\left(B,C\right)\otimes ev\right)} & T\left(\underline{\mathcal A}\left(B,C\right)\otimes B\right)
}
$$
Therefore, we have the following diagram
$$
\shorthandoff{;:!?}
\xymatrix @!0 @C=4cm @R=2cm{\relax 
 \underline{\mathcal A}\left(B,C\right) \otimes \underline{\mathcal A}\left(A,B\right) \otimes TA \ar[rr]^{\underline{\mathcal A}\left(B,C\right)\otimes\sigma} \ar[rd]_{\sigma} && \underline{\mathcal A}\left(B,C\right)\otimes T\left(\underline{\mathcal A}\left(A,B\right)\otimes A\right) \ar[ld]^{\sigma} \\ 
& T\left(\underline{\mathcal A}\left(B,C\right)\otimes \underline{\mathcal A}\left(A,B\right)\otimes A\right) 
}
$$
which clearly commutes by the associativity axiom of $\sigma$.

Hence $\varphi_{A,B}$ satisfies the composition axiom.

It remains to prove that $\textit{(b)}$ implies $\textit{(a)}$. First, we prove that $\sigma_{A,B}$ satisfies the unit axiom

$$
\shorthandoff{;:!?}
\xymatrix @!0 @C=3cm @R=1.5cm{\relax 
 I\otimes TA \ar[rr]^{l_{TA}} \ar[rd]_{\sigma_{I,A}} &&  TA  \\ 
& T\left(I\otimes A\right) \ar[ru]_{T\left( l_{A}\right) }
}
$$ 
By definition of $\sigma$
$$
\shorthandoff{;:!?}
\xymatrix @C=3.5cm @R=1.5cm{\relax 
 I\otimes TA \ar[d]_{\gamma_{A}\otimes A} \ar[r]^{l_{TA}} & TA \\
 \underline{\mathcal A}\left(A,I\otimes A\right) \otimes TA \ar[d]_{\varphi_{A,I\otimes A}\otimes TA} \\
 \underline{\mathcal B}\left(TA,T\left( I\otimes A\right)\right)  \otimes TA \ar[r]^{ev} &  T\left(I\otimes A\right) \ar[uu]_{T\left( l_{A}\right) } 
}
$$
By naturality of $ev$
$$
\shorthandoff{;:!?}
\xymatrix @C=3.5cm @R=1.5cm{\relax 
 I\otimes TA \ar[d]_{\gamma_{A}\otimes A} \ar[r]^{l_{TA}} & TA \\
 \underline{\mathcal A}\left(A,I\otimes A\right) \otimes TA \ar[d]_{\varphi_{A,I\otimes A}\otimes TA} \\
 \underline{\mathcal B}\left(TA,T\left( I\otimes A\right)\right)  \otimes TA \ar[r]^{\underline{\mathcal B}\left(TA,T\left(l_{A}\right)\right) \otimes TA} &   \underline{\mathcal B}\left(TA,TA \right)\otimes TA  \ar[uu]_{ev} 
}
$$
Then by adjunction
$$
\shorthandoff{;:!?}
\xymatrix @C=4.5cm @R=1.5cm{\relax 
 I \ar[d]_{\gamma_{A}} \ar[r]^{j_{TA}} & \underline{\mathcal B}\left(TA,TA\right) \\
 \underline{\mathcal A}\left(A,I\otimes A\right) \ar[d]_{\varphi_{A,I\otimes A}} \\
 \underline{\mathcal B}\left(TA,T\left( I\otimes A\right)\right) \ar[r]^{\underline{\mathcal B}\left(TA,T\left(l_{A}\right)\right)} & \underline{\mathcal B}\left(TA,TA \right)  \ar@{=}[uu] 
}
$$
By naturality of $\varphi$
$$
\shorthandoff{;:!?}
\xymatrix @C=4.8cm @R=2cm{\relax 
 I \ar[d]_{\gamma_{A}} \ar[r]^{j_{TA}} & \underline{\mathcal B}\left(TA,TA\right) \\
 \underline{\mathcal A}\left(A,I\otimes A\right) \ar[r]^{ \underline{\mathcal A}\left(A,l_{A}\right)} & \underline{\mathcal A}\left(A,A \right) \ar[u]_{\varphi_{A,A}}
}
$$
Furthermore, we have $\hat{j}_{A}=l_{A}=ev\circ j_{A}\otimes 1$. Hence this diagram is equivalent to
$$
\shorthandoff{;:!?}
\xymatrix @C=4.5cm @R=1.5cm{\relax 
 I \ar[d]_{\gamma_{A}} \ar[r]^{j_{TA}} & \underline{\mathcal B}\left(TA,TA\right) \\
 \underline{\mathcal A}\left(A,I\otimes A\right) \ar[d]_{ \underline{\mathcal A}\left(A,j_{A}\otimes A\right)} \\
 \underline{\mathcal A}\left(A, \underline{\mathcal A}\left(A,A \right)\otimes A \right) \ar[r]^{ \underline{\mathcal A}\left(A,ev \right)} & \underline{\mathcal A}\left(A,A \right) \ar[uu]_{\varphi_{A,A}}
}
$$
Finally, by naturality of $\gamma$
$$
\shorthandoff{;:!?}
\xymatrix @C=4.5cm @R=1.5cm{\relax 
 I \ar[d]_{j_{A}} \ar[r]^{j_{TA}} & \underline{\mathcal B}\left(TA,TA\right) \\
 \underline{\mathcal A}\left(A,A\right) \ar[d]_{\gamma_{A}} \\
 \underline{\mathcal A}\left(A, \underline{\mathcal A}\left(A,A \right)\otimes A \right) \ar[r]^{ \underline{\mathcal A}\left(A,ev \right)} & \underline{\mathcal A}\left(A,A \right) \ar[uu]_{\varphi_{A,A}}
}
$$
But we have $ \underline{\mathcal A}\left(A,ev\right)\circ \gamma_{A}=1_{\underline{\mathcal A}\left(A,A\right)}$.
Hence, we have the following diagram
$$
\shorthandoff{;:!?}
\xymatrix @!0 @C=3cm @R=1.5cm{\relax 
I \ar[rd]_{j_A} \ar[rr]^{j_{TA}}  && \underline{\mathcal B}\left(TA,TA\right)  \\ 
& \underline{\mathcal A}\left(A,A\right)  \ar[ru]_{\varphi_{A,A}} 
}
$$
which clearly commutes by the unit axiom of $\varphi$.


Finally, we prove that $\sigma_{A,B}$ satisfies the associativity axiom
$$
\shorthandoff{;:!?}
\xymatrix @!0 @C=3cm @R=1.5cm{\relax 
 X\otimes Y\otimes TA \ar[rr]^{\sigma_{X\otimes Y,A}} \ar[rd]_{1\otimes\sigma_{Y,A}} && T\left(X\otimes Y\otimes A\right) \\ 
& X\otimes T\left(Y\otimes A\right) \ar[ru]_{\sigma_{X,Y\otimes A}}  
}
$$
Extending definitions
$$
\shorthandoff{;:!?}
\xymatrix @C=2.5cm @R=1.5cm{\relax 
X\otimes Y\otimes TA \ar[r]^{\sigma_{X\otimes Y,A}} \ar[d]_{1\otimes\gamma_{A}\otimes 1} & T\left(X\otimes Y\otimes A\right) \\
X\otimes \underline{\mathcal A}\left(A,Y\otimes A\right)\otimes TA \ar[d]_{1\otimes \varphi\otimes 1} \\
X\otimes \underline{\mathcal B}\left(TA,T\left( Y\otimes A\right)\right) \otimes TA \ar[d]_{1\otimes ev} & \underline{\mathcal B}\left(T\left( Y\otimes A\right),T\left(X\otimes Y\otimes A\right)\right) \otimes T\left( Y\otimes A\right) \ar[uu]_{ev} \\
X\otimes T\left( Y\otimes A\right) \ar[r]^{\gamma_{Y\otimes A}\otimes 1} & \underline{\mathcal A}\left( Y\otimes A,X\otimes Y\otimes A\right) \otimes T\left( Y\otimes A\right) \ar[u]_{\varphi\otimes 1} 
}
$$
By bifunctoriality of tensor product
$$
\shorthandoff{;:!?}
\xymatrix @!0 @C=4.5cm @R=2.5cm{\relax 
X\otimes Y\otimes TA \ar[rr]^{\sigma_{X\otimes Y,A}} \ar[d]_{\gamma_{Y\otimes A}\otimes\gamma_{A}\otimes 1} && T\left(X\otimes Y\otimes A\right) \\
\underline{\mathcal A}\left( Y\otimes A,X\otimes Y\otimes A\right) \otimes \underline{\mathcal A}\left( A,Y\otimes A\right) \otimes TA \ar[rd]_{\varphi \otimes \varphi \otimes 1} && \underline{\mathcal B}\left(T\left( Y\otimes A\right),T\left(X\otimes Y\otimes A\right)\right) \otimes T\left( Y\otimes A\right) \ar[u]_{ev} \\
 & \underline{\mathcal B}\left( T(Y\otimes A),T(X\otimes Y\otimes A)\right) \otimes \underline{\mathcal B}\left( TA,T\left( Y\otimes A\right)\right)\otimes TA \ar[ru]_{1\otimes ev} 
}
$$
Furthermore, by definition of composition $ev\circ1\otimes ev=ev\circ c\otimes 1$, we have
$$
\shorthandoff{;:!?}
\xymatrix @!0 @C=4.5cm @R=2.5cm{\relax 
X\otimes Y\otimes TA \ar[rr]^{\sigma_{X\otimes Y,A}} \ar[d]_{\gamma_{Y\otimes A}\otimes\gamma_{A}\otimes 1} && T\left(X\otimes Y\otimes A\right) \\
\underline{\mathcal A}\left( Y\otimes A,X\otimes Y\otimes A\right) \otimes \underline{\mathcal A}\left( A,Y\otimes A\right) \otimes TA \ar[rd]_{\varphi \otimes \varphi \otimes 1} && \underline{\mathcal B}\left(TA,T\left(X\otimes Y\otimes A\right)\right) \otimes TA \ar[u]_{ev} \\
 & \underline{\mathcal B}\left( T(Y\otimes A),T(X\otimes Y\otimes A)\right) \otimes \underline{\mathcal B}\left( TA,T\left( Y\otimes A\right)\right)\otimes TA \ar[ru]_{c\otimes 1} 
}
$$
By the composition axiom of $\varphi$, we have $c\circ \varphi\otimes \varphi=\varphi \circ c$
$$
\shorthandoff{;:!?}
\xymatrix @!0 @C=4.5cm @R=2.5cm{\relax 
X\otimes Y\otimes TA \ar[rr]^{\sigma_{X\otimes Y,A}} \ar[d]_{\gamma_{Y\otimes A}\otimes\gamma_{A}\otimes 1} && T\left(X\otimes Y\otimes A\right) \\
\underline{\mathcal A}\left( Y\otimes A,X\otimes Y\otimes A\right) \otimes \underline{\mathcal A}\left( A,Y\otimes A\right) \otimes TA \ar[rd]_{c \otimes 1} && \underline{\mathcal B}\left(TA,T\left(X\otimes Y\otimes A\right)\right) \otimes TA \ar[u]_{ev} \\
 & \underline{\mathcal A}\left( A,X\otimes Y\otimes A\right) \otimes TA \ar[ru]_{\varphi \otimes 1} 
}
$$
To conclude this proof, we need the following result
\begin{lem}
\label{lemgam}
Let $\mathcal E$ be a closed symmetric monoidal category. Let $\mathcal A$ and $\mathcal B$ be two categories tensored over $\mathcal E$.
The following diagram is commutative
$$
\shorthandoff{;:!?}
\xymatrix @!0 @C=3.5cm @R=2cm{\relax 
 A\otimes B \ar[rr]^{\gamma_{C}} \ar[rd]_{\gamma_{B\otimes C}\otimes\gamma_{C}} && \underline{\mathcal A}\left( C,A\otimes B\otimes C\right)\\ 
& \underline{\mathcal A}\left( B\otimes C,A\otimes B\otimes C\right)\otimes \underline{\mathcal A}\left( C,B\otimes C\right) \ar[ru]_{c}
}
$$
\end{lem}
\begin{proof} 

By adjunction, this diagram is equivalent to
$$
\shorthandoff{;:!?}
\xymatrix @C=3cm @R=2cm{\relax 
 A\otimes B\otimes C \ar@{=}[r] \ar[d]_{\gamma_{B\otimes C}\otimes\gamma_{C}\otimes 1} & A\otimes B\otimes C \\ 
 \underline{\mathcal A}\left( B\otimes C,A\otimes B\otimes C\right)\otimes \underline{\mathcal A}\left( C,B\otimes C\right)\otimes C \ar[r]^{c\otimes 1} & \underline{\mathcal A}\left( C,A\otimes B\otimes C\right) \ar[u]_{ev}
}
$$ 
By definition of composition $ev\circ c\otimes 1=ev\circ 1\otimes ev$, we have
$$
\shorthandoff{;:!?}
\xymatrix @C=2.8cm @R=2cm{\relax 
 A\otimes B\otimes C \ar[d]_{\gamma_{B\otimes C}\otimes 1\otimes 1} &  \underline{\mathcal A}\left( B\otimes C,A\otimes B\otimes C\right)\otimes B\otimes C \ar[l]_{ev}\\ 
 \underline{\mathcal A}\left( B\otimes C,A\otimes B\otimes C\right)\otimes B \otimes C \ar[r]^{1\otimes \gamma_{C}\otimes 1} & \underline{\mathcal A}\left( B\otimes C,A\otimes B\otimes C\right)\otimes \underline{\mathcal A}\left( C,B\otimes C\right)\otimes C \ar[u]_{1\otimes ev}
}
$$
But we have $\hat{\gamma}=ev\circ \gamma \otimes 1=1$, hence this diagram clearly commutes. 
\end{proof}
Going back to the proof of Proposition \ref{propfonf-e}.
By Lemma \ref{lemgam}, the diagram is equivalent to
$$
\shorthandoff{;:!?}
\xymatrix @C=4cm @R=2cm{\relax 
X\otimes Y\otimes TA \ar[r]^{\sigma_{X\otimes Y,A}} \ar[d]_{\gamma_{A}\otimes 1} & T\left(X\otimes Y\otimes A\right) \\
\underline{\mathcal A}\left( A,X\otimes Y\otimes A\right) \otimes TA \ar[r]^{\varphi\otimes 1} & \underline{\mathcal B}\left(TA,T\left(X\otimes Y\otimes A\right)\right) \otimes TA \ar[u]_{ev}
}
$$
which commutes by definition of $\sigma_{X,A}$.
\end{proof}

The following lemma provides the correspondence between the composite strength and the composite enrichment.

\begin{lem}
Let $\mathcal E$ be a closed symmetric monoidal category. Given three $\mathcal E$-categories $\mathcal A$, $\mathcal B$ and $\mathcal C$ tensored over $\mathcal E$, let $T_{1}:\mathcal A\rightarrow \mathcal B$ and $T_{2}:\mathcal B\rightarrow \mathcal C$ be two functors. 

The composite strength 
$$\sigma_{2,1}:A\otimes T_{2}T_{1}B\xrightarrow{\sigma_{2}} T_{2}\left( A\otimes T_{1}B\right) \xrightarrow{T_{2}\left( \sigma_{1}\right) } T_{2}T_{1}\left( A\otimes B\right)$$
corresponds to the composite enrichment
$$\varphi_{2}\circ \varphi_{1}:\underline{\mathcal A}\left( A,B\right)\xrightarrow{\varphi_{1}}\underline{\mathcal B}\left( T_{1}A,T_{1}B\right)\xrightarrow{\varphi_{2}}\underline{\mathcal C}\left( T_{2}T_{1}A,T_{2}T_{1}B\right)$$
\end{lem}
\begin{proof}
To prove that the two composits correspond mutually, we need to see that $$\widehat{\varphi_{2}\circ\varphi_{1}}=T_{2}T_{1}(ev)\circ T_{2}(\sigma_{1})\circ \sigma_{2} $$
We want to prove the commutativity of the following diagram
$$
\shorthandoff{;:!?}
\xymatrix @C=3.5cm @R=2cm{\relax
\underline{\mathcal A}\left(A,B\right)\otimes T_{2}T_{1}A \ar[r]^{\varphi_{1}\otimes T_{2}T_{1}A} \ar[d]_{\sigma_{2}} & \underline{\mathcal B}\left(T_{1}A,T_{1}B\right)\otimes T_{2}T_{1}A \ar[d]^{\varphi_{2}\otimes T_{2}T_{1}A} \\
T_{2}\left(\underline{\mathcal A}\left(A,B\right)\otimes T_{1}A\right) \ar[d]_{T_{2}(\sigma_{1})} & \underline{\mathcal C}\left(T_{2}T_{1}A,T_{2}T_{1}B\right)\otimes T_{2}T_{1}A \ar[d]^{ev}\\
T_{2}T_{1}\left(\underline{\mathcal A}\left(A,B\right)\otimes A\right) \ar[r]^{T_{2}T_{1}(ev)} & T_{2}T_{1}B 
}
$$
But we have $ev\circ\varphi_{2}\otimes 1=\hat{\varphi}_{2}$ and $\hat{\varphi}_{2}=T_{2}(ev)\circ \sigma_{2}$, hence
$$
\shorthandoff{;:!?}
\xymatrix @C=3.5cm @R=2cm{\relax
\underline{\mathcal A}\left(A,B\right)\otimes T_{2}T_{1}A \ar[r]^{\varphi_{1}\otimes T_{2}T_{1}A} \ar[d]_{\sigma_{2}} & \underline{\mathcal B}\left(T_{1}A,T_{1}B\right)\otimes T_{2}T_{1}A \ar[d]^{\sigma_{2}} \\
T_{2}\left(\underline{\mathcal A}\left(A,B\right)\otimes T_{1}A\right) \ar[r]^{T_{2}\left(\varphi_{1}\otimes T_{2}T_{1}A\right)} \ar[d]_{T_{2}(\sigma_{1})} & T_{2}\left( \underline{\mathcal B}\left(T_{1}A,T_{1}B\right)\otimes T_{1}A\right)  \ar[d]^{T_{2}(ev)}\\
T_{2}T_{1}\left(\underline{\mathcal A}\left(A,B\right)\otimes A\right) \ar[r]^{T_{2}T_{1}(ev)} & T_{2}T_{1}B 
}
$$

The upper diagram commutes by naturality of $\sigma_{2}$. Observing the lower diagram, on the one side we have $$T_{2}(ev)\circ T_{2}(\varphi_{1}\otimes 1)=T_{2}(ev\circ \varphi_{1}\otimes 1)=T_{2}(\hat{\varphi}_{1})$$ and on the other $$T_{2}T_{1}(ev)\circ T_{2}(\sigma_{1})= T_{2}(T_{1}(ev)\circ \sigma_{1})=T_{2}(\hat{\varphi}_{1})$$ Hence, the lower diagram also commutes.  
\end{proof}

\subsection{Strong and enriched natural transformations}

Once we have an equivalence between strong and enriched functors (Proposition \ref{propfonf-e}), the natural way is to relate strong and enriched natural transformations.


\begin{prop}\label{proptranf-e}
Let $\mathcal E$ be a closed symmetric monoidal category. Let $\mathcal A$ and $\mathcal B$ be two $\mathcal E$-categories and $T_{1},T_{2}:\mathcal A\rightarrow\mathcal B$ two $\mathcal E$-functors. 
Given a natural transformation $\Psi:T_{1}\Rightarrow T_{2}$, the following conditions are equivalent:
\begin{enumerate}
 \item The natural transformation $\Psi$ extends to a strong natural transformation;
 \item The natural transformation $\Psi$ extends to a $\mathcal E$-natural transformation.
\end{enumerate}
\end{prop}
\begin{proof}
We need to prove that the following diagram commutes
$$
\shorthandoff{;:!?}
\xymatrix @C=3.5cm @R=2cm{\relax
\underline{\mathcal A}\left(A,B\right) \ar[r]^{\varphi_{1}} \ar[d]_{\varphi_{2}} & \underline{\mathcal B}\left(T_{1}A,T_{1}B\right) \ar[d]^{\underline{\mathcal B}\left(T_{1}A,\Psi_{B}\right)} \\
\underline{\mathcal B}\left(T_{2}A,T_{2}B\right) \ar[r]^{\underline{\mathcal B}\left(\Psi_{A},T_{2}B\right)} & \underline{\mathcal B}\left(T_{1}A,T_{2}B\right) 
}
$$
By adjunction, this diagram is equivalent to 
$$
\shorthandoff{;:!?}
\xymatrix @C=4cm @R=2cm{\relax
\underline{\mathcal A}\left(A,B\right)\otimes T_{1}A \ar[r]^{\hat{\varphi}_{1}} \ar[d]_{\underline{\mathcal A}\left(A,B\right)\otimes \Psi_{A}} & T_{1}B \ar[d]^{\Psi_{B}} \\
\underline{\mathcal A}\left(A,B\right)\otimes T_{2}A \ar[r]^{\hat{\varphi}_{2}} & T_{2}B
}
$$
But we have $\hat{\varphi}=T(ev)\circ \sigma$, hence
$$
\shorthandoff{;:!?}
\xymatrix @C=2.5cm @R=2cm{\relax
\underline{\mathcal A}\left(A,B\right)\otimes T_{1}A \ar[r]^{\sigma_{1}} \ar[d]_{\underline{\mathcal A}\left(A,B\right)\otimes \Psi_{A}} & T_{1}\left(\underline{\mathcal A}\left(A,B\right)\otimes A\right) \ar[r]^{T_{1}(ev)} & T_{1}B \ar[d]^{\Psi_{B}} \\
\underline{\mathcal A}\left(A,B\right)\otimes T_{2}A \ar[r]^{\sigma_{2}} & T_{2}\left(\underline{\mathcal A}\left(A,B\right)\otimes A\right) \ar[r]^{T_{2}(ev)} & T_{2}B
}
$$
By naturality of $\Psi$
$$
\shorthandoff{;:!?}
\xymatrix @C=2cm @R=2cm{\relax
\underline{\mathcal A}\left(A,B\right)\otimes T_{1}A \ar[r]^{\sigma_{1}} \ar[d]_{\underline{\mathcal A}\left(A,B\right)\otimes \Psi_{A}} & T_{1}\left(\underline{\mathcal A}\left(A,B\right)\otimes A\right) \ar[r]^{\Psi} & T_{2}\left(\underline{\mathcal A}\left(A,B\right)\otimes A\right) \ar[d]^{T_{2}(ev)} \\
\underline{\mathcal A}\left(A,B\right)\otimes T_{2}A \ar[r]^{\sigma_{2}} & T_{2}\left(\underline{\mathcal A}\left(A,B\right)\otimes A\right) \ar[r]^{T_{2}(ev)} & T_{2}B
}
$$
But this diagram is equivalent to the following diagram
$$
\shorthandoff{;:!?}
\xymatrix @C=4cm @R=2cm{\relax
\underline{\mathcal A}\left(A,B\right)\otimes T_{1}A \ar[r]^{\sigma_{1}} \ar[d]_{\underline{\mathcal A}\left(A,B\right)\otimes \Psi_{A}} & T_{1}\left(\underline{\mathcal A}\left(A,B\right)\otimes A\right) \ar[d]_{\Psi} \\
\underline{\mathcal A}\left(A,B\right)\otimes T_{2}A \ar[r]^{\sigma_{2}} & T_{2}\left(\underline{\mathcal A}\left(A,B\right)\otimes A\right)
}
$$
which is commutative, since $\Psi$ is a strong natural transformation.

It remains to prove that $\Psi$ extends to a strong natural transformation. We have to prove that the following diagram commutes
$$
\shorthandoff{;:!?}
\xymatrix @C=3cm @R=2cm{\relax
X\otimes T_{1}A \ar[r]^{\sigma_{1}} \ar[d]_{X\otimes \Psi_{A}} & T_{1}\left(X\otimes A\right) \ar[d]^{\Psi_{X\otimes A}}\\
X\otimes T_{2}A \ar[r]^{\sigma_{2}} & T_{2}\left(X\otimes A\right) 
}
$$
By adjunction, this diagram is equivalent to 
$$
\shorthandoff{;:!?}
\xymatrix @C=4cm @R=2cm{\relax
X \ar[r]^{\hat{\sigma}_{1}} \ar[d]_{\hat{\sigma}_{2}} & \underline{\mathcal B}\left(T_{1}A,T_{1}\left(X\otimes A\right)\right) \ar[d]^{\underline{\mathcal B}\left(T_{1}A,\Psi_{X\otimes A}\right)}\\
\underline{\mathcal B}\left(T_{2}A,T_{2}\left(X\otimes A\right)\right) \ar[r]^{\underline{\mathcal B}\left(\Psi_{A},T_{2}\left(X\otimes A\right)\right)} & \underline{\mathcal B}\left(T_{1}A,T_{2}\left(X\otimes A\right)\right) 
}
$$
But we have $\hat{\sigma}=\varphi \circ \gamma$, hence
$$
\shorthandoff{;:!?}
\xymatrix @C=1.7cm @R=1.4cm{\relax
X \ar[r]^{\gamma_{A}}  \ar[d]_{\gamma_{A}} & \underline{\mathcal A}\left(A,X\otimes A\right) \ar[r]^{\varphi_{1}} & \underline{\mathcal B}\left(T_{1}A,T_{1}\left(X\otimes A\right)\right) \ar[dd]^{\underline{\mathcal B}\left(T_{1}A,\Psi_{X\otimes A}\right)}\\
\underline{\mathcal A}\left(A,X\otimes A\right) \ar[d]_{\varphi_{2}}\\
\underline{\mathcal B}\left(T_{2}A,T_{2}\left(X\otimes A\right)\right) \ar[rr]^{\underline{\mathcal B}\left(\Psi_{A},T_{2}\left(X\otimes A\right)\right)} && \underline{\mathcal B}\left(T_{1}A,T_{2}\left(X\otimes A\right)\right) 
}
$$
and this diagram is equivalent to the following one
$$
\shorthandoff{;:!?}
\xymatrix @C=4cm @R=2cm{\relax
\underline{\mathcal A}\left(A,X\otimes A\right) \ar[r]^{\varphi_{1}} \ar[d]_{\varphi_{2}} & \underline{\mathcal B}\left(T_{1}A,T_{1}\left(X\otimes A\right)\right) \ar[d]^{\underline{\mathcal B}\left(T_{1}A,\Psi_{X\otimes A}\right)}\\
\underline{\mathcal B}\left(T_{2}A,T_{2}\left(X\otimes A\right)\right) \ar[r]^{\underline{\mathcal B}\left(\Psi_{A},T_{2}\left(X\otimes A\right)\right)} & \underline{\mathcal B}\left(T_{1}A,T_{2}\left(X\otimes A\right)\right) 
}
$$
which is commutative, since $\Psi$ is a $\mathcal E$-natural transformation.
\end{proof}


\begin{lem}\label{lemcomptranf-e}

Let $\mathcal E$ be a closed symmetric monoidal category and let $\mathcal A$ and $\mathcal B$ be two tensored $\mathcal E$-categories. Given three $\mathcal E$-functors $T_{1},T_{2}, T_{3}:\mathcal A\rightarrow\mathcal B$, let $\alpha: T_{1}\Rightarrow T_{2}$ and $\beta: T_{2}\Rightarrow T_{3}$ be two $\mathcal E$-natural transformations.

Then the composite $\beta\circ\alpha$ of two $\mathcal E$-natural transformations coressponds to the composite $\beta\circ\alpha$ of two strong natural transformations.
\end{lem}
\begin{proof}
Consider three $\mathcal E$-functors $\left(T_{1},\varphi_{1}\right)$, $\left(T_{2},\varphi_{2}\right)$, $\left(T_{1},\varphi_{3}\right)$. Using Definition \ref{deftnatcomp}, the composite $\beta\circ\alpha:T_{1}\Rightarrow T_{3}$ of two $\mathcal E$-natural transformations is a $\mathcal E$-natural transformation given by the following commutative diagram
\begin{diagram}{}\label{diagcomptne}
$$
\shorthandoff{;:!?}
\xymatrix @C=4cm @R=2cm{\relax
\underline{\mathcal A}\left(A,B\right) \ar[r]^{\varphi_{1}} \ar[d]_{\varphi_{3}} & \underline{\mathcal B}\left(T_{1}A,T_{1}B\right) \ar[d]^{\underline{\mathcal B}\left(\alpha^{-1}_{A},\alpha_{B}\right)} \\
\underline{\mathcal B}\left(T_{3}A,T_{3}B\right) \ar[r]^{\underline{\mathcal B}\left(\beta_{A},\beta_{B}^{-1}\right)} & \underline{\mathcal B}\left(T_{2}A,T_{2}B\right) 
}
$$
\end{diagram}
Similarly, consider three strong functors $\left(T_{1},\sigma_{1}\right)$, $\left(T_{2},\sigma_{2}\right)$, $\left(T_{1},\sigma_{3}\right)$. Using Definition \ref{defstnatt} the composite $\beta\circ\alpha:T_{1}\Rightarrow T_{3}$ of two strong natural transformations is a strong natural transformation given by the following commutative diagram
\begin{diagram}{}\label{diagcomptnf}
$$
\shorthandoff{;:!?}
\xymatrix @C=4cm @R=2cm{\relax
X\otimes T_{1}A \ar[r]^{\sigma_{1}} \ar[d]_{X\otimes\left(\beta\circ\alpha\right)_{A}} & T_{1}\left(X\otimes A\right) \ar[d]^{\left(\beta\circ\alpha\right)_{X\otimes A}}\\
X\otimes T_{3}A \ar[r]^{\sigma_{3}} & T_{3}\left(X\otimes A\right) 
}
$$
\end{diagram}

We prove that Diagram \ref{diagcomptnf} commutes using the commutativity of Diagram \ref{diagcomptne}.

By adjunction, Diagram \ref{diagcomptnf} is equivalent to 
$$
\shorthandoff{;:!?}
\xymatrix @C=3.5cm @R=2.3cm{\relax
X \ar[r]^{\hat{\sigma}_{1}} \ar[d]_{\hat{\sigma}_{3}} & \underline{\mathcal B}\left(T_{1}A,T_{1}\left(X\otimes A\right)\right) \ar[d]^{\underline{\mathcal B}\left(1,\left(\beta\circ\alpha\right)_{X\otimes A}\right)}\\
\underline{\mathcal B}\left(T_{3}A,T_{3}\left(X\otimes A\right)\right) \ar[r]^{\underline{\mathcal B}\left(\left(\beta\circ\alpha\right)_{A},1\right)} & \underline{\mathcal B}\left(T_{1}A,T_{3}\left(X\otimes A\right)\right) 
}
$$ 
But we have $\hat{\sigma}=\varphi\circ\gamma$, hence
$$
\shorthandoff{;:!?}
\xymatrix @C=2cm @R=1.3cm{\relax
X \ar[r]^{\gamma_{A}} \ar[d]_{\gamma_{A}} & \underline{\mathcal A}\left(A,X\otimes A\right) \ar[r]^{\varphi_{1}} & \underline{\mathcal B}\left(T_{1}A,T_{1}\left(X\otimes A\right)\right) \ar[dd]^{\underline{\mathcal B}\left(1,\left(\beta\circ\alpha\right)_{X\otimes A}\right)}\\
\underline{\mathcal A}\left(A,X\otimes A\right) \ar[d]_{\varphi_{3}}\\
\underline{\mathcal B}\left(T_{3}A,T_{3}\left(X\otimes A\right)\right) \ar[rr]^{\underline{\mathcal B}\left(\left(\beta\circ\alpha\right)_{A},1\right)} && \underline{\mathcal B}\left(T_{1}A,T_{3}\left(X\otimes A\right)\right) 
}
$$
But this diagram is equivalent to the following one
$$
\shorthandoff{;:!?}
\xymatrix @C=3cm @R=1.8cm{\relax
\underline{\mathcal A}\left(A,X\otimes A\right) \ar[r]^{\varphi_{1}} \ar[d]_{\varphi_{3}} & \underline{\mathcal B}\left(T_{1}A,T_{1}\left(X\otimes A\right)\right) \ar[d]^{\underline{\mathcal B}\left(1,\alpha_{X\otimes A}\right)}\\
\underline{\mathcal B}\left(T_{3}A,T_{3}\left(X\otimes A\right)\right) \ar[d]_{\underline{\mathcal B}\left(\beta_{A},1\right)} & \underline{\mathcal B}\left(T_{1}A,T_{2}\left(X\otimes A\right)\right) \ar[d]^{\underline{\mathcal B}\left(1,\beta_{X\otimes A}\right)}\\ 
\underline{\mathcal B}\left(T_{2}A,T_{3}\left(X\otimes A\right)\right) \ar[r]^{\underline{\mathcal B}\left(\alpha_{A},1\right)} & \underline{\mathcal B}\left(T_{1}A,T_{3}\left(X\otimes A\right)\right) 
}
$$ 
Since $\alpha$ and $\beta$ are $\mathcal E$-natural transformations, we have $\underline{\mathcal B}\left(1,\alpha \right)\circ \varphi_{1}= \underline{\mathcal B}\left(\alpha,1 \right)\circ \varphi_{2}$ and $\underline{\mathcal B}\left(\beta,1\right)\circ \varphi_{3}= \underline{\mathcal B}\left(1,\beta \right)\circ \varphi_{2}$.
Hence
$$
\shorthandoff{;:!?}
\xymatrix @C=3cm @R=1.8cm{\relax
\underline{\mathcal A}\left(A,X\otimes A\right) \ar[r]^{\varphi_{2}} \ar[d]_{\varphi_{2}} & \underline{\mathcal B}\left(T_{2}A,T_{2}\left(X\otimes A\right)\right) \ar[d]^{\underline{\mathcal B}\left(\alpha_{A},1\right)}\\
\underline{\mathcal B}\left(T_{2}A,T_{2}\left(X\otimes A\right)\right) \ar[d]_{\underline{\mathcal B}\left(1,\beta_{X\otimes A}\right)} & \underline{\mathcal B}\left(T_{1}A,T_{2}\left(X\otimes A\right)\right) \ar[d]^{\underline{\mathcal B}\left(1,\beta_{X\otimes A}\right)}\\ 
\underline{\mathcal B}\left(T_{2}A,T_{3}\left(X\otimes A\right)\right) \ar[r]^{\underline{\mathcal B}\left(\alpha_{A},1\right)} & \underline{\mathcal B}\left(T_{1}A,T_{3}\left(X\otimes A\right)\right) 
}
$$ 
But this diagram is equivalent to the following one
$$
\shorthandoff{;:!?}
\xymatrix @C=3cm @R=2.2cm{\relax
\underline{\mathcal B}\left(T_{2}A,T_{2}\left(X\otimes A\right)\right) \ar[r]^{\underline{\mathcal B}\left(\alpha_{A},1\right)} \ar[d]_{\underline{\mathcal B}\left(1,\beta_{X\otimes A}\right)} & \underline{\mathcal B}\left(T_{1}A,T_{2}\left(X\otimes A\right)\right) \ar[d]^{\underline{\mathcal B}\left(1,\beta_{X\otimes A}\right)}\\ 
\underline{\mathcal B}\left(T_{2}A,T_{3}\left(X\otimes A\right)\right) \ar[r]^{\underline{\mathcal B}\left(\alpha_{A},1\right)} & \underline{\mathcal B}\left(T_{1}A,T_{3}\left(X\otimes A\right)\right) 
}
$$
which clearly commutes.
\end{proof}

\section{The canonical 2-isomorphism between $\mathbf{StrongCat}$ and $\mathcal E\mathbf{-Cat}$}\label{S2CTh}

In this section, we state the main theorem of this chapter, Theorem \ref{th2-cat} which shows that there is a 2-isomorphism between the 2-category of strong functors and natural transformations and the 2-category of $\mathcal E$-functors and $\mathcal E$-natural transformations. It is followed by Corollary \ref{cormons-e} which relates strong and enriched monads.

In order to do that we will need somme background on the language of 2-categories.

\subsection{2-Categories}

While a category has just objects and arrows, the category of categories and functors can be provided with additional devices, namely natural transformations between functors. This leads to the richer notion of a 2-category, where besides objects and arrows one gives also 2-cells between the arrows. Analogously, there are corresponding enrichments of the notions of functor, natural transformation, adjoint functors and so on.

In this subsection, we give basic definition and examples of 2-categories \cite{KelBCECT,KelRE2C}.

\begin{definition} \label{def2cat}
A 2-category $\mathcal C$ consists in giving:
\begin{enumerate}
\item A class $\mathcal C_{0}$, whose objects are called 0-cells;
\item For every pair of objects $\left(X,Y\right) $ in $\mathcal C$, a category $\underline{\mathcal C}\left(X,Y\right)$ whose 
\begin{itemize}
 \item Objects $f:X\rightarrow Y$ are called 1-cells;
 \item Morphisms $\alpha:f\Rightarrow g$ are called 2-cells; their composition $\circ_{1}$ is called vertical composition.
\end{itemize}
\item For every triple of objects $\left(X,Y,Z \right) $ in $\mathcal C$, a functor of horizontal composition
 $$\circ_{0}:\underline{\mathcal C}\left(Y,Z\right)\times \underline{\mathcal C}\left(X,Y\right)\rightarrow \underline{\mathcal C}\left(X,Z\right) ; $$ 
\item For every object $X$ in $\mathcal C$, a functor of horizontal identity
 $$j_{X}: I\rightarrow \underline{\mathcal C}\left(X,X\right) $$
such that the following coherence diagrams commute
\begin{diagram}{}
$$
\xymatrix {\relax
\left( \underline{\mathcal C}\left(Z,T\right) \times \underline{\mathcal C}\left(Y,Z\right)\right)\times \underline{\mathcal C}\left(X,Y\right) \ar[rr]^{\alpha} \ar[d]_{\circ_{0}\times 1} && \underline{\mathcal C}\left(Z,T\right)\times\left( \underline{\mathcal C}\left(Y,Z\right)\times \underline{\mathcal C}\left(X,Y\right)\right) \ar[d]^{1\times \circ_{0}}\\
\underline{\mathcal C}\left(Y,T\right)\times\underline{\mathcal C}\left(X,Y\right) \ar[rd]_{\circ_{0}} && \underline{\mathcal C}\left(Z,T\right)\times\underline{\mathcal C}\left(X,Z\right) \ar[ld]^{\circ_{0}} \\
& \underline{\mathcal C}\left(X,T\right) 
}
$$
\end{diagram}
\begin{diagram}{}
$$
\xymatrix @C=2cm @R=1.5cm{\relax  
I\times \underline{\mathcal C}\left(X,Y\right) \ar[r]^l \ar[d]_{j_{Y}\times 1} & \underline{\mathcal C}\left(X,Y\right) \ar[d]_{Id} & \underline{\mathcal C}\left(X,Y\right)\times I \ar[l]_r \ar[d]^{1\times j_{X}}\\
\underline{\mathcal C}\left(Y,Y\right)\times \underline{\mathcal C}\left(X,Y\right) \ar[r]^{\circ_{0}} & \underline{\mathcal C}\left(X,Y\right) & \underline{\mathcal C}\left(X,Y\right)\times \underline{\mathcal C}\left(X,X\right) \ar[l]_{\circ_{0}} 
}
$$
\end{diagram}
\end{enumerate}
\end{definition}

One can see a 2-category as an enriched category over a category of categories, called $\mathbf{Cat}$.

\begin{ex}~\\
\begin{itemize}
 \item One can extend $\mathbf{Cat}$ to a 2-category $\mathbf{Cat}_{2}$ where 0-cells are categories, 1-cells are functors and 2-cells are natural transformations between functors;
 \item One can also define the 2-category $\mathcal E\mathbf{-Cat}_{2}$ of $\mathcal E$-categories, $\mathcal E$-functors and $\mathcal E$-natural transformations. When $\mathcal E=Set$ the 2-category $\mathcal E\mathbf{-Cat}_{2}$ coincides with $\mathbf{Cat}_{2}$.
\end{itemize}
\end{ex}



\begin{rem}
There exists a 2-functor $\varphi$:
$$\varphi: \mathcal E-Cat_{2}\rightarrow Cat_{2}$$
$$\underline{\mathcal C}\longrightarrow \underline{\mathcal C}_{0}  $$
$$F:\underline{\mathcal C}\rightarrow \underline{\mathcal D}\longrightarrow F_{0}:\underline{\mathcal C}_{0}\rightarrow \underline{\mathcal D}_{0}  $$
$$\alpha:F\Rightarrow G \longrightarrow \alpha_{0}:F_{0}\Rightarrow G_{0}  $$
The functor $\varphi$ associates to a $\mathcal E$-category $\underline{\mathcal C}$ its underlying category $\underline{\mathcal C}_{0}$, to a $\mathcal E$-functor F its underlying functor $F_{0}$, to a $\mathcal E$-natural tranformation $\alpha$ a natural transformation $\alpha_{0}$ between underlying functors.  
\end{rem} 

We have the corresponding definition of a monad in a 2-category.
\begin{definition}
Let $\mathcal E$ be a 2-category. 
A monad $(T, \mu, \eta )$ on an object $\mathcal C$ of $\mathcal E$ consists in giving:
\begin{enumerate}
\item 1-cell $T:\mathcal C\rightarrow \mathcal C$;
\item 2-cell $\eta:Id_{\mathcal C}\Rightarrow T$ called the unit of a monad; 
\item 2-cell $\mu:TT\Rightarrow T$ called the multiplication of a monad;

such that the following diagrams commute
\end{enumerate}
\begin{diagram}{}
$$
\shorthandoff{;:!?}
\xymatrix @!0 @C=2.5cm @R=2.3cm{\relax
T\ar[r]^{\eta T} \ar@{=}[rd] & TT \ar[d]^{\mu} & T \ar[l]_{T\eta} \ar@{=}[ld] \\
& T
}
$$
\end{diagram} 
\begin{diagram}{}
$$
\shorthandoff{;:!?}
\xymatrix @!0 @C=3cm @R=2.3cm{\relax
TTT \ar[r]^{\mu T} \ar[d]_{T\mu} & TT \ar[d]^{\mu} \\
TT \ar[r]^{\mu} & T
}
$$
\end{diagram}
\end{definition}

\begin{rem}
A monad in $\mathcal E=\mathbf{Cat_{2}}$ gives the usual definition of a monad in a category.
\end{rem}

\subsection{2-Isomorphism between $\mathbf{StrongCat}$ and $\mathcal E\mathbf{-Cat}$}

\begin{lem}
Strong functors and strong natural transformations constitute the 1-cells and 2-cells of a 2-category of $\mathcal E$-tensored categories, written \textbf{StrongCat}.
\end{lem}

\begin{proof}
Since by composing two strong functors we acquire another strong functor (see Definition \ref{deffcomp}), the properties involving 1-cells have been checked.
It is clear that strong natural transformations compose. Axioms of Definition \ref{def2cat} are satisfied by arguments similar to arguments used in the case of a 2-category of categories, functors and natural transformations.   
\end{proof}

We state the theorem: 
\begin{theoreme}\label{th2-cat}
The following 2-categories of tensored $\mathcal E$-categories are 2-isomorphic:
\begin{enumerate}
 \item The 2-category of strong functors and strong natural transformations of tensored $\mathcal E$-categories, $\mathbf{StrongCat}$; 
 \item The 2-category of $\mathcal E$-functors and $\mathcal E$-natural transformations of tensored $\mathcal E$-categories, $\mathcal E\mathbf{-Cat}$.
\end{enumerate}
\end{theoreme}

\begin{proof}
By Proposition \ref{propfonf-e}, strong functors $(T,\sigma)$ are exactly $\mathcal E$-functors $(T,\varphi)$.
By Proposition \ref{proptranf-e}, strong natural transformations are exactly $\mathcal E$-natural transformations.
By Lemma \ref{lemcomptranf-e}, the composite of strong natural transformations corresponds to the composite of $\mathcal E$-natural transformations.
Therefore, the 2-category $\mathbf{StrongCat}$ is 2-isomorphic to the 2-category $\mathcal E\mathbf{-Cat}$. 
\end{proof}


\begin{cor}\label{cormons-e} 
Given a monad $\left(T,\mu,\eta\right)$ in a category $\mathcal C$, the following conditions are equivalent:
\begin{enumerate}
 \item The monad $\left(T,\mu,\eta\right)$ extends to a strong monad $\left(T,\mu,\eta,\sigma \right)$;
 \item The monad $\left(T,\mu,\eta\right)$ extends to a $\mathcal E$-monad $\left(T,\mu,\eta,\varphi\right)$.   
\end{enumerate}
\end{cor}

\begin{rem}
A strong monad $\left(T,\mu,\eta,\sigma \right)$ is a monad $\left(T,\mu,\eta \right)$ in the 2-category $\mathbf{StrongCat}$ of strong functors and strong natural transformations of tensored $\mathcal E$-categories.
\end{rem}

\begin{proof}
Since by Theorem \ref{th2-cat}, the 2-category $\mathbf{StrongCat}$ is 2-isomorphic to the 2-category $\mathcal E\mathbf{-Cat}$, strong monads  $\left(T,\mu,\eta,\sigma \right)$ correspond to $\mathcal E$-monads $\left(T,\mu,\eta,\varphi\right)$. 
\end{proof}

\begin{rem}
We give a detail on the passage from a $\mathcal E$-monad $\left(T,\mu,\eta,\varphi\right)$ to a strong monad $\left(T,\mu,\eta,\sigma \right)$. 

By Theorem \ref{th2-cat}, $\mathcal E$-functors and $\mathcal E$-natural transformations of $\mathcal E$-tensored categories correspond to strong functors and strong natural transformations of $\mathcal E$-tensored categories.

In particular, a functor $T:\mathcal E\rightarrow \mathcal E$ is a $\mathcal E$-functor if and only if $(T,\sigma)$ is a strong functor i.e. T has a tensorial strength $\sigma_{A,B}:A\otimes TB\rightarrow T(A\otimes B)$ and the axioms given by the commutativity of the following diagrams are satisfied: 
\begin{diagram}{Unit condition for $\sigma$}
$$
\shorthandoff{;:!?}
\xymatrix @!0 @C=3cm @R=1.5cm{\relax 
 I\otimes TA \ar[rr]^{\sigma_{I,A}} \ar[rd]_{l_{TA}} && T\left(I\otimes A\right) \ar[ld]^{T\left( l_{A}\right) } \\ 
& TA 
}
$$ 
\end{diagram} 
\begin{diagram}{Associativity condition for $\sigma$}
$$
\shorthandoff{;:!?}
\xymatrix @!0 @C=3cm @R=1.5cm{\relax 
 A\otimes B\otimes TC \ar[rr]^{A\otimes\sigma_{B,C}} \ar[rd]_{\sigma_{A\otimes B,C}} && A\otimes T\left(B\otimes C\right) \ar[ld]^{\sigma_{A,B\otimes C}} \\ 
& T\left(A\otimes B\otimes C\right) 
}
$$
\end{diagram}
A natural transformation $\eta:Id_{\mathcal E}\rightarrow T$ is a $\mathcal E$-natural transformation if and only if $\eta$ is a strong natural transformation, with $T_{1}=Id_{\mathcal E}, T_{2}=T, \sigma_{1}=Id$ and $\sigma_{2}=\sigma $ i.e. $\eta$ is given by the commutative diagram
\begin{diagram}{Strong naturality of $\eta$}
$$
\shorthandoff{;:!?}
\xymatrix @!0 @C=3cm @R=1.5cm{\relax 
 A\otimes TB \ar[rr]^{\sigma_{A,B}} && T\left(A\otimes B\right)  \\ 
& A \otimes B \ar[lu]^{1\otimes \eta_{B}} \ar[ru]_{\eta_{A\otimes B}} 
}
$$ 
\end{diagram}
Finally, the natural transformation $\mu:TT\rightarrow T$ is a $\mathcal E$-natural transformation if and only if $\mu$ is a strong natural transformation, with $T_{1}=TT, T_{2}=T, \sigma_{1}=T_{2}(\sigma_{1})\circ\sigma_{2}$ and $\sigma_{2}=\sigma$ i.e. $\mu$ is given by the commutative diagram
\begin{diagram}{Strong naturality of $\mu$}
$$
\shorthandoff{;:!?}
\xymatrix @!0 @C=4.7cm @R=2.3cm{\relax
A \otimes T^{2}B \ar[r]^{\sigma_{A,TB}} \ar[d]_{1\otimes \mu_{B}} & T\left(A\otimes TB\right) \ar[r]^{T\left(\sigma_{A,B}\right)} & T^{2}\left(A\otimes B\right) \ar[d]^{\mu_{A\otimes B}}\\
A\otimes TB \ar[rr]^{\sigma_{A,B}} && T\left(A\otimes B\right) 
}
$$
\end{diagram}
Therefore, we obtain exactly the diagrams defining a strong monad $\left(T,\mu,\eta, \sigma \right)$.
\end{rem}

There exists a notion of a monoidal monad $\left(T,\mu,\eta,\lambda \right) $ in a category $\mathcal E$. It is a monad in $\mathcal E$ which is equipped with a natural transformation $$\lambda_{A,B}: TA\otimes TB\rightarrow T\left(A\otimes B\right)$$ and satisfying some corresponding axioms. These axioms simply translate the fact that we require for T to be a monoidal functor (to satisfy some unit and associativity conditions) and for $\mu$ and $\eta$ to be monoidal transformations. 

Furthermore, one can require for a monoidal monad $\left(T,\mu,\eta,\lambda \right) $ to be symmetric. 
\begin{definition}
A monoidal monad $\left(T,\mu,\eta,\lambda\right)$ in a symmetric monoidal category $\mathcal E$ is called symmetric if the following diagram
\begin{diagram}{Symmetry condition}
$$
\shorthandoff{;:!?}
\xymatrix @C=3cm @R=2cm{\relax
TA\otimes TB \ar[r]^{\lambda_{A,B}} \ar[d]_{a_{TA,TB}} & T\left(A\otimes B\right) \ar[d]^{T\left(a_{A,B}\right)} \\
TB\otimes TA \ar[r]^{\lambda_{B,A}} & T\left(B\otimes A\right) 
}
$$
\end{diagram}
commutes.
\end{definition}

\begin{rem}
A symmetric monoidal monad $\left(T,\mu,\eta, \lambda \right)$ is a monad $\left(T,\mu,\eta \right)$ in the 2-category $\mathbf{SMonCat}$ of symmetric monoidal functors and symmetric monoidal transformations of tensored $\mathcal E$-categories.
\end{rem}

There is a correspondence between symmetric monoidal monads and some particular strong monads. These particular strong monads are the commutative strong monads.

\begin{definition}
A strong monad $\left(T,\mu,\eta,\sigma \right)$ in a symmetric monoidal category $\mathcal E$ is called commutative if the following diagram
\begin{diagram}{Commutativity condition}
$$
\shorthandoff{;:!?}
\xymatrix @!0 @C=4.7cm @R=2.3cm{\relax
TA\otimes TB \ar[r]^{\sigma'_{A,TB}} \ar[d]_{\sigma_{TA,B}} & T\left(A\otimes TB\right) \ar[r]^{T\left(\sigma_{A,B}\right)} & T^{2}\left(A\otimes B\right) \ar[d]^{\mu_{A\otimes B}}\\
T\left( TA\otimes B\right) \ar[r]^{T(\sigma'_{A,B})} & T^{2}\left(A\otimes B\right) \ar[r]^{\mu_{A\otimes B}} & T\left(A\otimes B\right) 
}
$$
\end{diagram}
commutes, where the dual tensorial strength $\sigma'_{A,B}: TA\otimes B\longrightarrow T\left(A\otimes B\right)$ is induced by the symmetry of $\mathcal E$ (see Remark following Definition \ref{defStMon}).  
\end{definition}

\begin{prop}{(\cite{KockSF,GoubLRMT})}\label{propSTMMon}
For a symmetric monoidal category $\mathcal E$, there is a canonical bijection between commutative strong monads and symmetric monoidal monads.
\end{prop}

\section{Strong monads and Day convolution}\label{SSmDayC}

In this section we use Day convolution to construct a strong monad on the category $\mathcal E^{\mathcal A}$ of functors $\mathcal A\rightarrow \mathcal E$ where $\mathcal A$ and $\mathcal E$ are supposed to be symmetric monoidal.

\begin{lem}\label{lemMonAlgT}
Consider categories $\mathcal E$ and $\mathcal A$ where $\mathcal A$ is supposed to be small. Let T be a monad on $\mathcal E$.
Then $T$ induces by postcomposition a monad on the category $\mathcal E^{\mathcal A}$ of functors $\mathcal A\rightarrow \mathcal E$.
In fact, we have the following diagram
\begin{diagram}{}
$$
\shorthandoff{;:!?}
\xymatrix @!0 @C=2cm @R=2cm{\relax
\left( \mathcal E^{T}\right)^{\mathcal A} \ar[rd]_{\left(U_{T}\right)^{\mathcal A} } \ar[rr]^{\omega} && \left( \mathcal E^{\mathcal A}\right)^{T \circ -} \ar[ld]^{U} \\
 & \mathcal E^{\mathcal A} 
}
$$
\end{diagram}
in which the functor $\omega$ induces an isomorphism between the category of functors $\mathcal A\rightarrow \mathcal E^{T}$ and the category of $\left( T \circ -\right) $-algebras on $\mathcal E^{\mathcal A}$.  
\end{lem}
\begin{proof}
By definition, a $\left( T \circ -\right) $-algebra on $\mathcal E^{\mathcal A}$ consists in giving:
\begin{itemize}
 \item A functor $F:\mathcal A\rightarrow \mathcal E$;
 \item A natural transformation $\xi_{F}:TF\Rightarrow F$.
 \end{itemize}
satisfying the usual axioms of an algebra over a monad. This amounts to a lifting $\overline{F}$ of the functor $F$ to the category of $T$-algebras such that the following diagram
\begin{diagram}{}
$$
\shorthandoff{;:!?}
\xymatrix @!0 @C=2cm @R=2cm{\relax
& \mathcal E^{T} \ar[rd]^{U_{T}}\\
\mathcal A \ar[rr]^{F} \ar[ru]^{\overline{F}}  && \mathcal E 
}
$$
\end{diagram}
commutes.
\end{proof}



\begin{definition}
Consider two functors $F:\mathcal A\rightarrow \mathcal B$ and $G:\mathcal A \rightarrow \mathcal C$ . 

The left Kan extension of G along F,
\begin{diagram}{}
$$
\shorthandoff{;:!?}
\xymatrix @!0 @C=2cm @R=2cm{\relax 
\mathcal{A} \ar[rr]^{F} \ar[rd]_{G} && \mathcal{B} \ar@{.>}[ld]^{K} \\ 
& \mathcal{C}   
}
$$
\end{diagram}
is a pair $\left(K,\alpha \right)$ where $K:\mathcal B\rightarrow \mathcal C$ is a functor and $\alpha: G\Rightarrow K\circ F$ is a natural transformation satisfying the universal property: 

if $\left(H,\beta \right)$ is another pair where $H:\mathcal B\rightarrow \mathcal C$ is a functor and $\beta: G\Rightarrow H\circ F$ is a natural transformation, there exists a unique natural transformations $\gamma:K\Rightarrow H$ such that the following diagram commutes 
\begin{diagram}{}
$$
\shorthandoff{;:!?}
\xymatrix @!0 @C=2cm @R=2cm{\relax 
G \ar[rr]^{\beta} \ar[rd]_{\alpha} && HF \\ 
& KF  \ar@{.>}[ru]_{\gamma_{F}} 
}
$$
\end{diagram}
We write $Lan_{F}G$ for the left Kan extension. 
There is a dual notion of right Kan extension, written $Ran_{F}G$.
\end{definition}

\begin{definition}\label{defDayconv}
For symmetric monoidal categories $\mathcal A$ and $\mathcal E$, with $\mathcal A$ small, the functor category $\mathcal E^{\mathcal A}$ carries a symmetric monoidal structure $\mathcal E^\mathcal A \times \mathcal E^\mathcal A\to \mathcal E^\mathcal A:(F,G) \rightarrow F\otimes G$, where $F\otimes G$ is defined by \emph{Day convolution}, i.e. by the following diagram 

\begin{diagram}{}
$$
\shorthandoff{;:!?}
\xymatrix @!0 @C=4.5cm @R=3cm{\relax
\mathcal A \times \mathcal A \ar[rr]^{F \times G} \ar[d]_{\otimes_{\mathcal A} } \ar@{}[rd]|{\Downarrow} \ar[rrd]^{F\hat{\otimes} G} && \mathcal E \times \mathcal E \ar[d]^{\otimes_{\mathcal E} } \\
\mathcal A  \ar[rr]^{F \otimes G} && \mathcal E 
}
$$
\end{diagram}
in which $F\otimes G$ is the left Kan extension of $F\hat{\otimes}G$ along $\otimes_\mathcal A$.
\end{definition}

\begin{prop}\label{PropStDay}
Let $\mathcal A$ and $\mathcal E$ be symmetric monoidal categories with $\mathcal A$ small.
Then any strong monad $T$ on $\mathcal E$ induces by postcomposition a strong monad $T\!\circ\!-$ on $\mathcal E^\mathcal A$ with respect to the symmetric monoidal structure given by Day convolution.
\end{prop}

\begin{proof}
In order to prove the existence of a tensorial strength for the induced monad $T\!\circ\!-$ on $\mathcal E^\mathcal A$ we use the canonical natural transformation $F\hat{\otimes}TG \Rightarrow T(F\hat{\otimes} G)$ which is induced by the tensorial strength of $T$. Indeed, the following diagram


\begin{diagram}{}\label{diagstmEA}
$$
\shorthandoff{;:!?}
\xymatrix @!0 @C=4.5cm @R=3cm{\relax
\mathcal A \times \mathcal A \ar[r]^{F \times G}  \ar[rd]^{F \times TG} \ar@/_6pc/[rrd]_{F \hat{\otimes} TG} \ar@/^4pc/[rr]^{F \hat{\otimes} G} & \mathcal E \times \mathcal E \ar[d]^{Id \times T } \ar@{}[rd]|{\Uparrow} \ar[r]^{\otimes \mathcal E} & \mathcal E \ar[d]^{T}\\
& \mathcal E \times \mathcal E  \ar[r]^{\otimes \mathcal E} & \mathcal E 
}
$$
\end{diagram}
yields the required natural transformation $F\otimes TG \Rightarrow T(F\otimes G)$ by functoriality and unversality of left Kan extensions.
\end{proof}

The following corollary is a direct consequence of Proposition \ref{PropStDay}.

\begin{cor}
Let $\mathcal E$ be a symmetric monoidal category and $\mathcal A$ a small symmetric monoidal category. 
If T is a strong monad on $\mathcal E$, then the category of  $\left( T \circ -\right) $-algebras on $\mathcal E^{\mathcal A}$ is enriched over $\mathcal E^{\mathcal A}$.
\end{cor}

\begin{proof}
By Proposition \ref{PropStDay}, the strong monad $T$ induces a strong monad $\left( T \circ -\right) $ on $\mathcal E^{\mathcal A}$.
By Corollary \ref{cormons-e}, there is an equivalence between strong and enriched monads.
Therefore, by Proposition \ref{propenr}, the category $\left( \mathcal E^{\mathcal A}\right)^{T \circ -}$ is enriched over $\mathcal E^{\mathcal A}$.
\end{proof}

\chapter{Morita theory in enriched context}\label{ChMthEcon}

This chapter is devoted to homotopical Morita theorem which provides, in the context of a strong monad $T$ on a monoidal model category $\mathcal E$, a Quillen equivalence between the category of $T$-algebras and the category of modules over a monoid in $\mathcal E$.
It is organized as follows: in Section \ref{SMT}, we give some classical results in Morita theory. 
In Section \ref{STaM}, we discover that in the context of a strong monad, the category of $T$-algebras is enriched, tensored and cotensored over $\mathcal E$. Moreover, the enrichment of $\mathcal Alg_{T}$, provides $T\left(I\right)$ with the structure of a monoid in $\mathcal E$. 
In Section \ref{SMth}, we state homotopical Morita theorem, Theorem \ref{thMain}. In Section \ref{GspGrGth}, we provide basic definitions and results of the category of $\Gamma$-spaces. 
In Section \ref{sAsmPpax}, we prove that the Bousfiled-Friedlander's stable model structure of $\Gamma$-spaces is a monoidal model category. Finally, in Section \ref{sRecResSS} we apply our homotopical Morita theorem for $\mathcal E=\Gamma$-spaces and we recover a theorem of Stefan Schwede \cite{SchSHAT}.

\section{A glimpse at classical Morita theory}\label{SMT}

In 1958, Morita \cite{MorDM} established a result on the equivalences between module categories via the functor $Hom\left(P,-\right)$, with P a projective generator. Gabriel \cite{GDCA} proved that there is an equivalence between an abelian category and a module category over a ring via the functor $Hom\left(P,-\right)$. Few years later, Gabriel and Popescu \cite{GPCCA} studied the localization of module categories over a ring.

The term Morita theory is now used for results concerning equivalences of various kinds of module categories. In particular, Morita theory was studied for categories of regular algebras, for derived categories, for stable model categories.

In this section, we provide some results of classical Morita theory in the most general direction.

The following result, known as Gabriel's theorem gives a characterization of module categories among abelian categories. 

\begin{theoreme}{(Gabriel,\cite{GDCA}) }\label{thG}
 Let $Ab$ be an abelian category and P an object in $Ab$ with a morphism of rings: $\varphi: R\rightarrow Hom\left(P,P\right)$.
The following are equivalent:
\begin{enumerate}
 \item The functor 
\begin{eqnarray*}
F:Ab & \rightarrow & Mod_{R} \\
X & \rightarrow & Hom_{Ab}\left(P,X\right)
\end{eqnarray*}
is an equivalence between categories;
 \item The object P is a projective generator in $Ab$ and $\varphi$ is an isomorphism i.e. $R\cong Hom\left(P,P\right)$.
\end{enumerate}
Furthermore, all equivalences between $Ab$ and $Mod_{R}$ are of this form up to isomorphism. 
\end{theoreme}

The following corollary is a direct consequence of Gabriel's theorem.
\begin{cor}
Let R and S be two rings and $P_{S}$ an object in $Mod_{S}$ with the morphism of rings: $\varphi: R\rightarrow Hom_{S}\left(P_{S},P_{S}\right)$.
The following are equivalent:
\begin{enumerate}
 \item The functor
\begin{eqnarray*}
F:Mod_{S} & \rightarrow & Mod_{R}\\
X & \rightarrow & Hom_{S}\left(P_{S},X\right)
\end{eqnarray*}
is an equivalence between categories;
 \item The object P is a projective generator in $Mod_{S}$ and $\varphi$ is an isomorphism i.e. $R\cong Hom_{S}\left(P_{S},P_{S}\right)$.
\end{enumerate}
Furthermore, all equivalences between $Mod_{S}$ and $Mod_{R}$ are of this form up to isomorphism.
\end{cor}



\section{Algebras over strong monads}\label{STaM}


\subsection{Enrichment}

Using the concept of enriched monad $T$, we prove that the category of $T$-algebras is enriched (cf. \cite{BungeRFCCA, LintRFSAR}).

\begin{prop}\label{propenr}
Let $\mathcal E$ be a closed symmetric monoidal category with equalizers and $\left(T,\mu,\eta, \varphi \right)$ an enriched monad over $\mathcal E$. 
Then the category $Alg_{T}$ of $T$-algebras is canonically enriched over $\mathcal E$. 
Moreover, the $\mathcal E$-object $\underline{\mathcal Alg}_{T}\left( X,Y\right)$ is given by the equalizer
\begin{diagram}{}
$$
\xymatrix {\relax
&&& \underline{\mathcal E}\left( TX,TY\right) \ar[rd]^{\underline{\mathcal E}\left( TX,\xi_{Y}\right)} \\
\underline{\mathcal Alg}_{T}\left( X,Y\right) \ar@{^{(}->}[rr]^{i} && \underline{\mathcal E}\left( X,Y\right) \ar[ru]^{\varphi_{T}} \ar[rr]_{\underline{\mathcal E}\left(\xi_{X},Y\right)} && \underline{\mathcal E}\left( TX,Y\right) \\
& I \ar@{-->}[lu]^{\exists!\psi} \ar[ru]_{f}  \\
}
$$
\end{diagram}
\end{prop}
\begin{proof}
From the morphism of $T$-algebras, given by the following commutative diagram 
$$
\shorthandoff{;:!?}
\xymatrix @!0 @C=3cm @R=2.3cm{\relax
TX \ar[r]^{T\left(f \right)} \ar[d]_{\xi_{X}} & TY \ar[d]^{\xi_{Y}} \\
X \ar[r]^{f} & Y
}
$$
we obtain the equalizer
$$
\shorthandoff{;:!?}
\xymatrix @!0 @C=3cm @R=2.3cm{\relax
&&& \underline{\mathcal E}\left( TX,TY\right) \ar[rd]^{\underline{\mathcal E}\left( TX,\xi_{Y}\right)} \\
\underline{\mathcal Alg}_{T}\left( X,Y\right) \ar@{^{(}->}[rr]^{i} && \underline{\mathcal E}\left( X,Y\right) \ar[ru]^{\varphi_{T}} \ar[rr]_{\underline{\mathcal E}\left(\xi_{X},Y\right)} && \underline{\mathcal E}\left( TX,Y\right) \\
& I \ar@{-->}[lu]^{\exists!\psi} \ar[ru]_{f}  \\
}
$$
So, we define the $\mathcal E$-object $\underline{\mathcal Alg}_{T}\left( X,Y\right)$ to be the equalizer of the above diagram and we will note $U_{T}:\underline{\mathcal Alg}_{T}\left( X,Y\right)\rightarrow \underline{\mathcal E}\left( X,Y\right)$ to be the $\mathcal E$-morphism canonically associated with this equalizer.

First, we prove that there exists a unit morphism i.e. we prove that there are two morphisms which equalize the unit morphism $j_{X}:I\rightarrow \underline{\mathcal E}\left( X,X\right)$ in order to obtain the universal property of the equalizer.

We need to prove that the following diagram commutes
$$
\shorthandoff{;:!?}
\xymatrix @!0 @C=4cm @R=2cm{\relax
& \underline{\mathcal E}\left( X,X\right) \ar[r]^{\varphi_{T}} & \underline{\mathcal E}\left( TX,TX\right) \ar[dd]^{\underline{\mathcal E}\left( TX,\xi_{Y}\right)} \\
I \ar[rd]_{j_{X}} \ar[ru]^{j_{X}}\\
& \underline{\mathcal E}\left( X,X\right) \ar[r]^{\underline{\mathcal E}\left(\xi_{X}, X \right)} & \underline{\mathcal E}\left( TX,X\right)\\ 
}
$$
Since $\left(T,\mu,\eta, \varphi \right)$ is an enriched monad, T is a $\mathcal E$-functor, so $\varphi_{T}\circ j_{X}= j_{TX}$. Hence we have
$$
\shorthandoff{;:!?}
\xymatrix @!0 @C=5cm @R=4cm{\relax
I \ar[r]^{j_{TX}} \ar[d]_{j_{X}} & \underline{\mathcal E}\left( TX,TX\right) \ar[d]^{\underline{\mathcal E}\left( TX,\xi_{Y}\right)} \\
\underline{\mathcal E}\left( X,X\right) \ar[r]^{\underline{\mathcal E}\left(\xi_{X}, X \right)} & \underline{\mathcal E}\left( TX,X\right)\\ 
}
$$
But $\underline{\mathcal E}\left(I, \underline{\mathcal E}\left(TX,X\right)\right)\cong \underline{\mathcal E}\left(I\otimes TX,X\right)\cong \underline{\mathcal E}\left(TX,X\right)=\xi_{X}$ and we have that $\underline{\mathcal E}\left(TX,\xi_{X}\right) \circ j_{TX}=\xi_{X}$ and $\underline{\mathcal E}\left(\xi_{X},X\right) \circ j_{X}=\xi_{X}$.
Therefore, by the universal property of the equalizer there exists one unique morphism $$j^{\mathcal Alg_{T}}_{X}:I\rightarrow \underline{\mathcal Alg}_{T}\left( X,X\right)$$ such that $j_{X}= U_{T}\circ j^{\mathcal Alg_{T}}_{X}$.

Second, we prove the existence of the composition morphism i.e. we prove that there are two morphisms which equalize $$\underline{\mathcal Alg}_{T}\left(Y,Z\right)\otimes \underline{\mathcal Alg}_{T}\left( X,Y\right) \xrightarrow{U_{T}\otimes U_{T}} \underline{\mathcal E}\left( Y,Z\right)\otimes \underline{\mathcal E}\left( X,Y\right)\xrightarrow{c_{\mathcal E}} \underline{\mathcal E}\left( X,Z\right)$$ in order to obtain the universal property of the equalizer
\small{
\begin{diagram}{}
$$
\shorthandoff{;:!?}
\xymatrix  @C=2.6pc @R=3.3pc {\relax
&&& \underline{\mathcal E}\left( TX,TZ\right) \ar[rd]^{\underline{\mathcal E}\left( TX,\xi_{Z}\right)} \\
\underline{\mathcal Alg}_{T}\left( Y,Z\right)\otimes \underline{\mathcal Alg}_{T}\left( X,Y\right) \ar@{-->}[rd]_{\exists!c_{\mathcal Alg_{T}}} \ar[r]^{U_{T}\otimes U_{T}} & \underline{\mathcal E}\left( Y,Z\right)\otimes\underline{\mathcal E}\left( X,Y\right) \ar[r]^{c_{\mathcal E}} & \underline{\mathcal E}\left( X,Z\right) \ar[ru]^{\varphi_{T}} \ar[rr]_{\underline{\mathcal E}\left(\xi_{X},Z\right)} && \underline{\mathcal E}\left( TX,Z\right) \\
& \underline{\mathcal Alg}_{T}\left( X,Z\right) \ar[ru]_{U_{T}}  \\
}
$$
\end{diagram}
}
We need to prove that the following diagram commutes
$$
\shorthandoff{;:!?}
\xymatrix @!0 @C=7cm @R=2cm{\relax
\underline{\mathcal Alg}_{T}\left(Y,Z\right)\otimes \underline{\mathcal Alg}_{T}\left( X,Y\right) \ar[r]^{U_{T}\otimes U_{T}} \ar[d]_{U_{T}\otimes U_{T}} & \underline{\mathcal E}\left( Y,Z\right)\otimes \underline{\mathcal E}\left( X,Y\right) \ar[d]^{c} \\
\underline{\mathcal E}\left( Y,Z\right)\otimes \underline{\mathcal E}\left( X,Y\right) \ar[d]_{c} & \underline{\mathcal E}\left(X,Z\right) \ar[dd]^{\underline{\mathcal E}\left(\xi_{X},Z\right)} \\
\underline{\mathcal E}\left(X,Z\right) \ar[d]_{\varphi_{T}} \\
\underline{\mathcal E}\left( TX,TZ\right) \ar[r]^{\underline{\mathcal E}\left( TX,\xi_{Z}\right)} & \underline{\mathcal E}\left( TX,Z\right)
}
$$
Since $T$ is a $\mathcal E$-functor, by the composition axiom $c\circ \varphi_{T}\otimes \varphi_{T}=\varphi_{T}\circ c$ we have 
$$
\shorthandoff{;:!?}
\xymatrix @!0 @C=7cm @R=2cm{\relax
\underline{\mathcal Alg}_{T}\left(Y,Z\right)\otimes \underline{\mathcal Alg}_{T}\left( X,Y\right) \ar[r]^{U_{T}\otimes U_{T}} \ar[d]_{U_{T}\otimes U_{T}} & \underline{\mathcal E}\left( Y,Z\right)\otimes \underline{\mathcal E}\left( X,Y\right) \ar[d]^{c} \\
\underline{\mathcal E}\left( Y,Z\right)\otimes \underline{\mathcal E}\left( X,Y\right) \ar[d]_{\varphi_{T}\otimes \varphi_{T}} & \underline{\mathcal E}\left( X,Z\right) \ar[dd]^{\underline{\mathcal E}\left(\xi_{X},Z\right)} \\
\underline{\mathcal E}\left( TY,TZ\right)\otimes \underline{\mathcal E}\left( TX,TY\right) \ar[d]_{c}\\
\underline{\mathcal E}\left( TX,TZ\right) \ar[r]^{\underline{\mathcal E}\left( TX,\xi_{Z}\right)} & \underline{\mathcal E}\left( TX,Z\right)
}
$$
By naturality of the composition
$$
\shorthandoff{;:!?}
\xymatrix @!0 @C=7cm @R=2cm{\relax
\underline{\mathcal Alg}_{T}\left(Y,Z\right)\otimes \underline{\mathcal Alg}_{T}\left( X,Y\right) \ar[r]^{U_{T}\otimes U_{T}} \ar[d]_{U_{T}\otimes U_{T}} & \underline{\mathcal E}\left( Y,Z\right)\otimes \underline{\mathcal E}\left( X,Y\right) \ar[d]^{1\otimes \underline{\mathcal E}\left(\xi_{X}, Y \right) } \\
\underline{\mathcal E}\left( Y,Z\right)\otimes \underline{\mathcal E}\left( X,Y\right) \ar[d]_{\varphi_{T}\otimes \varphi_{T}} & \underline{\mathcal E}\left( Y,Z\right)\otimes \underline{\mathcal E}\left( TX,Y\right) \ar[dd]^{c} \\
\underline{\mathcal E}\left( TY,TZ\right)\otimes \underline{\mathcal E}\left( TX,TY\right) \ar[d]_{\underline{\mathcal E}\left(TY, \xi_{Z} \right)\otimes 1} \\
\underline{\mathcal E}\left( TY,Z\right)\otimes \underline{\mathcal E}\left( TX,TY\right) \ar[r]^{c} & \underline{\mathcal E}\left( TX,Z\right)
}
$$
Since $\underline{\mathcal Alg}_{T}\left(X,Y\right)$ is the equalizer of $\left(\underline{\mathcal E}\left(TX, \xi_{Y} \right)\circ\varphi_{T}, \underline{\mathcal E}\left(\xi_{X}, Y \right)\right) $, we have $$\underline{\mathcal E}\left(TX, \xi_{Y} \right)\circ\varphi_{T}\circ U_{T}= \underline{\mathcal E}\left(\xi_{X}, Y \right)\circ U_{T}$$ 
Tensoring with $U_{T}:\underline{\mathcal Alg}_{T}\left(Y,Z\right)\rightarrow \underline{\mathcal E}\left(Y,Z\right)$, we have $$Id\otimes\underline{\mathcal E}\left(TX, \xi_{Y} \right)\circ Id\otimes\varphi_{T}\circ U_{T}\otimes U_{T}= Id\otimes\underline{\mathcal E}\left(\xi_{X}, Y \right)\circ U_{T}\otimes U_{T}$$

Hence, the previous diagram is equivalent to the following one
$$
\shorthandoff{;:!?}
\xymatrix @!0 @C=7cm @R=2cm{\relax
\underline{\mathcal Alg}_{T}\left(Y,Z\right)\otimes \underline{\mathcal Alg}_{T}\left( X,Y\right) \ar[r]^{U_{T}\otimes U_{T}} \ar[d]_{U_{T}\otimes U_{T}} & \underline{\mathcal E}\left( Y,Z\right)\otimes \underline{\mathcal E}\left( X,Y\right) \ar[d]^{1 \otimes\varphi_{T}} \\
\underline{\mathcal E}\left( Y,Z\right)\otimes \underline{\mathcal E}\left( X,Y\right) \ar[d]_{\varphi_{T}\otimes \varphi_{T}} &  \underline{\mathcal E}\left( Y,Z\right)\otimes \underline{\mathcal E}\left( TX,TY\right) \ar[d]^{1 \otimes \underline{\mathcal E}\left(TX, \xi_{Y}\right)} \\
\underline{\mathcal E}\left( TY,TZ\right)\otimes \underline{\mathcal E}\left( TX,TY\right) \ar[d]_{\underline{\mathcal E}\left(TY, \xi_{Z} \right)\otimes 1} & \underline{\mathcal E}\left( Y,Z\right)\otimes \underline{\mathcal E}\left( TX,Y\right) \ar[d]^{c}\\
\underline{\mathcal E}\left( TY,Z\right)\otimes \underline{\mathcal E}\left( TX,TY\right) \ar[r]^{c} & \underline{\mathcal E}\left( TX,Z\right)
}
$$
Similarly, tensoring the equalizer $\underline{\mathcal Alg}_{T}\left(Y,Z\right)$ with 
$$\underline{\mathcal Alg}_{T}\left(X,Y\right)\xrightarrow{U_{T}} \underline{\mathcal E}\left(X,Y\right) \xrightarrow{\varphi_{T}} \underline{\mathcal E}\left(TX,TY\right)$$ we have $$\underline{\mathcal E}\left(TY, \xi_{Z}\right)\otimes Id \circ \varphi_{T}\otimes\varphi_{T}\circ U_{T}\otimes U_{T}= \underline{\mathcal E}\left(\xi_{Y}, Z \right)\otimes \varphi_{T}\circ U_{T}\otimes U_{T}$$
Hence, this diagram is equivalent to the following one
$$
\shorthandoff{;:!?}
\xymatrix @!0 @C=7cm @R=2cm{\relax
\underline{\mathcal Alg}_{T}\left(Y,Z\right)\otimes \underline{\mathcal Alg}_{T}\left( X,Y\right) \ar[r]^{U_{T}\otimes U_{T}} \ar[d]_{U_{T}\otimes U_{T}} & \underline{\mathcal E}\left( Y,Z\right)\otimes \underline{\mathcal E}\left( X,Y\right) \ar[d]^{1\otimes\varphi_{T}} \\
\underline{\mathcal E}\left( Y,Z\right)\otimes \underline{\mathcal E}\left( X,Y\right) \ar[dd]_{\underline{\mathcal E}\left(\xi_{Y}, Z\right)\otimes \varphi_{T}} &  \underline{\mathcal E}\left( Y,Z\right)\otimes \underline{\mathcal E}\left( TX,TY\right) \ar[d]^{1\otimes \underline{\mathcal E}\left(TX, \xi_{Y}\right)} \\
& \underline{\mathcal E}\left( Y,Z\right)\otimes \underline{\mathcal E}\left( TX,Y\right) \ar[d]^{c}\\
\underline{\mathcal E}\left( TY,Z\right)\otimes \underline{\mathcal E}\left( TX,TY\right) \ar[r]^{c} & \underline{\mathcal E}\left( TX,Z\right)
}
$$ 
But we have $c\circ Id\otimes\underline{\mathcal E}\left(TX, \xi_{Y}\right)= c\circ \underline{\mathcal E}\left(\xi_{Y}, Z\right)\otimes Id $.

Therefore, we obtain the following diagram
$$
\shorthandoff{;:!?}
\xymatrix @!0 @C=7cm @R=2cm{\relax
\underline{\mathcal Alg}_{T}\left(Y,Z\right)\otimes \underline{\mathcal Alg}_{T}\left( X,Y\right) \ar[r]^{U_{T}\otimes U_{T}} \ar[d]_{U_{T}\otimes U_{T}} & \underline{\mathcal E}\left( Y,Z\right)\otimes \underline{\mathcal E}\left( X,Y\right) \ar[d]^{1\otimes\varphi_{T}} \\
\underline{\mathcal E}\left( Y,Z\right)\otimes \underline{\mathcal E}\left( X,Y\right) \ar[dd]_{\underline{\mathcal E}\left(\xi_{Y}, Z\right)\otimes \varphi_{T}} &  \underline{\mathcal E}\left( Y,Z\right)\otimes \underline{\mathcal E}\left( TX,TY\right) \ar[d]^{\underline{\mathcal E}\left(\xi_{Y}, Z\right)\otimes 1} \\
& \underline{\mathcal E}\left( TY,Z\right)\otimes \underline{\mathcal E}\left( TX,TY\right) \ar[d]^{c}\\
\underline{\mathcal E}\left( TY,Z\right)\otimes \underline{\mathcal E}\left( TX,TY\right) \ar[r]^{c} & \underline{\mathcal E}\left( TX,Z\right)
}
$$ 
which is clearly commutative.

Therefore, by the universal property of the equalizer, there exists one unique morphism $$c_{\mathcal Alg_{T}}:\underline{\mathcal Alg}_{T}\left(Y,Z\right)\otimes \underline{\mathcal Alg}_{T}\left( X,Y\right)\rightarrow \underline{\mathcal Alg}_{T}\left( X,Z\right)$$ such that $c_{\mathcal E}\circ U_{T}\otimes U_{T}=U_{T}\circ c_{\mathcal Alg_{T}}$.  

It remains to prove that the coherence axioms are satisfied.
Since the category $\mathcal E$ is enriched over itself, the coherence axioms are satisfied for $\mathcal E$. Using the fact that each $U_{T}^{\left(X,Y\right)}:\underline{\mathcal Alg}_{T}\left(X,Y\right)\rightarrow \underline{\mathcal E}\left(X,Y\right)$ is a monomorphism, commutativity of the diagrams for $\mathcal E$ extends by $U_{T}$ to the commutativity of the external diagrams, providing the coherence axioms for $Alg_{T}$.

Thus, the internal object $\underline{\mathcal Alg}_{T}\left(X,Y\right)$, the unit and the composition morphism together with their coherence axioms provide the category $\mathcal Alg_{T}$ with the $\mathcal E$-category structure.     
\end{proof}

It follows then naturally that the adjunction between $T$-algebras and the corresponding monoidal category will be enriched as well (cf. \cite{BungeRFCCA, LintRFSAR}).
 
\begin{prop}\label{propenradj} 
Let $\mathcal E$ be a closed symmetric monoidal category with equalizers and $\left(T,\mu,\eta, \varphi \right)$ a $\mathcal E$-monad. Then:
\begin{enumerate}
 \item The forgetful functor $U_{T}:\mathcal Alg_{T}\rightarrow \mathcal E$ is a $\mathcal E$-functor;
 \item The free functor $F_{T}:\mathcal E\rightarrow \mathcal Alg_{T} $ is a $\mathcal E$-functor;
 \item The pair $\left( F_{T},U_{T}\right)$ forms a $\mathcal E$-adjunction i.e. there is an isomorphism in $\mathcal E$
$$ \underline{Alg}_{T} \left(F_{T}X,Y \right)\cong \underline{\mathcal E}\left(X,U_{T}Y \right)   $$ 
which is $\mathcal E$-natural in X and Y;
 \item $\mathcal Alg_T$ is canonically $\mathcal E$-cotensored (compatibly with $U_T$).

\end{enumerate}
\end{prop}
\begin{proof}~\\
\begin{enumerate}
 \item Since $\left(T,\mu,\eta, \varphi \right)$ is a $\mathcal E$-monad, by Proposition \ref{propenr}, the category $\mathcal Alg_{T}$ is enriched over $\mathcal E$. 
It is obvious that for every pair of objects $\left( X,Y\right)$ of $\mathcal E$ there exists a morphism $\varphi_{U_{T}}:\underline{\mathcal Alg}_{T}\left(X,Y\right) \rightarrow \underline{\mathcal E}\left(U_{T}X,U_{T}Y\right)$ in $\mathcal E$.
We saw in the proof of Proposition \ref{propenr}, that the unit morpism in $\mathcal Alg_{T}$ is given by the universal property of the equalizer i.e. there is an unique morphism $j^{\mathcal Alg_{T}}_{X}:I\rightarrow \underline{\mathcal Alg}_{T}\left( X,X\right)$ such that 
the following diagram
$$
\shorthandoff{;:!?}
\xymatrix @!0 @C=3cm @R=2cm{\relax
\underline{\mathcal Alg}_{T}\left(X,X\right) \ar[rr]^{\varphi_{U_{T}}} && \underline{\mathcal E}\left(U_{T}X,U_{T}X\right) \\
& I \ar[lu]^{j_{X}} \ar[ru]_{j_{U_{T}X}}
}
$$   
commutes. But this diagram corresponds exactly to the unit axiom of the functor $U_{T}$.

Similarly, the composition morphism in $\mathcal Alg_{T}$ is given by the universal property of the equalizer i.e. there is an unique morphism $$c_{\mathcal Alg_{T}}:\underline{\mathcal Alg}_{T}\left(Y,Z\right)\otimes \underline{\mathcal Alg}_{T}\left( X,Y\right)\rightarrow \underline{\mathcal Alg}_{T}\left( X,Z\right)$$ such that 
the following diagram
\begin{diagram}{}
$$
\shorthandoff{;:!?}
\xymatrix @!0 @C=7cm @R=3cm{\relax
\underline{\mathcal Alg}_{T}\left( Y,Z\right)\otimes \underline{\mathcal Alg}_{T}\left( X,Y\right) \ar[d]_{c_{\mathcal Alg_{T}}} \ar[r]^{\varphi_{U_{T}}\otimes \varphi_{U_{T}}} & \underline{\mathcal E}\left( U_{T}Y,U_{T}Z\right)\otimes\underline{\mathcal E}\left(U_{T}X,U_{T}Y\right) \ar[d]^{c_{\mathcal E}} \\ 
\underline{\mathcal Alg}_{T}\left( X,Z\right) \ar[r]^{\varphi_{U_{T}}} & \underline{\mathcal E}\left( U_{T}X,U_{T}Z\right)  \\
}
$$
\end{diagram}
commutes. But this diagram corresponds exactly to the composition axiom of the functor $U_{T}$.
 \item 
We need to prove that there exists the enrichment morphism of the functor $F_{T}$
$$\varphi_{F_{T}}:\underline{\mathcal E}\left(X,Y\right)\rightarrow \underline{\mathcal Alg}_{T}\left(F_{T}X,F_{T}Y\right) $$
We prove that the following diagram commutes
$$
\shorthandoff{;:!?}
\xymatrix @!0 @C=7cm @R=3cm{\relax
\underline{\mathcal E}\left( X,Y\right) \ar[d]_{\varphi_{T}} \ar[r]^{\varphi_{T}} & \underline{\mathcal E}\left( TX, TY\right) \ar[dd]^{\underline{\mathcal E}\left( \mu_{X},TY\right)} \\
\underline{\mathcal E}\left( TX, TY\right) \ar[d]_{\varphi_{T}}\\  
\underline{\mathcal E}\left( TTX, TTY\right) \ar[r]^{\underline{\mathcal E}\left( TTX,\mu_{Y}\right)} & \underline{\mathcal E}\left( TTX,TY\right)  \\
}
$$
Since $\left(T,\mu,\eta, \varphi \right)$ is a $\mathcal E$-monad, $\mu$ is a $\mathcal E$-natural transformation. 
It is given by a family of morphisms $\mu_{X}:TTX\rightarrow TX $ indexed by the objects of $\mathcal E$ such that the following diagram commutes
\begin{diagram}{}
$$
\shorthandoff{;:!?}
\xymatrix @!0 @C=7cm @R=3cm{\relax
\underline{\mathcal E}\left( X,Y\right) \ar[d]_{\varphi_{TT}} \ar[r]^{\varphi_{T}} & \underline{\mathcal E}\left( TX, TY\right) \ar[d]^{\underline{\mathcal E}\left( \mu_{X},TY\right)} \\  
\underline{\mathcal E}\left( TTX, TTY\right) \ar[r]^{\underline{\mathcal E}\left( TTX,\mu_{Y}\right)} & \underline{\mathcal E}\left( TTX,TY\right)  \\
}
$$
\end{diagram}
But $\varphi_{TT}=\varphi_{T}\circ\varphi_{T}$, and the previous diagram commutes.
Thus, by the universal property of the equalizer 
\begin{diagram}{}
$$
\shorthandoff{;:!?}
\xymatrix @!0 @C=3cm @R=2.3cm{\relax
&&& \underline{\mathcal E}\left( TTX,TTY\right) \ar[rd]^{\underline{\mathcal E}\left( TTX,\mu_{Y}\right)} \\
\underline{\mathcal Alg}_{T}\left( F_{T}X,F_{T}Y\right) \ar[rr]^{U_{T}} && \underline{\mathcal E}\left( TX,TY\right) \ar[ru]^{\varphi_{T}} \ar[rr]_{\underline{\mathcal E}\left(\mu_{X},TY\right)} && \underline{\mathcal E}\left( TTX,TY\right) \\
& \underline{\mathcal E}\left( X,Y\right) \ar@{-->}[lu]^{\exists!\varphi_{F_{T}}} \ar[ru]_{\varphi_{T}}  \\
}
$$
\end{diagram}
there exists an unique morphism $\varphi_{F_{T}}$ such that $\varphi_{U_{T}}\circ \varphi_{F_{T}}=\varphi_{T}$.
 \item
Since functors $\left( F_{T},U_{T}\right)$ form a classical adjunction, we need to prove that the unit and the counit of the adjunction are $\mathcal E$-natural transformations.
The unit of the adjunction $\eta:Id_{\mathcal E}\rightarrow U_{T}F_{T}$ is exactly the unit of a $\mathcal E$-monad $\left(T,\mu,\eta, \varphi \right)$. 
It remains to verify that the counit of the adjunction $\epsilon:F_{T}U_{T}\rightarrow Id_{\mathcal Alg_{T}}$ is a $\mathcal E$-natural transformation. We have to prove the commutativity of the following diagram
$$
\shorthandoff{;:!?}
\xymatrix @!0 @C=4cm @R=2cm{\relax
\underline{\mathcal Alg}_{T}\left( X,Y\right) \ar[rd]_{F_{T}U_{T}} \ar[rr]^{\underline{\mathcal Alg}_{T}\left( \epsilon,1\right)} && \underline{\mathcal Alg}_{T}\left( TX,Y\right)\\  
& \underline{\mathcal Alg}_{T}\left( TX,TY\right) \ar[ru]_{\underline{\mathcal Alg}_{T}\left( 1,\epsilon\right)} \\
}
$$  
Since $U_{T}$ is an equalizer, it is equivalent to prove that the following diagram
$$
\shorthandoff{;:!?}
\xymatrix @!0 @C=5cm @R=3cm{\relax
\underline{\mathcal Alg}_{T}\left( X,Y\right) \ar[d]_{U_{T}} \ar[r]^{\underline{\mathcal Alg}_{T}\left(\epsilon,1\right)} & \underline{\mathcal Alg}_{T}\left( TX,Y\right) \ar[r]^{U_{T}} & \underline{\mathcal E}\left( TX, Y\right)  \\  
\underline{\mathcal E}\left( X, Y\right) \ar[r]^{F_{T}} & \underline{\mathcal Alg}_{T}\left( TX,TY\right) \ar[r]^{\underline{\mathcal Alg}_{T}\left(1, \epsilon\right)} & \underline{\mathcal Alg}_{T}\left( TX,Y\right) \ar[u]_{U_{T}} \\
}
$$
commutes, which is equivalent to
$$
\shorthandoff{;:!?}
\xymatrix @!0 @C=5cm @R=3cm{\relax
\underline{\mathcal Alg}_{T}\left( X,Y\right) \ar[d]_{U_{T}} \ar[r]^{\underline{\mathcal Alg}_{T}\left(\epsilon,1\right)} & \underline{\mathcal Alg}_{T}\left( TX,Y\right) \ar[r]^{U_{T}} & \underline{\mathcal E}\left( TX, Y\right)  \\  
\underline{\mathcal E}\left( X, Y\right) \ar[r]^{F_{T}} & \underline{\mathcal Alg}_{T}\left( TX,TY\right) \ar[r]^{U_{T}} & \underline{\mathcal E}\left( TX,TY\right) \ar[u]_{\underline{\mathcal E}\left( 1,\epsilon \right)} \\
}
$$
Since $U_{T}\circ F_{T}=T$, we have
$$
\shorthandoff{;:!?}
\xymatrix @!0 @C=5cm @R=3cm{\relax
\underline{\mathcal Alg}_{T}\left( X,Y\right) \ar[d]_{U_{T}} \ar[r]^{\underline{\mathcal Alg}_{T}\left(\epsilon,1\right)} & \underline{\mathcal Alg}_{T}\left( TX,Y\right) \ar[r]^{U_{T}} & \underline{\mathcal E}\left( TX, Y\right)  \\  
\underline{\mathcal E}\left( X, Y\right) \ar[rr]^{T} && \underline{\mathcal E}\left( TX,TY\right) \ar[u]_{\underline{\mathcal E}\left( 1,\epsilon \right)} \\
}
$$
Since $U_{T}$ equalizes $$\underline{\mathcal E}\left( X, Y\right)\xrightarrow{\underline{\mathcal E}\left( \epsilon, 1\right)} \underline{\mathcal E}\left( TX, Y\right)$$ and $$\underline{\mathcal E}\left( X, Y\right)\xrightarrow{T} \underline{\mathcal E}\left( TX, TY\right) \xrightarrow{\underline{\mathcal E}\left( 1, \epsilon \right)} \underline{\mathcal E}\left( TX, Y\right)$$ we have the following diagram
$$
\shorthandoff{;:!?}
\xymatrix @!0 @C=5cm @R=3cm{\relax
\underline{\mathcal Alg}_{T}\left( X,Y\right) \ar[d]_{U_{T}} \ar[r]^{\underline{\mathcal Alg}_{T}\left(\epsilon,1\right)} & \underline{\mathcal Alg}_{T}\left( TX,Y\right) \ar[d]^{U_{T}}  \\  
\underline{\mathcal E}\left( X, Y\right) \ar[r]^{\underline{\mathcal E}\left( \epsilon,1 \right)} & \underline{\mathcal E}\left( TX, Y\right)  \\
}
$$
which clearly commutes.
\item
In order to prove that the category $\mathcal Alg_{T}$ is cotensored over $\mathcal E$ we will use the dual tensorial strength. More precisely, by the dual tensorial strength we have $$T(Y^A)\otimes A\xrightarrow{\sigma'_{Y^{A},A}} T(Y^A\otimes A) \xrightarrow{T(ev)} T(Y)$$ 
which corresponds by adjunction to $T(Y^A)\rightarrow T(Y)^A$, which is natural in $Y$ and $A$. 
If $Y$ is a $T$-algebra, then this provides $Y^A$ with a structure of $T$-algebra, for every object $A$ in $\mathcal E$. 
This gives us exactly the cotensor.
\end{enumerate}
\end{proof}

\subsection{Tensors}

Using the concept of strong monad we prove that the category $\mathcal Alg_{T}$ is $\mathcal E$-tensored as well.   

\begin{prop}\label{proptens}
Let $\mathcal E$ be a closed symmetric monoidal category with equalizers and $\left(T,\mu,\eta, \sigma \right)$ a strong monad on $\mathcal E$. 
If the category $\mathcal Alg_{T}$ of $T$-algebras has reflexive coequalizers, then the category $\mathcal Alg_{T}$ is enriched, tensored and cotensored over $\mathcal E$. 
The tensors are given by the following reflexive coequalizer diagrams in $\mathcal Alg_T$
\begin{diagram}{}
$$
\shorthandoff{;:!?}
\xymatrix @!0 @C=3cm @R=1.8cm{\relax
T\left( Z\otimes TX\right) \ar[rr]^{T\left(Z\otimes \xi_{X} \right) } \ar[rd]_{T\sigma} && T\left( Z\otimes X \right) \ar[r]^{\xi} & Z\otimes X\\
& TT\left( Z\otimes X\right) \ar[ru]_{\mu}  
}
$$
\end{diagram}
with X an object in $\mathcal Alg_{T}$ and Z an object in $\mathcal E$.
\end{prop}

\begin{proof}
By Propositions \ref{propenr} and \ref{propenradj}, the only thing that remains to be shown is the existence of tensors. 
We begin by giving a heuristic argument how to construct these tensors in $\mathcal Alg_T$. 
Indeed, by Lemma \ref{lemabscoeq}, for every $T$-algebra $X$, the following diagram is a coequalizer in $\mathcal Alg_{T}$
$$
\shorthandoff{;:!?}
\xymatrix @!0 @C=3cm {\relax
 TTX \ar@<2pt>[r]^{\mu_{X}} \ar@<-2pt>[r]_{T\left(\xi_{X} \right) } & TX \ar[r]^{\xi_{X}} & X 
}
$$
It is transformed by the forgetful functor $U_{T}:Alg_{T}\longrightarrow \mathcal E $ into a split coequalizer in $\mathcal E$ (which is preserved as a coequalizer by any functor).

Applying the functor $Z\otimes -$, we obtain a coequalizer in $\mathcal E $
$$
\shorthandoff{;:!?}
\xymatrix @!0 @C=4cm {\relax
 Z\otimes TTX \ar@<2pt>[r]^{Z\otimes \mu_{X}} \ar@<-2pt>[r]_{Z\otimes T\left(\xi_{X} \right) } & Z\otimes TX \ar[r]^{\xi_{X}} & Z\otimes X 
}
$$
We have the following diagram
$$
\shorthandoff{;:!?}
\xymatrix @!0 @C=3cm @R=2cm{\relax
Z\otimes TTX \ar@<2pt>[rr]^{Z\otimes \mu_{X}} \ar@<-2pt>[rr]_{Z\otimes T\left(\xi_{X} \right) } \ar[d]_{\sigma_{Z,TX}} && Z\otimes TX \ar[rd]^{\xi_{X}} \ar[d]^{\sigma_{Z,X}} \\
T\left(Z\otimes TX \right) \ar[rr]^{T\left(Z\otimes \xi_{X}\right) } \ar[rd]_{T\sigma_{Z,X}} && T\left(Z\otimes X \right) \ar[r]^{\xi_{Z\otimes X}} & Z\otimes X \\
& TT\left(Z\otimes TX \right) \ar[ru]_{\mu_{Z\otimes X}} 
}
$$
where the internal diagram
$$
\shorthandoff{;:!?}
\xymatrix @!0 @C=4.7cm @R=2.3cm{\relax
Z\otimes TTX \ar[r]^{Z\otimes T\left(\xi_{X} \right) } \ar[d]_{\sigma_{Z,TX}} & Z\otimes TX \ar[d]^{\sigma_{Z,X}}\\
T\left(Z\otimes TX\right) \ar[r]^{T\left(Z\otimes \xi_{X} \right) } & T\left(Z\otimes X\right) 
}
$$
commutes by naturality of the tensorial strength $\sigma$.

The other internal diagram
$$
\shorthandoff{;:!?}
\xymatrix @!0 @C=4.7cm @R=2.3cm{\relax
Z\otimes TTX \ar[rr]^{Z\otimes \mu_{X}} \ar[d]_{\sigma_{Z,TX}} && Z\otimes TX \ar[d]^{\sigma_{Z,X}}\\
T\left(Z\otimes TX\right) \ar[r]^{T\sigma_{Z,X} } & TT\left(Z\otimes X\right) \ar[r]^{\mu_{Z\otimes X} }  &  T\left(Z\otimes X\right) 
}
$$
commutes, since $\mu$ is a strong natural transformation.
Therefore, it is natural to define the tensor $Z\otimes X$ in $\mathcal Alg_T$ by the following diagram
$$
\shorthandoff{;:!?}
\xymatrix @!0 @C=3cm @R=1.8cm{\relax
T\left( Z\otimes TX\right) \ar[rr]^{T\left(Z\otimes \xi_{X} \right) } \ar[rd]_{T\sigma} && T\left( Z\otimes X \right) \ar[r]^{\xi} & Z\otimes X\\
& TT\left( Z\otimes X\right) \ar[ru]_{\mu}  
}
$$
which is a reflexive coequalizer in $\mathcal Alg_{T}$. 
If the monad $T$ happens to preserve reflexive coequalizers, the latter can be calculated in $\mathcal E$ and the argument above shows that we get indeed the correct tensor. 
In general, in order to validate our definition, we have to verify that our tensors fulfill the following adjunction relation
$$\underline{\mathcal Alg}_{T}\left(Z\otimes X,Y\right)\cong \underline{\mathcal E}\left(Z,\underline{\mathcal Alg}_{T}\left(X,Y\right)\right)$$

We have $\underline{\mathcal Alg}_{T}\left(Z\otimes X,Y\right)\cong \underline{\mathcal Alg}_{T}\left(Coeq\left\lbrace T(Z\otimes TX)\rightrightarrows T(Z\otimes X) \right\rbrace  ,Y\right)$
$$\cong Eq\left(\underline{\mathcal Alg}_{T}\left( T(Z\otimes X),Y\right) \rightrightarrows\underline{\mathcal Alg}_{T}\left( T(Z\otimes TX),Y\right)  \right) $$

By Proposition \ref{propenradj} (c), functors $\left( F_{T},U_{T}\right)$ form a $\mathcal E$-adjunction 
$$\underline{\mathcal Alg}_{T}\left(F_{T}X,Y \right)\cong \underline{\mathcal E}\left(X,U_{T}Y \right) $$ 

Therefore, $$Eq \left\lbrace  \underline{\mathcal Alg}_{T}\left( T(Z\otimes X),Y\right) \rightrightarrows \underline{\mathcal Alg}_{T}\left( T(Z\otimes TX),Y\right) \right\rbrace$$ 
$$\cong Eq \left\lbrace \underline{\mathcal E} \left( Z\otimes X,Y\right) \rightrightarrows \underline{\mathcal E} \left( Z\otimes TX,Y\right) \right\rbrace$$ $$\cong \underline{\mathcal E}\left(Z, Eq \left\lbrace \underline{\mathcal E} \left( X,Y\right)\rightrightarrows \underline{\mathcal E} \left( TX,Y\right) \right\rbrace \right) $$

Since, $\underline{\mathcal Alg}_{T} \left(X,Y \right)$ is the equalizer of $\underline{\mathcal E} \left( X,Y\right)\rightrightarrows \underline{\mathcal E} \left( TX,Y\right)$, we have 
$$\underline{\mathcal E}\left(Z, Eq \left\lbrace \underline{\mathcal E} \left( X,Y\right)\rightrightarrows \underline{\mathcal E}\left( TX,Y\right) \right\rbrace \right)  \cong \underline{\mathcal E} \left( Z,\underline{\mathcal Alg}_{T} \left(X,Y \right) \right).$$
\end{proof}

\subsection{The endomorphism monoid}

We study the object $T\left(I\right) $ in closed symmetric monoidal category $\mathcal E$, which in general does not have the structure of monoid in $\mathcal E$. 
We prove here that if $\left(T,\mu,\eta, \varphi \right)$ is a $\mathcal E$-monad, this will be the case.

\begin{prop}\label{propmonTI}
Let $\left(T,\mu,\eta, \varphi \right)$ be a $\mathcal E$-monad. Then the object $T\left( I\right) $ has a structure of a monoid, namely it may be identified with $\underline{\mathcal Alg}_{T}\left(T\left(I \right),T\left(I \right)\right)$.
\end{prop}

\begin{proof}
In fact, we have $$T\left( I\right)\cong \underline{\mathcal E}\left(I,T\left(I \right)\right)\cong \underline{\mathcal Alg}_{T}\left(T\left(I \right),T\left(I \right)\right)$$ where the second isomorphism is obtained by adjunction. 

Since $\left(T,\mu,\eta, \varphi \right)$ is a $\mathcal E$-monad, by Proposition \ref{propenr} the category $Alg_{T}$ is a $\mathcal E$-category. 
The enriched endomorphism object $\underline{\mathcal Alg}_{T}\left(T\left(I \right),T\left(I \right)\right) \cong T(I)$ of $\mathcal E$ has indeed a structure of monoid in $\mathcal E$.

The unit morphism $I\rightarrow \underline{\mathcal Alg}_{T}\left(T\left(I \right),T\left(I \right)\right)$ is given by the unit morphism in the category $Alg_{T}$.
The multiplication morphism $$\underline{\mathcal Alg}_{T}\left(T\left(I \right),T\left(I \right)\right) \otimes \underline{\mathcal Alg}_{T}\left(T\left(I \right),T\left(I \right)\right) \rightarrow \underline{\mathcal Alg}_{T}\left(T\left(I \right),T\left(I \right)\right)$$ is given by the composition in the $\mathcal E$-category $Alg_{T}$.   
\end{proof}

\begin{lem}\label{lemMonStr}
The multiplication of the endomorphism monoid $T(I)$ of a $\mathcal{E}$-monad may be deduced from its tensorial strength through the formula 
$$T(I)\otimes T(I) \xrightarrow{\sigma_{T(I),I}} T(T(I)\otimes I) \xrightarrow{T(r)} T(T(I)) \xrightarrow{\mu} T(I)$$ 
\end{lem}

\begin{proof}
The composed arrow corresponds to the structure of a right $T(I)$-module $T(I)$. Obviously, this corrsponds to the multiplication.
\end{proof}

Before proving that $\lambda:-\otimes T\left(I\right)\rightarrow T $ is a morphism of strong monads (Proposition \ref{propmorsm}), we will need the following result.
\begin{lem}\label{lemstmon}
Let $\mathcal E$ be a monoidal category and suppose that M is a monoid in $\mathcal E$.

Then the endofunctor $-\otimes M:\mathcal E\rightarrow \mathcal E$ has a canonical structure of a strong monad and the tensorial strength is given by the associativity isomorphism in $\mathcal E$.  
\end{lem}

\begin{proof}

For every monoid $M$, we know that $\left( -\otimes M,\mu, \eta\right) $ has a structure of a monad. 
In order to prove that $S=-\otimes M$ is a strong monad, we need to prove the existence of a tensorial strength together with the corresponding axioms.
For every pair of objects $\left( X,Y\right) $ in $\mathcal E$, the tensorial strength $\sigma_{X,Y}:X\otimes SY\rightarrow S\left(X\otimes Y \right) $ is given by
$$X\otimes \left( Y\otimes M\right) \rightarrow \left( X\otimes Y\right) \otimes M $$ which corresponds exactly to the associativity isomorphism in $\mathcal E$.

Indeed, by definition the tensorial strength is given by the following commutative diagram
$$
\shorthandoff{;:!?}
\xymatrix @!0 @C=8cm @R=2.5cm{\relax
X\otimes \left(Y\otimes M \right)    \ar[r]^{\sigma_{X,Y}} \ar[d]_{\gamma_{Y}\otimes \left(Y\otimes M \right) } & \left(X\otimes Y \right)\otimes M \\
\mathcal E\left(Y,X\otimes Y \right)\otimes  \left(Y\otimes M \right) \ar[r]^{\varphi\otimes \left(Y\otimes M \right) } & \mathcal E\left(Y\otimes M,\left( X\otimes Y\right) \otimes M \right)\otimes  \left(Y\otimes M \right) \ar[u]_{ev_{Y\otimes M}} 
}
$$
which is equivalent to
$$
\shorthandoff{;:!?}
\xymatrix @!0 @C=8cm @R=2.5cm{\relax
X\otimes \left(Y\otimes M \right)    \ar[r]^{\sigma_{X,Y}} \ar[d]_{\gamma_{Y}\otimes \left(Y\otimes M \right) } & \left(X\otimes Y \right)\otimes M \\
\mathcal E\left(Y,X\otimes Y \right)\otimes  \left(Y\otimes M \right) \ar[r]^{a} & \left( \mathcal E\left(Y,\left( X\otimes Y\right) \right)\otimes Y\right) \otimes M  \ar[u]_{ev_{Y}\otimes M} 
}
$$
By naturality of the associativity isomorphism
$$
\shorthandoff{;:!?}
\xymatrix @!0 @C=8cm @R=2.5cm{\relax
X\otimes \left(Y\otimes M \right)    \ar[r]^{\sigma_{X,Y}} \ar[d]_{a}  & \left(X\otimes Y \right)\otimes M \\
\left( X\otimes Y\right) \otimes M  \ar[r]^{\gamma_{Y}\otimes \left(Y \right)\otimes M } &  \mathcal E\left(Y,\left( X\otimes Y\right) \right)\otimes Y \otimes M  \ar[u]_{ev_{Y}\otimes M} 
}
$$
But we have $ev_{Y} \circ \gamma_{Y}= Id_{X\otimes Y}$, thus $\sigma_{X,Y}=a_{XYM}$.

Since the tensorial strength is given by the associativity isomorphism, the unit and the associativty axiom and the strong naturality conditions for $\eta$ and $\mu$ are obtained automatically.
\end{proof}

\begin{rem}
Another way to prove that the endofunctor $-\otimes M$ has a structure of a strong monad is to use the correspondance between
the tensorial strength and the enrichment. 
More precisely, by Lemma \ref{lememon}, the monad $\left( -\otimes M,\mu, \eta\right) $ has a structure of a $\mathcal E$-monad and by 
Corollary \ref{cormons-e} $\left( -\otimes M,\mu, \eta\right) $ is a strong monad.
\end{rem}

\begin{prop}\label{propmorsm}
Let $\mathcal E$ be a monoidal category and suppose that  $\left(T,\mu,\eta, \sigma \right)$ is a strong monad in $\mathcal E$.

There is a canonical map of strong monads $\lambda:-\otimes T\left(I\right)\rightarrow T $ given by the tensorial strength
\begin{diagram}{}
$$
\shorthandoff{;:!?}
\xymatrix @!0 @C=3cm @R=2.5cm{\relax
X\otimes T\left(I\right)    \ar[rr]^{\lambda_{X}} \ar[rd]_{\sigma_{X,I}} && TX \\
& T\left(X\otimes I \right) \ar[ru]_{T\left(r\right) }  
}
$$
\end{diagram}

This map is an isomorphism if and only if the monad T is induced by a monoid.
\end{prop}

\begin{proof}
By Lemma \ref{lemstmon}, the endofunctor $S=-\otimes T\left(I\right)$ has a structure of a strong monad.
Given monads $\left(T,\mu,\eta, \sigma_{T} \right)$ and $\left( S,\tilde{\mu},\tilde{\eta}, \sigma_{S} \right)$, we have to prove that $\lambda:S\rightarrow T $ is a morphism of monads.

First, we have to verify that $\lambda$ is a strong natural transformation i.e. that the following diagram
$$
\shorthandoff{;:!?}
\xymatrix @C=4cm @R=2.3cm{\relax
X\otimes \left( Y\otimes T\left( I\right)  \right) \ar[r]^{\sigma_{S}} \ar[d]_{X\otimes \lambda_{Y}} & \left(X\otimes Y\right)\otimes T\left( I\right)  \ar[d]^{\lambda_{X\otimes Y}}\\
X\otimes TY \ar[r]^{\sigma_{T}} & T\left(X\otimes Y\right) 
}
$$
is commutative. Extending definitions:
$$
\shorthandoff{;:!?}
\xymatrix @C=4cm @R=2.3cm{\relax
X\otimes \left( Y\otimes T\left( I\right)  \right) \ar[r]^{a_{XYI}} \ar[d]_{X\otimes \sigma} & \left(X\otimes Y\right)\otimes T\left( I\right)  \ar[d]^{\sigma}\\
X\otimes T\left( Y\otimes I  \right) \ar[r]^{\sigma_{X,Y\otimes I}} \ar[d]_{X\otimes T\left(r\right) } & T\left(X\otimes Y\otimes I\right) \ar[d]^{T\left(r\right) }\\
X\otimes TY \ar[r]^{\sigma_{T}} & T\left(X\otimes Y\right) 
}
$$
The upper diagram commutes by the associativity axiom of a strong functor T.
The lower diagram commutes by naturality of $\sigma$.

It remains to prove that the two axioms which define a morphism of monads are satisfied.
First, we need to prove the commutativity of the following diagram
$$
\shorthandoff{;:!?}
\xymatrix  @!0 @C=2cm @R=2cm{\relax 
 X\otimes T\left(I \right)  \ar[rr]^{\lambda}  && TX \\ 
& X \ar[ru]_{\eta} \ar[lu]^{\tilde{\eta}} 
}
$$
Extending definitions: 
$$
\shorthandoff{;:!?}
\xymatrix @!0 @C=6cm @R=2cm{\relax 
 X\otimes T\left(I \right)  \ar[r]^{\sigma} & T\left(X\otimes I\right) \ar[r]^{T\left(r\right) } & TX \\
& X\otimes I \ar[u]_{\eta_{X\otimes I}} \ar[lu]_{X\otimes \eta_{I}} \ar[d]^{r} \\  
& X \ar[ruu]_{\eta_{X}} \ar[luu]^{\tilde{\eta}} 
}
$$
The left upper diagram commutes by strong naturality of $\eta$.
The left lower diagram commutes since r is the right unit isomorphism.
The right diagram commutes by naturality of $\eta$.

Finally, one has to prove that the following diagram commutes
$$
\shorthandoff{;:!?}
\xymatrix @!0 @C=5cm @R=2.5cm{\relax
\left(X\otimes T\left(I \right) \right)\otimes T\left( I\right)   \ar[r]^{\lambda\circ \lambda} \ar[d]_{\tilde{\mu}} & TTX \ar[d]^{\mu} \\
 X\otimes T\left(I \right) \ar[r]^{\lambda} & TX
}
$$
Extending definitions 
$$
\shorthandoff{;:!?}
\xymatrix @!0 @C=5cm @R=2.5cm{\relax
\left(X\otimes T\left(I \right) \right)\otimes T\left( I\right)   \ar[rr]^{\lambda\circ \lambda} \ar[d]_{\tilde{\mu}} && TTX \ar[d]^{\mu} \\
 X\otimes T\left(I \right)  \ar[r]^{\sigma} & T\left(X\otimes I\right) \ar[r]^{T\left(r\right) } & TX
}
$$
and
$$
\shorthandoff{;:!?}
\xymatrix @!0 @C=5cm @R=2cm{\relax
\left(X\otimes T\left(I \right) \right)\otimes T\left( I\right) \ar@{}[rrd]|{I}  \ar[rr]^{\sigma} \ar[d]_{\tilde{\mu}} && T(X\otimes T(I)\otimes I) \ar[d]^{T(r)}   \\
 X\otimes T\left(I \right) \ar@{}[rrd]|{II} \ar[d]_{\sigma} & X\otimes TT\left(I \right) \ar[l]_{X\otimes \mu_{I}}  \ar[r]^{\sigma} & T(X\otimes T(I)) \ar[d]^{T\sigma}\\
T\left(X\otimes I\right) \ar[d]_{T\left(r\right)} \ar@{}[rrd]|{III} && TT\left(X\otimes I\right) \ar[d]^{TTr} \ar[ll]_{\mu_{X\otimes I}}\\
TX && TTX \ar[ll]_{\mu_{X}}
}
$$
Diagram II commutes by strong naturality of $\mu$ (see Diagram \ref{diagstmu}).
Diagram III commutes by naturality of $\mu$.
We look more carefully Diagram I
$$
\shorthandoff{;:!?}
\xymatrix @!0 @C=5cm @R=2cm{\relax
\left(X\otimes T\left(I \right) \right)\otimes T\left( I\right) \ar[rd]^{X\otimes \lambda} \ar@{}[rrd]|{IV} \ar[rr]^{\sigma} \ar[d]_{\tilde{\mu}} && T(X\otimes T(I)\otimes I) \ar[d]^{T(r)}   \\
 X\otimes T\left(I \right) & X\otimes TT\left(I \right) \ar[l]_{X\otimes \mu_{I}}  \ar[r]^{\sigma} & T(X\otimes T(I))
}
$$
Diagram IV commutes by the naturality of $\sigma$. 
\end{proof}

The following corollary is a consequence of Proposition \ref{propenr}.
\begin{cor}
Let $\mathcal E$ be a closed symmetric monoidal category with equalizers and $-\otimes T\left(I \right) $ be a $\mathcal E$-monad, induced by the monoid $T\left(I\right) $. 
Then the category $Mod_{T\left(I\right) }$ of modules over $T\left(I \right)$ is canonically enriched over $\mathcal E$. 
\end{cor}

\begin{proof}
We have $Mod_{T\left(I\right) }=\mathcal Alg_{-\otimes T\left(I \right) }$.
Since $-\otimes T\left(I \right) $ is a $\mathcal E$-monad, by Proposition \ref{propenr}, the category $\mathcal Alg_{-\otimes T\left(I \right) }$ is canonically enriched over $\mathcal E$.   
\end{proof}

\section{Homotopical Morita theorem}\label{SMth}
In this section, we state the homotopical Morita theorem, Theorem \ref{thMain} which says that, under suitable conditions on $T$ and $\mathcal E$, there is a Quillen equivalence
between the category of $T$-algebras and the category of modules over the monoid $T\left(I\right)$.

\begin{theoreme}\label{thMain}{(Homotopical Morita theorem)}
Let $\mathcal E$ be a cofibrantly generated monoidal model category with cofibrant unit $I$ and with generating cofibrations having cofibrant domain.
Assume given a strong monad $\left(T,\mu,\eta, \sigma \right)$ on $\mathcal E$ such that
\begin{enumerate}
 \item The category of T-algebras $\mathcal Alg_{T}$ admits a transferred model structure;
 \item The unit $\eta_{X}: X\rightarrow T\left( X\right) $ is a cofibration at each cofibrant object X in $\mathcal E$;
 \item The tensorial strength $$\sigma_{X,Y}: X\otimes TY \xrightarrow{\sim} T\left(X \otimes Y \right)$$ is a weak equivalence for all cofibrant objects X,Y in $\mathcal E$;
 \item The forgetful functor takes free cell attachments in $\mathcal Alg_T$ to homotopical cell attachments in $\mathcal E$ (cf. Definition \ref{defcellatt}).

\end{enumerate}

Then the monad morphism $\lambda: -\otimes T\left(I\right)\rightarrow T $ induces a Quillen equivalence between the category of $T\left( I\right)$-modules and the category of T-algebras:  
$$ Ho\left( Mod_{T\left( I\right) }\right) \simeq Ho\left( \mathcal Alg_{T}\right) $$
\end{theoreme}

\begin{rem}~\\

\begin{enumerate}
 \item Since by Proposition \ref{propenradj}, the category of $T$-algebras $\mathcal Alg_{T}$ is cotensored over $\mathcal E$ (and $\mathcal Alg_{T}$ fulfills the adjoint of Hovey's pushout-product axiom rel. to $\mathcal E$), each fibrant $T$-algebra possesses a path object. Indeed, it suffices to factor the folding map of the unit $I\sqcup I\rightarrow I$ into a cofibration followed by a weak equivalence $I\sqcup I\to H\to I$. For each fibrant $T$-algebra $X$, the induced maps of $T$-algebras $X=X^I\to X^H\to X^{I\sqcup I}=X\times X$ then define a path-object for $X$. 
Therefore, by Theorem \ref{thcalgT} $(b)$, hypothesis $(a)$ essentially amounts to the existence of a fibrant replacement functor for $T$-algebras. 
 \item Hypothesis $(c)$ of the theorem is slightly redundant under assumption $(b)$. Namely if all $\sigma_{Y,I}$ for cofibrant objects $Y$ are weak equivalences, then it follows from $(b)$ and from Brown's Lemma \ref{lemBrowns} that $X\otimes\sigma_{Y,I}$ is a weak equivalence for cofibrant objects X. Consequently, Diagram \ref{diagsteta} in the definition of strong monad implies that $\sigma_{X,Y}$ is a weak equivalence for cofibrant objects $X$,$Y$.

\item Hypothesis $(d)$ is admittedly of a technical nature. 
It is likely that slightly different conditions also suffice to ensure the validity of the theorem.
We chose this hypothesis because on one side it allows a quite direct application of Reedy's patching lemma and on the other side
it is satisfied in the main example we study, cf. Sections \ref{sRecResSS} and \ref{SRealMModCat}, especially Proposition \ref{propForgfuncRE}. 
\end{enumerate}
\end{rem}

\begin{proof}

By Corollary \ref{cormons-e}, $\left(T,\mu,\eta \right)$ extends to a strong monad if and only if it extends to a $\mathcal E$-monad.
Therefore, by Proposition \ref{propenr}, \ref{propenradj} and \ref{proptens}, the category of $T$-algebras is enriched, tensored and cotensored over $\mathcal E$.
By Proposition \ref{propmonTI}, the object $T\left(I \right) $ has the structure of a monoid in $\mathcal E$.
We can give a look at the category $Mod_{T\left( I\right)}$ of modules over a monoid $T\left(I \right)$.
Since by hypothesis, $I$ is cofibrant, the unit $\eta_I:I \to T(I)$ is a cofibration by hypothesis $(b)$, i.e. the monoid $T(I)$ is well-pointed. 
Therefore, by Proposition \ref{propmodmod}, the category $Mod_{T\left( I\right)}$ admits a transferred model structure.

We know that $Mod_{T\left( I\right) }=\mathcal Alg_{-\otimes T\left( I\right) }$.

By Proposition \ref{propmorsm}, there is a canonical map  $\lambda:-\otimes T\left(I\right)\rightarrow T $ of strong monads.

There exists a functor $$ G:\mathcal Alg_{T}\rightarrow \mathcal Alg_{-\otimes T\left( I\right)}  $$ 
which takes the $T$-algebra $(X,\xi_X:TX\rightarrow X)$ to the $T(I)$-module $(X,\xi_X\lambda_X:X\otimes T(I)\rightarrow X)$. 
Observe in particular that $G$ preserves the underlying objects. 
The category $\mathcal Alg_{T}$ is cocomplete and therefore $\mathcal Alg_{T}$ admits reflexive coequalizers. 
Hence, by Proposition \ref{propleftadj}, the functor $G$ has a left adjoint 
$$F:\mathcal Alg_{-\otimes T\left( I\right)} \rightarrow \mathcal Alg_{T}$$
and we thus have a commutative diagram of right adjoint functors
$$
\shorthandoff{;:!?}
\xymatrix @!0 @C=3cm @R=1.8cm{\relax
\mathcal Alg_{T} \ar[rr]^{G} \ar[rd]_{U_{T}} && Mod_{T\left( I\right)} \ar[ld]^{V} \\
& \mathcal E 
}
$$
It follows that we have an analogous commutative diagram of left adjoint functors,
i.e. the left adjoint functor $F$ takes free $T(I)$-modules to free $T$-algebras.

Fibrations and weak equivalences in $Mod_{T\left( I\right) }$ are exactly the fibrations and weak equivalences in $\mathcal E$,
and fibrations and weak equivalences in $\mathcal Alg_{T}$  are exactly the fibrations and weak equivalences in $\mathcal E$.

Therefore, the right adjoint $G$ preserves and even reflects fibrations and weak equivalences.
In order to be a right Quillen functor, $G$ needs to preserve fibrations and acyclic fibrations.
Hence, $G$ is a right Quillen functor and $\left( F,G\right) $ forms a Quillen adjunction.

It remains to prove that $\left( F,G\right) $ is a Quillen equivalence.

Since the functor $G$ preserves and reflects fibrations and weak equivalences, $\left( F,G\right) $ is a Quillen equivalence 
if and only if for every cofibrant module $M$ the unit of the adjunction $\eta_{M}: M\rightarrow GFM$ is a weak equivalence. 

Since the left adjoint $F$ takes free $T(I)$-modules to free $T$-algebras, the unit of the adjunction at a free module $X\otimes T(I)$ is given by 
$$\eta_{X\otimes T(I)}:X\otimes T(I)\to TX$$ and coincides with the tensorial strength $\sigma_{X,I}$. 
It is therefore a weak equivalence if $X$ is cofibrant in $\mathcal E$, since $I$ is cofibrant by hypothesis. 
Using the patching and telescope lemmas of Reedy (Lemma \ref{lemPatch} and \ref{lemTeles})
we shall now extend this property to all cofibrant $T(I)$-modules.

We first show that the property ``$\eta_Z:Z \to GF(Z)$ is a weak equivalence between cofibrant objects'' is closed under cobase change of $Z$ along free $T(I)$-maps on cofibrations beween cofibrant objects in $\mathcal E$. 
Indeed, let us consider the following cube
$$
\shorthandoff{;:!?}
 \xymatrix {
    & T\left(X \right)  \ar[rr] \ar[dd]  && GF(Z') \ar[dd] \\
    X\otimes T\left(I \right) \ar[ru]^{\sim} \ar[rr] \ar[dd] && Z' \ar[ru]^{\sim} \ar[dd] \\
    & T\left(Y \right) \ar[rr] |!{[ur];[dr]}\hole && GF(Z) \\
     Y\otimes T\left(I \right)  \ar[ru]^{\sim} \ar[rr] && Z \ar[ru]^{\sim} \\
  }
$$
in which we suppose (inductively) that $\eta_{Z'}:Z' \to GF(Z')$ is a weak equivalence between cofibrant objects in $\mathcal{E}$. 
Since $X,Y,T(I)$  and $Z'$ are cofibrant, it follows from the pushout-product axiom that the front square is a (homotopical) cell attachment in the sense of Definition \ref{defcellatt}. It suffices thus to prove that the back square is a homotopical cell attachment as well.
But this follows from hypothesis $(d)$ since the back square is the image under $G$ of a free cell attachment in $\mathcal Alg_T$ and hence, using that $VG=U_T$, a homotopical cell attachment in $\mathcal{E}$.




Therefore, by the generalized Reedy patching Lemma \ref{lemReedGen}, $\eta_Z:Z\to GF(Z)$ is a weak equivalence between cofibrant objects as required for the inductive step. Any cellular $T(I)$-module is obtained from the initial $T(I)$-module by (possibly transfinite) composition of cobase changes of the aforementioned kind.
It is here that we need that the generating cofibrations of $\mathcal E$ have cofibrant domains. Therefore, Reedy's telescope Lemma \ref{lemTeles} implies that $\eta_Z$ is a weak equivalence for all cellular $T(I)$-modules. Finally, any cofibrant $T(I)$-module is retract of a cellular one, so that $\eta_Z$ is a weak equivalence for all cofibrant $T(I)$-modules $Z$ as required.



\end{proof}

\section{$\Gamma$-spaces, $\Gamma$-rings and $\Gamma$-theories}\label{GspGrGth}

The category of $\Gamma$-spaces was introduced by Segal \cite{SECCT}, who showed that it has a homotopy category equivalent to the stable homotopy category of connective spectra. Bousfield and Friedlander \cite{BouHT} considered a bigger category of $\Gamma$-spaces in which the ones introduced by Segal appeared as the special $\Gamma$-spaces. Their category admits a closed simplicial model category structure with a notion of stable equivalences giving rise again to the homotopy category of connective spectra. Then Lydakis \cite{LydSG} showed that $\Gamma$-spaces admit internal function objects and a symmetric monoidal smash product with good homotopical properties.

In this section we give some basic definitions and results of the category of $\Gamma$-spaces $\mathcal{GS}$.

Consider the simplicial category $\Delta$ with objects  $\left[ n\right] =\left\lbrace 0,1,...,n\right\rbrace $, for $n\geq 0$, and morphisms the maps $f:\left[ n\right]\rightarrow \left[ k\right]$ such that $x\leq y$ implies  $f(x)\leq f(y)$.

A simplicial set is a functor $X:\Delta^{op} \rightarrow Set$.

Simplicial sets and morphisms of simplicial sets, which are simply natural transformations of functors, constitute the category of simplicial sets, written $SSet$.

Consider the category $\Gamma^{op}$ where the objects are finite sets $n=\left\lbrace 0,1,...,n\right\rbrace $, for $n\geq 0$, and morphisms are the maps of sets which send 0 to 0. 
The category $\Gamma^{op}$ is the opposite of Segal's category $\Gamma$ (\cite{SECCT}), cf. the proof of Proposition \ref{propGamPG}. 

A $\Gamma$-space is a functor $A:\Gamma^{op} \rightarrow \mathcal{E}$ such that $A\left(0 \right)=* $. 
The category $\mathcal{E}$ is either Top or $SSet$.
$\Gamma$-spaces and morphisms of $\Gamma$-spaces, which are simply natural transformations of functors, constitute the category of  $\Gamma$-spaces, written $\mathcal{GS}$.
We have been careful to keep our arguments general enough so as to be independent of the particular choice of monoidal model category $\mathcal{E}$ for the homotopy theory of spaces. The formal properties we need are those appearing in the notion of solid monoidal model category $\mathcal{E}$ in (\cite{BMENRC}, Chapter 7). In particular the latter imply that the category of all functors $\Gamma^{op} \to \mathcal{E}$ is again a monoidal model category with pointwise weak equivalences, and cofibrations as well as fibrations of Reedy type, and with symmetric monoidal structure induced by Day convolution. This monoidal model structure restricts in a straighforward way to the full subcategory of those functors $A:\Gamma^{op} \to \mathcal{E}$ such that $A(0)=*$. If $\mathcal E=SSet$ or $\mathcal E=Top$ the resulting model structure is precisely Bousfield-Friedlander's strict model structure on $\Gamma$-spaces, equipped with the smash-product of Lydakis \cite{LydSG}.

Since 0 is a zero object in $\Gamma$, a $\Gamma$-space actually takes values in $\mathcal{E}_{*}$, the pointed category of based objects in $\mathcal{E}$. Moreover, the category of $\Gamma$-spaces is itself pointed by the representable $\Gamma$-space $\Gamma^0=\Gamma(-,0)$.

We assume here (as is the case for $SSet$ and $Top$) that the monoidal structure of $\mathcal{E}$ is given by the cartesian product. A based space is then an object $X$ of $\mathcal{E}$ together with a map $*\rightarrow X$ where $*$ denotes a terminal object of $\mathcal{E}$ (which serves at the same time as unit for the monoidal structure of $\mathcal{E}$). In particular, the category $\mathcal{E}_{*}$ of based spaces is again a monoidal model category with monoidal structure given by the smash-product $(X,*)\wedge (Y,*)=X\times Y /(X\times *)\cup(*\times Y)$. The category $\mathcal{E}_{*}$ is pointed by $*$.
The category of sets embeds into $\mathcal{E}$ by the functor which takes a set $X$ to the coproduct $\coprod_XI$ of $X$ copies of the unit $I$. For $\mathcal{E}=SSet$ (resp. $\mathcal{E}=Top$) this identifies sets with discrete simplicial sets (resp. discrete topological spaces). In particular, any representable $\Gamma$-set $\Gamma^k=\Gamma(-,k)$ may be considered as a discrete $\Gamma$-space. 

A $\Gamma$-space $A:\Gamma^{op}\rightarrow \mathcal{E}_{*}$ can be prolonged to an endofunctor $\underline{A}:\mathcal{E_{*}} \rightarrow \mathcal{E_{*}}$ by enriched left Kan extension along the canonical inclusion $\Gamma^{op}\rightarrow \mathcal{E_{*}}$.
This enriched Kan extension can be expressed as a coend, which was the way Segal \cite{SECCT} originally proceeded (cf. also Berger \cite{BIWPSCILS}, Section 2.6). More precisely, each based space $(X,*)$ induces a functor $X^-:\Gamma\rightarrow \mathcal{E}$ which takes $n$ to the $n$-fold cartesian product $X^n$ (observe the variance!).
The based space $\underline{A}(X)$ is then given by the coend $A\otimes_{\Gamma} X^-$, i.e. as a canonical quotient of $\coprod_{n\geq 0}A(n)\times X^n$ in $\mathcal{E}$.

It is now of fundamental importance that the endofunctor $\underline{A}$ of $\mathcal{E}_{*}$, associated to any $\Gamma$-space $A$, is a strong endofunctor of $\mathcal{E}_{*}$, i.e. it comes equipped with a strength $X \wedge \underline{A}(Y)\to\underline{A}(X \wedge Y)$. The latter is induced by the aforementioned coend formula together with a canonical trinatural transformation $X \wedge (Y^-)\rightarrow (X\wedge Y)^-$.

It follows that there are \emph{two} monoidal structures on $\Gamma$-spaces. A symmetric monoidal \emph{smash}-product $A\wedge B$, induced from the smash-product on $\Gamma^{op}$ by Day convolution, cf. Section \ref{SSmDayC}, and a non-symmetric \emph{circle}-product $A\circ B$ induced by the composition of the associated endofunctors of $\mathcal{E}_{*}$. Both monoidal structures share the same unit, namely the representable $\Gamma$-space $\Gamma^1=\Gamma(-,1)$.

More precisely, the smash-product $A\wedge B$ is characterized by the property that maps $A\wedge B \rightarrow C$ correspond one-to-one to binatural families of based maps $A(m)\wedge B(n)\rightarrow C(m\wedge n)$. The circle product $A\circ B$ is characterized by a binatural isomorphism $\underline{A\circ B}\cong\underline{A}\circ\underline{B}$, in particular $\underline{A\circ B}(n)=\underline{A}(B(n))$.

This leads to the following two definitions:

\begin{definition}
A $\Gamma$-ring is a monoid in $\Gamma$-spaces for the $\wedge$-product.
\end{definition}

\begin{definition}
A $\Gamma$-theory is a monoid in $\Gamma$-spaces for the $\circ$-product. 
\end{definition}

\begin{lem}\label{defassmap} 
There is a binatural map (the so called assembly map) $$A\wedge B\rightarrow A \circ B$$ from the smash product of $\Gamma$-spaces to the circle product of $\Gamma$-spaces which is in a suitable sense associative and unital. 
Formally, the identity functor is a lax monoidal functor from $(\mathcal{GS},\wedge, \Gamma^{1})$ to $(\mathcal{GS},\circ,\Gamma^{1})$. 
\end{lem}
\begin{proof}
The costrength of $\underline{A}$ and the strength of $\underline{B}$ induce maps
$$A(n)\wedge B(m) \rightarrow \underline{A}(n\wedge B(m))\rightarrow \underline{A}(B(n\wedge m))$$ The characterizations of the smash and circle products then give the desired assembly map.
\end{proof}

Consequently, every $\Gamma$-theory A gives rise to a $\Gamma$-ring $A^S$ (with same underlying $\Gamma$-space) by pulling back the multiplication of the theory along the assembly map.
$\Gamma$-theories correspond precisely to strong monads on based spaces which are determined (through enriched left Kan extension) by their values on finite based sets. 
Schwede \cite{SchSHAT} shows actually that $\Gamma$-theories can be considered as pointed algebraic theories in $\mathcal{E}$. 
By Proposition \ref{PropStDay} the strong monad $\underline{A}$ on $\mathcal{E}_*$ induces a strong monad $(A\circ-)$ on $\Gamma$-spaces and hence,
by Proposition \ref{propmonTI}, an endomorphism monoid $A=A\circ \Gamma^1$ in $\Gamma$-spaces. This construction is consistent with the aforementioned construction of the $\Gamma$-ring $A^S$ as follows from:

\begin{prop}\label{propGring}
For each $\Gamma$-theory A, the endomorphism monoid of the associated strong monad $(A\circ-)$ on $\Gamma$-spaces may be identified with the $\Gamma$-ring $A^S$ induced by the assembly map.

\end{prop}
\begin{proof}
The underlying $\Gamma$-space is $A$ in both cases. It remains to be shown that the multiplication coincides as well. 
This follows from Lemma \ref{lemMonStr} since the latter says that the multiplication of the endomorphism monoid $A=A\circ\Gamma^1$ is given by $$A \wedge A \to A\circ A\to A$$ 
where the first map is induced by the strength of the monad $(A\circ-)$ and the second by the theory multiplication.
This is precisely the multiplication $A^S \wedge A^S \to A^S$ induced by the assembly map, cf. the proof of Lemma \ref{defassmap}.
\end{proof} 

\section{The stable model structure on $\Gamma$-spaces}\label{sAsmPpax}

The category of $\Gamma$-spaces has a strict model structure of Bousfield and Friedlander, where the weak equivalences are pointwise, cofibrations and fibrations are Reedy-like (cf. \cite{BouHT,LydSG, BMENRC}).
In particular, the category of $\Gamma$-spaces for the strict structure is a monoidal model category by a result of Berger-Moerdijk on generalized Reedy categories (cf. \cite{BMENRC}, Theorem 7.6 and Example 7.7b).

Nevertheless, the strict model structure does not permit an application of our homotopical Morita theorem since the assembly map is in general not a weak equivalence in the strict model structure. Nevertheless, it was one of the insights of Lydakis \cite{LydSG} that the assembly map is a \emph{stable} equivalence in quite generality. So, only in the stable world, the category of models for a $\Gamma$-theory can possibly become Quillen equivalent to the category of modules over the associated $\Gamma$-ring. That this is indeed the case has been shown by Schwede \cite{SchSHAT}. We get Schwede's theorem as a formal consequence of our homotopical Morita theorem.

In order to render our thesis as self-contained as possible we include here a new proof of Lydakis' theorem on the assembly map. Our proof has the advantage of being independent of the choice of the cartesian model $\mathcal{E}$ for the homotopy theory of spaces. It is also interesting to observe that Lydakis' theorem implies that the stable model category of $\Gamma$-spaces is a monoidal model category, which is the other hypothesis for our homotopical Morita theorem to be applicable. 
In course of proving Lydakis' theorem on the assembly map, we relate in an interesting way the "$\Gamma$-spheres" $\Gamma^n/\partial\Gamma^n$ to the $\Gamma$-sets which represent the $n$-fold smash product of based spaces, cf. Proposition \ref{propGamPG}. 

The stable model structure of Bousfield and Friedlander is obtained by left Bousfield localization of the strict model structure. 
More precisely, while cofibrations are fixed, one extends weak equivalences. 
Bousfield and Friedlander \cite{BouHT} proved that such a localization exists by taking as new weak equivalences precisely the stable equivalences.
Schwede \cite{SStHomAlgGS} considered slightly different strict and stable model structures on $\Gamma$-spaces, the so called Q-model structures with less cofibrations and more fibrations, but the same classes of weak equivalences.

We recall here the \emph{stable model structure} on $\Gamma$-spaces (as defined by Bousfield and Friedlander) by specifying its cofibrations and 
weak equivalences. Since $\Gamma^{op}$ is a generalized Reedy category in the sense of Berger and Moerdijk, each $\Gamma$-space $A$ has a skeletal filtration $$...\rightarrow sk_{n-1}(A) \rightarrow sk_n (A) \rightarrow sk_{n+1}(A)\rightarrow...$$ 
where $sk_n(A)$ denotes the sub-$\Gamma$-object of $A$ generated by the based spaces $A(0),\dots,A(n)$. 
The $n$-th latching object $L_n(A)$ is then defined to be the based space $(sk_{n-1}(A))(n)$. 
The latter comes equipped with a $\Sigma_n$-equivariant map $L_n(A)\rightarrow A(n)$, where $\Sigma_n$ denotes the automorphism-group of $n$ in $\Gamma^{op}$.

A map of $\Gamma$-spaces $A \rightarrow B$ is then called a \emph{cofibration} if the induced comparison map $A(n)\cup_{L_n(A)}L_n(B)\rightarrow B(n)$ is a $\Sigma_n$-cofibration in $\mathcal E$
(which means that it has the left lifting property with respect to $\Sigma_n$-equivariant acyclic fibrations) for each $n \geq 1$. 
A discrete $\Gamma$-space $A$ is cofibrant if and only if for all $n\geq 1$, $\Sigma_n$ acts freely on $A(n)-L_n(A)$. 
In particular, all subobjects of the representable $\Gamma$-spaces $\Gamma^n$ are cofibrant.
This property fails for Schwede's strict and stable $Q$-model structure, which is the main reason for which we are forced to use 
Bousfield and Friedlander's strict and stable model structures instead.

A map of $\Gamma$-spaces $A \rightarrow B$ is called a stable equivalence if its \emph{spectrification} $\Phi A\rightarrow \Phi B$ is a stable equivalence of spectra. We recall that a \emph{spectrum} in $\mathcal E$ consists of a sequence $(X_n)$ of objects in $\mathcal{E}_*$ equipped with structural maps $S^1\wedge X_n\rightarrow X_{n+1}$ ($n\geq 0$). 
Any \emph{strong} endofunctor $\underline{A}$ takes spectra to spectra. Indeed, $(\underline{A}(X_n))$ comes equipped with structural maps $$S^1\wedge \underline{A}(X_n)\rightarrow \underline{A}(S^1\wedge X_n)\rightarrow \underline{A}(X_{n+1})$$

Since $\mathcal{E}_*$ is a pointed model category, there is a canonical suspension functor for $\mathcal{E}_*$. In particular, $\mathcal{E}_*$ possesses a \emph{sphere-spectrum} $\mathbb{S}$ such that $\mathbb{S}_n$ is a model for the $n$-sphere $S^n$. The spectrification functor of Segal \cite{SECCT} is then defined by $\Phi A=\underline{A}(\mathbb{S})$.

Recall also that a stable equivalence of spectra $(X_n)\rightarrow (Y_n)$ is defined to be a map of spectra that induces an isomorphism on \emph{stable homotopy groups}. 

Bousfield and Friedlander show that the category of $\Gamma$-spaces equipped with these cofibrations and weak equivalences (i.e. stable equivalences) forms a model category. We shall call a map of $\Gamma$-spaces, which is at once a cofibration and a stable equivalence, a \emph{stably acyclic cofibration}. We define the stable homotopy groups of a $\Gamma$-space through its spectrification: $\pi_n^{st}(A)=\pi_n^{st}(\Phi A)$. 

\begin{lem}\label{lemLESHG}
Every cofibration of $\Gamma$-spaces $A \rightarrowtail B$ induces a long exact sequence of stable homotopy groups
$$ ....\rightarrow \pi^{st}_n(A)\rightarrow \pi^{st}_n(B)\rightarrow \pi^{st}_n(B/A)\rightarrow...$$ 
In particular,
\begin{enumerate}
 \item The cofibration is a stably acyclic cofibration if and only if its cofiber is stably acyclic (i.e. the stable homotopy groups of $B/A$ are trivial); 
 \item For every natural transformation of cofiber sequences 
\begin{diagram}{}
\label{assff}
$$
\shorthandoff{;:!?}
\xymatrix @!0 @C=2.5cm @R=2cm{\relax 
 A \ar[r]^{f} \ar[d]_{\alpha} &  B \ar[r]^{g} \ar[d]_{\beta} & B/A \ar[d]_{\gamma} \\ 
 A' \ar[r]^{f'} &  B' \ar[r]^{g'} & B'/A'
}
$$
\end{diagram}
if two among $(\alpha,\beta,\gamma)$ are stable equivalences then so is the third.
\end{enumerate}
\end{lem}

\begin{proof}
The main statement follows from the analogous statement for cofiber sequences of spectra because Segal's spectrification functor $\Phi:(\mathcal {GS})\rightarrow (Spectra)$ is a left Quillen functor with respect to a suitable stable model structure on spectra (cf. \cite{BouHT} and \cite{SStHomAlgGS}, Lemma 1.3.).
Then $(a)$ and $(b)$ are immediate corollaries of well known results in homological algebra (for example $(b)$ is a consequence of the Five Lemma).
\end{proof}

For any $\Gamma$-space $A$ and any based space $X$ there is a $\Gamma$-space $X\wedge A$ defined by $(X \wedge A)(n)=X\wedge A(n)$. 
In other words, the category of $\Gamma$-spaces is tensored over the category $\mathcal {E}_*$ of based spaces. 
This structure is compatible with Segal's spectrification functor in the following sense: there is a canonical map of spectra $X\wedge\Phi(A)\to\Phi(X\wedge A)$ which is a stable equivalence for any cofibrant $\Gamma$-space A and any cofibrant based space $X$, cf. Lemma 4.1 of Bousfield-Friedlander \cite{BouHT}.

The following lemma is a fundamental tool (also used by Lydakis \cite{LydSG}, 3.11).
\begin{lem}\label{lemQuotisom}
For any cofibrant $\Gamma$-space A and any $n>0$, there is a cofibrant based space $\tilde{A}(n)$ such that the quotient 
$sk_n(A)/sk_{n-1}(A)$ is isomorphic (as $\Gamma$-space) to $\tilde{A}(n)\wedge \Gamma^n/\partial\Gamma^n$, where the boundary $\partial\Gamma^n$ is given by $sk_{n-1}(\Gamma^n)$.
\end{lem}
\begin{proof}
It follows from the definition of the skeletal filtration of a $\Gamma$-space (cf. \cite{BouHT}, \cite{LydSG}, \cite{BMENRC}) that $sk_n(A)$ is obtained from $sk_{n-1}(A)$ by attaching the cofibration 
$$(L_n(A)\wedge \Gamma^n)\cup(A(n)\wedge \partial\Gamma^n)\rightarrow A(n)\wedge \Gamma^n$$ along a canonical attaching map to $sk_{n-1}(A)$. 
Taking cofibers, we get an isomorphism (confer the proof of Corollary \ref{propBFMMC})
$$(A(n)/L_n(A))\wedge (\Gamma^n/\partial\Gamma^n)\cong sk_n(A)/sk_{n-1}(A).$$ 
Since $A$ is a cofibrant $\Gamma$-space the inclusion  $L_n(A) \rightarrow A(n)$  is a cofibration of $\Sigma_n$-spaces; therefore we can take the cofibrant quotient $A(n)/L_n(A)$ for $\tilde{A}(n)$.
\end{proof}

\begin{prop}\label{propGamPG}
For each $n>0$ one has:
\begin{enumerate}
 \item The endofunctor $\underline{\Gamma^n}$ takes a based space $X$ to the $n$-fold cartesian product $X^n$; 
 \item The boundary $\partial\Gamma_n=sk_{n-1}(\Gamma^n)$ contains a uniquely determined $\Gamma$-subset $\partial_{out}\Gamma^n$ ("outer" boundary) with the property that the endofunctor associated to the quotient $\Gamma^n/\partial_{out}\Gamma^n$ takes a based space $X$ to the $n$-fold smash product $X^{\wedge n}$;
 \item The poset of \emph{monogenic} subobjects of $\Gamma^n$ not contained in $\partial_{out}\Gamma^n$ is anti-isomorphic to the partition 
lattice $\Pi_n$ of an $n$-element set;
 \item The quotient map $\Gamma^n\to\Gamma^n/\partial_{out}\Gamma^n$ takes the lattice of monogenic subobjects of $\Gamma^n$ to an isomorphic lattice of subobjects of $\Gamma^n/\partial_{out}\Gamma^n$. 
Each of these image-subobjects is isomorphic to a $\Gamma^k/\partial_{out}\Gamma^k$ where $k$ is the number of pieces of the corresponding partition;
 \item There is a cofiber sequence of (discrete) $\Gamma$-spaces 
$$\partial\Gamma^n/\partial_{out}\Gamma^n\to\Gamma^n/\partial_{out}\Gamma^n\to\Gamma^n/\partial\Gamma^n$$
where the first term is canonical colimit (over $\Pi_n^*$) of $\Gamma$-spaces of the form $\Gamma^k/\partial_{out}\Gamma^k$. 
\end{enumerate}
\end{prop}

\begin{proof}~\\
\begin{enumerate}
 \item $n$ is the n-fold coproduct of $1$ in $\Gamma^{op}$ so that $\Gamma^n$ is the $n$-fold cartesian product of $\Gamma^1$ in the category of (discrete) $\Gamma$-spaces. The coend formula for the associated endofunctors commutes with these finite products.
 \item We shall use Segal's original description of the operators of $\Gamma$,
 i.e. an operator $\phi:k\to n$ is given by a $k$-tuple $(\phi_1,\dots,\phi_k)$ of pairwise disjoint subsets of $\{1,\dots,n\}$. 
Then the Reedy-boundary of $\Gamma^n=\Gamma(-,n)$ is given by $$\partial\Gamma^n(k)=\{\phi:k\to n\text{ non-invertible}\}$$
We define a second smaller boundary by $$\partial_{out}\Gamma^n(k)=\{\phi:k\to n\text{ non-covering}\}$$
where an operator $\phi=(\phi_1,\dots\phi_k):k\to n$ is \emph{non-covering} if $\phi_1\cup\cdots\cup\phi_k\not=\{1,\dots,n\}$.  
$\partial_{out}\Gamma^n$ is a union of $n$ copies of $\Gamma^{n-1}$ where the inclusions are induced by the $n$ outer face operators 
$$(\{1\},\dots,\{i-1\},\{i+1\},\dots,\{n\}):n-1\to n$$ in $\Gamma$. 
By $(a)$, each inclusion induces, on the level of associated endofunctors, a specific inclusion of the $(n-1)$-fold cartesian product into the $n$-fold cartesian product, namely by specifying which of the $n$ factors in the image is at the base point.
It follows that the quotient $\Gamma^n/\partial_{out}\Gamma^n$ defines the endofunctor $X\mapsto X^{\wedge n}$.
 \item We call a $\Gamma$-set \emph{monogenic} if it can be generated by a single element. 
The monogenic subobjects of $\Gamma^n$ not contained in $\partial_{out}\Gamma^n$ correspond thus to non-degenerate elements of $\Gamma^n$ which do not belong to $\partial_{out}\Gamma^n$. 
The latter are precisely the partitions of $\{1,\dots,n\}$ where the number of pieces is given by the domain of the operator $\phi=(\phi_1,\dots,\phi_k):k\to n$. 
A partition $\phi$ into $k$ pieces is refined by a partition $\psi$ into $l$ pieces if and only if there is a $\Gamma$-operator $\rho:k\to l$ such that $\phi=\psi\rho$; this means that the subobject generated by $\phi$ is contained in the subobject generated by $\psi$ if and only if the partition $\psi$ refines the partition $\phi$.  
 \item The first assertion follows from the fact that the generator of each of these subobjects does not belong to the outer boundary. 
The second assertion follows from the identification $\partial_{out}\Gamma^n\cap(\phi)=\partial_{out}\Gamma^k$ for each covering $\Gamma$-operator $\phi:k\to n$. 
 \item It is clear that the boundary $\partial\Gamma^n$ contains the outer boundary $\partial_{out}\Gamma^n$ so that the left arrow is well defined and monic. 
Since its quotient is cofibrant the left arrow is a cofibration. 
The second assertion follows then from $(d)$ and the fact that the boundary $\partial\Gamma^n$ is the union of proper, monogenic subobjects as considered in $(c)$. 
\end{enumerate}
\end{proof}

\begin{rem}
For $n=2$, the outer boundary $\partial_{out}\Gamma^2$ represents the endofunctor which sends $X$ to the wedge $X\vee X$. 
It is the union of two copies of $\Gamma^1$ (representing the identity functor) along the two face operators $(\{1\}):1\to 2$ and $(\{2\}):1\to 2$. 
The whole boundary $\partial\Gamma^2$ contains a third copy of $\Gamma^1$ glued in by the covering face operator $(\{1,2\}):1\to 2$. 
We therefore get $\partial\Gamma^2/\partial_{out}\Gamma^2=\Gamma^1/\partial_{out}\Gamma^1=\Gamma^1$. 
The cofiber sequence above reads then as follows: $$\Gamma^1\to\Gamma^2/\partial_{out}\Gamma^2\to\Gamma^2/\partial\Gamma^2$$
This shows (in virtue of Proposition \ref{propGamPG}$(b)$ and Lemma \ref{lemLESHG}(b)) that the endofunctor defined by the $\Gamma$-sphere $\Gamma^2/\partial\Gamma^2$ preserves stable equivalences between connective spectra. 
The same is true for the higher $\Gamma$-spheres $\Gamma^n/\partial\Gamma^n$ and indeed for any cofibrant $\Gamma$-space $A$, more precisely:
\end{rem}

\begin{prop}\label{propCircse}
For each cofibrant $\Gamma$-space $A$, the left circle-product $A\circ-$ takes stable equivalences between cofibrant $\Gamma$-spaces to stable equivalences.
\end{prop}

\begin{proof}
By the telescope lemma of Reedy, cf. Lemma \ref{lemTeles}, it suffices to establish the property for each $sk_n(A),\,n\geq 0$. 
For $n=0$ it holds. By Lemma \ref{lemLESHG} $(b)$ above and an induction on $n$, it then suffices to show the property for 
$sk_n(A)/sk_{n-1}(A)=\tilde{A}_n\wedge \Gamma^n/\partial\Gamma^n$ (by Lemma \ref{lemQuotisom} above). 
Smashing by a cofibrant based space preserves stable equivalences between cofibrant spectra so that it is finally sufficient to prove that $\Gamma^n/\partial\Gamma^n\circ-$ has the required property.

In other words, we have to show that for any stable equivalence of cofibrant $\Gamma$-spaces $B\rightarrow C$, the map of spectra 
$$\underline{\Gamma^n/\partial\Gamma^n}(\underline{B}(\mathbb{S}))\rightarrow\underline{\Gamma^n/\partial\Gamma^n}(\underline{C}(\mathbb{S}))$$
is a stable equivalence knowing that $\underline{B}(\mathbb{S})\rightarrow \underline{C}(\mathbb{S})$ is a stable equivalence of connective cofibrant spectra.
By Proposition \ref{propGamPG}$(e)$ we know that $\Gamma/\partial\Gamma^n$ is a quotient of $\Gamma^n/\partial_{out}\Gamma^n$ and that $\partial\Gamma^n/\partial_{out}\Gamma^n$ itself is a "nice" colimit of $\Gamma$-spaces of the form $\Gamma^k/\partial_{out}\Gamma^k$. 
Applying Lemma \ref{lemLESHG}$(b)$ and the fact that the circle-product preserves colimits in the \emph{first} variable, it is thus sufficient to show that the strong endofunctors $\underline{\Gamma^k/\partial_{out}\Gamma^k}$ preserve stable equivalences between \emph{connective} cofibrant spectra. 
According to Proposition \ref{propGamPG} $(b)$ these endofunctors are given by $k$-fold smash product.
For $k=1$ the $1$-fold smash product (i.e. the identity functor) certainly has the required property. 
For $k\geq 2$, the levelwise $k$-fold smash product of a \emph{connective} spectrum $X=(X_n)$ has the property that $X_n^{\wedge k}$ is $(nk-1)$-connected, and hence, the connectivity of $X_n^{\wedge k}$ minus $n$ goes to infinity when $n$ goes to $\infty$. 
Thus, for $k\geq 2$ and any connective spectrum $X$, $\underline{\Gamma^k/\partial_{out}\Gamma^k}(X)$ is stably trivial. 
It follows that for $k\geq 2$, the endofunctors $\underline{\Gamma^k/\partial_{out}\Gamma^k}$ also have the required property.
\end{proof}

Beside the preceding proposition there is one other ingredient for our proof of Lydakis's theorem on the assembly map. 
It concerns a homotopical property of reduced bifunctors $F:\Gamma^{op}\times\Gamma^{op}\to\mathcal E$ which we call \emph{bi-$\Gamma$-spaces}. 
A bifunctor $F$ is called \emph{reduced} if $F(0,n)=F(m,0)=*$ where $*$ is the terminal object of $\mathcal E$. The following property is reminiscent of the well-known property of the diagonal of a bisimplicial set (or space). 
It differs insofar as we are not using the diagonal, but left Kan extension along the \emph{smash-product} $s:\Gamma^{op} \times \Gamma^{op}\to\Gamma^{op}:(m,n)\mapsto mn$.

Observe that the product of two generalized Reedy-categories is again a generalized Reedy category in the sense of Berger-Moerdijk \cite{BMENRC}. 
Therefore, it makes sense to speak of cofibrant bi-$\Gamma$-spaces.

\begin{prop}\label{propKanextEqGS}
Let $\phi:F(\cdot,\cdot)\to G(\cdot,\cdot)$ be a map of cofibrant bi-$\Gamma$-spaces and assume that either, for each $m$, $\phi(m,\cdot):F(m,\cdot)\to G(m,\cdot)$ is a stable equivalence of $\Gamma$-spaces or, for each $n$, $\phi(\cdot,n):F(\cdot,n)\to G(\cdot,n)$ is a stable equivalence of $\Gamma$-spaces.

Then, the left Kan extension $s_!\phi:s_!F\to s_!G$ is again a stable equivalence of $\Gamma$-spaces.
\end{prop}

\begin{proof}
The idea is quite simple: we endow the category $\mathcal E_{red}^{{\Gamma^{op}}\times{\Gamma^{op}}}$ with a Quillen model structure such that 
\begin{enumerate}
 \item[\textit{(i)}] The weak equivalences are precisely the pointwise stable equivalences (with respect to one of the two variables);
 \item[\textit{(ii)}] The adjoint pair $(s_!,s^*)$ is a Quillen pair between this model structure on $\mathcal E_{red}^{{\Gamma^{op}}\times{\Gamma^{op}}}$ and the stable model structure on $\mathcal E_{red}^{\Gamma^{op}}$. 
\end{enumerate}

The proposition follows then from Brown's Lemma \ref{lemBrowns}(a).

Point $(i)$ is immediate since $\mathcal E_{red}^{{\Gamma^{op}}\times{\Gamma^{op}}}$ can be identified (in two ways, depending on the ordering of the variables) with $(\mathcal E_{red}^{\Gamma^{op}})_{red}^{\Gamma^{op}}$. The model structure, we are interested in, is the "Reedy" model structure on (reduced) $\Gamma$-objects in the stable model category of (reduced) $\Gamma$-spaces. 
It is thus a mixture between the stable (inside) and the strict (outside) model structure. Nevertheless, its existence follows from Bousfield-Friedlander \cite{BouHT} and Berger-Moerdijk \cite{BMENRC}.

Point $(ii)$ is more subtle since, a priori, we only know that $s^*$ takes stable equivalences (resp. stable fibrations) to pointwise stable equivalences (resp. pointwise stable fibrations). 
Therefore, $s^*$ preserves the respective notions of weak equivalence, but it has to be shown that $s^*$ also preserves the respective notions of fibration, i.e. that $s^*$ takes stable fibrations to Reedy fibrations for the "mixed" model structure on $(\mathcal E_{red}^{\Gamma^{op}})_{red}^{\Gamma^{op}}$.

In order to establish the latter property we shall use an idea of Barwick, cf. Lemma 3.21 of \cite{BarLRModC}. 
The latter is formulated for a functor $f:\mathcal R\to\mathcal S$ of strict Reedy categories and implies that $f_!:\mathcal E^\mathcal R\leftrightarrows\mathcal E^\mathcal S:f^*$ is a Quillen adjunction between the respective Reedy model structures provided that for each object $\sigma$ of $\mathcal S$ the inverse part of the comma category $f/\sigma$ is a coproduct of categories with terminal object. 
There is an analogous statement for a functor of generalized Reedy categories, replacing terminal by weakly terminal object. The proof idea consists roughly in showing that certain relevant matching maps in $\mathcal E^\mathcal R$ derive from the matching maps in $\mathcal E^\mathcal S$ by taking products.

It is now straightforward to show that the smash product functor $s:\Gamma^{op}\times\Gamma^{op}\to\Gamma^{op}$ has this property.
Indeed, one shows that for each object $m\wedge n\to p$ in $s/p$ there is an object $m'\wedge n'\to p$, weakly terminal in the inverse part of $s/p$, such that former map factors through the latter map by a pair of surjections $m\to m'$ and $n\to n'$ in $\Gamma^{op}$.

This shows that $(s_!,s^*)$ is a Quillen adjunction for the Reedy model structures on $\mathcal E_{red}^{{\Gamma^{op}}\times{\Gamma^{op}}}$ and $\mathcal E_{red}^{\Gamma^{op}}$. 
It also gives a Quillen adjunction for the mixed model structure on $\mathcal E_{red}^{{\Gamma^{op}}\times{\Gamma^{op}}}$ and the stable model structure on $\mathcal E_{red}^{\Gamma^{op}}$ since a product of stable fibrations is again a stable fibration. 
\end{proof}

\begin{theoreme}(Lydakis \cite{LydSG})\label{thasmap}
For any pair A,B of cofibrant $\Gamma$-spaces the assembly map $A\wedge B\rightarrow A \circ B$ is a stable equivalence. 
\end{theoreme}

\begin{proof}
We define three different cofibrant bi-$\Gamma$-spaces, $A\Box_1 B$, $A\Box_2 B$ and $A\Box_3 B$ by the formulas: 
$$(A\Box_1 B)(m,n)=A(m)\wedge B(n)\quad (A\Box_2 B)(m,n)=\underline{A}(m\wedge B(n))$$
$$\quad (A\Box_3 B)(m,n)=(A\circ B)(mn)=\underline{A}(B(mn))$$
There are canonical maps of cofibrant bi-$\Gamma$-spaces
$$A\Box_1B\to A\Box_2B\to A\Box_3 B$$ 
The first is a stable equivalence (for fixed n) with respect to the first variable by Lemma 4.1 of Bousfield-
Friedlander \cite{BouHT}, the second is a stable equivalence (for fixed m) with respect to the second variable by Lemma 4.1 of Bousfield-Friedlander \cite{BouHT} together with our 
Proposition \ref{propCircse}. It follows then from Proposition \ref{propKanextEqGS} that left Kan extension along the smash-product gives a stable equivalence:
$$s_!(A\Box_1B)\to s_!(A\Box_3 B)$$
The $\Box_1$-product is the external smash-product of $\Gamma$-spaces so that $s_!(A\Box_1 B)$ is the internal smash-product of $\Gamma$-spaces (obtained by Day convolution, cf. Definition\ref{defDayconv}). 
The $\Box_3$-product is just $s^*(A\circ B)$ so that $s_!(A\Box_3 B)=s_!s^*(A\circ B)=A\circ B$, since $s^*$ is fully faithful. 
The constructed stable equivalence $A\wedge B\to A\circ B$ coincides with the assembly map, cf. the proof of Lemma \ref{defassmap}.

\end{proof}


\begin{cor}\label{propBFMMC}
Bousfield and Friedlander's stable model category of $\Gamma$-spaces is a monoidal model category with cofibrant unit.
\end{cor}

\begin{proof}
The only axiom to be checked is Hovey's pushout-product axiom. 
For this consider Diagram \ref{diagppa} of Definition \ref{defmonmodc}. 
Since the non-localized category is monoidal, $f \Box g$ is a cofibration and it is sufficient to prove that $f\Box g$ is a stable equivalence if
either $f$ or $g$ is. By Lemma \ref{lemLESHG} $(a)$ this amounts to proving that the cofiber of $(f\Box g)$ is stably acyclic if either the cofiber of $f$ or the cofiber of $g$ is stably acyclic.
\emph{But the cofiber of $f\Box g$ is the smash product of the cofibers of $f$ and of $g$}. Therefore it suffices to prove that the smash product of two cofibrant $\Gamma$-spaces is stably acyclic as soon as one of the factors is. 
This follows from Lydakis' theorem \ref{thasmap} together with Proposition \ref{propCircse}. 
\end{proof}

\begin{cor}\label{corCofGS}
For each cofibration $X\to Y$ between cofibrant $\Gamma$-spaces and each cofibrant $\Gamma$-space $A$, 
the canonical map $(A\circ Y) / (A\circ X) \to A\circ(Y/X)$ is a stable equivalence. 
\end{cor}

\begin{proof}
The assembly map induces a natural transformation of cofiber sequences
$$
\shorthandoff{;:!?}
\xymatrix @!0 @C=4cm @R=2cm{\relax 
 A\wedge X \ar[r] \ar[d]_{\alpha} &  A\wedge Y \ar[r] \ar[d]_{\beta} & A\wedge (Y/X) \ar[d]_{\gamma} \\ 
 A\circ X \ar[r] & A\circ Y \ar[r] & (A\circ Y)/ (A\circ X)
}
$$
By Theorem \ref{thasmap} and Lemma \ref{lemLESHG} (b) the induced map on the quotients $\gamma$ is a stable equivalence. On the other hand, Theorem \ref{thasmap} also gives a stable equivalence $\delta:A\wedge(Y/X)\to A\circ (Y/X)$. The canonical map $(A\circ Y)/(A\circ X)\to A\circ(Y/X)$ precomposed by $\gamma$ yields $\delta$. The 2 out of 3 property of stable equivalences thus yields the asserted result. 
\end{proof}

\section{Recovering a theorem of Stefan Schwede}\label{sRecResSS}

For the following theorem of Schwede we restrict ourselves to the case where $\mathcal E$ is the category of simplicial sets. 
Therefore, from now on, $\Gamma$-spaces are understood to take values in the category of simplicial sets. 
A $\Gamma$-theory is said to be \emph{well-pointed} if its unit is a cofibration of $\Gamma$-spaces; this implies in particular that the $\Gamma$-theory has an underlying cofibrant $\Gamma$-space.
We begin by showing that the category of $\Gamma$-spaces has the required properties for an application of our homotopical Morita Theorem \ref{thMain} and our simplifying assumptions made in Section \ref{SRealMModCat}.






\begin{prop}\label{propsteq}
For every cofibrant $\Gamma$-space A and every pair X,Y of cofibrant $\Gamma$-spaces there is a stable equivalence
$$(A\circ X) \wedge Y\rightarrow A\circ \left( X \wedge Y\right) $$
\end{prop}

\begin{proof}

For two cofibrant $\Gamma$-spaces $X$ and $Y$, consider the following:
$$A \wedge X \wedge Y\rightarrow \left( A\circ X\right) \wedge Y \rightarrow A\circ \left( X \wedge Y\right) $$
Since by the pushout-product axiom (Definition \ref{defmonmodc}), the smash product $X\wedge Y$ is a cofibrant $\Gamma$-space, the composed arrow $$A \wedge X \wedge Y\rightarrow A\circ \left( X \wedge Y\right) $$ is a stable equivalence by Theorem \ref{thasmap}.

But $A \wedge X \rightarrow A\circ X$ is a stable equivalence of cofibrant $\Gamma$-spaces by Theorem \ref{thasmap} (it can be checked that the circle product of two cofibrant $\Gamma$-spaces is again cofibrant). Then Brown's Lemma \ref{lemBrowns} (a), the pushout-product axiom and Corollary \ref{propBFMMC} imply that smashing with a cofibrant $\Gamma$-space yields
$$A \wedge X \wedge Y\rightarrow \left( A\circ X\right) \wedge Y $$
also a stable equivalence.
Using the two out of three rule, we obtain that the right arrow $$\left( A\circ X\right) \wedge Y \rightarrow A\circ \left( X \wedge Y\right)$$ is a stable equivalence.
\end{proof}



\begin{lem}\label{lemRealExc}
The stable model category $\mathcal GS$ of $\Gamma$-spaces (with space=simplicial set) has the following properties:
\begin{enumerate}
 \item $\mathcal GS$ is endowed with a standard system of simplices whose associated realisation functor is good (cf. Definition \ref{defsimpobj});
 \item $\mathcal GS$ satisfies excision (cf. Definition \ref{defexc});
 \item For any well-pointed $\Gamma$-theory $A$, free cell extensions in $\mathcal GS^{A\circ}$ (cf. proof of Proposition \ref{propForgfuncRE}) have as underlying map a cofibration of $\Gamma$-spaces.
\end{enumerate}
\end{lem}

\begin{proof}~\\
\begin{enumerate}
 \item The Yoneda-embedding $\Delta\to\mathcal SSet$ defines a standard system of simplices for simplicial sets. There are strong symmetric monoidal
left Quillen functors $\mathcal SSet\to\mathcal SSet_*$ (adjunction of base point) and $\mathcal SSet_*\to
\mathcal GS$ (left adjoint of the underlying space functor $A\mapsto A(1)$). According to Berger-Moerdijk \cite{BerBoardVogtRes} Cor. A.14, this provides $\mathcal GS$ with a standard system of simplices. 

The associated realisation functor is good, since the Bousfield-Friedlander cofibrations in $\mathcal GS$ can be characterised as those monomorphisms $X\to Y$ for which the quotient $Y/X$ has the property that the non-degenerate simplices in $\Gamma$-degree $n>0$ have no isotropy for the canonical $\Sigma_n$-action. In particular, for any "intermediate" $\Gamma$-space $Z$ such that $X\subset Z\subset Y$ the two inclusions $X\to Z$ and $Z\to Y$ are also cofibrations. This implies that 
good simplicial objects in the sense of Definition \ref{defsimpobj} are actually Reedy-cofibrant (for the Reedy model structure on simplicial objects in $\mathcal GS$). Therefore, the realisation functor (which is a left Quillen functor with respect to this 
Reedy model structure, cf. \cite{BerBoardVogtRes} Lemma A.8 ) takes weak equivalence between good simplicial objects to weak equivalences by Corollary \ref{corQfWE} (a).

 \item Excision follows directly from Lemma \ref{lemLESHG} (b) by taking $\alpha$ to be the identity.

 \item We have to show that in any pushout diagram

$$
\shorthandoff{;:!?}
 \xymatrix @!0 @C=2.4cm @R=2.4cm{\relax
    A\circ X \ar[r] \ar[d] & W \ar[d] \\
    A\circ Y \ar[r]  & W' \\
  }
$$
in $\mathcal GS^{A\circ}$ the underlying map of $W\to W'$ is a cofibration of $\Gamma$-spaces as soon as $X\to Y$ is so. It follows from Lemma \ref{lemMonAlgT} that the category of algebras $\mathcal GS^{A\circ}$ is isomorphic (over $\mathcal GS$) to the category of reduced $\Gamma$-objects in $\mathcal SSet^{\underline{A}}_*$. Moreover, a cofibration of $\mathcal GS$-spaces is (as already mentioned above) a monomorphism $X\to Y$ such that the quotient $Y/X$ is cofibrant. 
It therefore suffices to show that free $\underline{A}$-extensions of $\underline{A}$-algebras in $\mathcal SSet_*$ are monic, and that $A\circ-$ preserves cofibrant $\Gamma$-spaces.
The circle product of two cofibrant $\Gamma$-spaces is cofibrant so that the second assertion follows from the well-pointedness of $A$. For the first assertion we use that the strong endofunctor $\underline{A}$ may be computed in each simplicial degree seperately. More precisely, the $\Gamma$-theory $A$ defines in each simplicial degree $n$ a discrete $\Gamma$-theory $A_n$ so that for an arbitrary pointed simplicial set $X$ the value $\underline{A}(X)$ in simplicial degree $n$ is given by $\underline{A_n}(X_n)$. 

Consider now a pushout square like above for a monomorphism of pointed simplicial sets $X\to Y$. Since any monomorphism of pointed sets $X_n\to Y_n$ is a split monomorphism (i.e. admits a retraction) the induced map $\underline{A_n}(X_n)\to\underline{A_n}(Y_n)$ is a split monomorphism in $\mathcal Set_*^{\underline{A_n}}$. Split monomorphisms are stable under pushout in any category. Thus, the pushout $W_n\to W'_n$ is a split monomorphism in $\mathcal Set_*^{\underline{A_n}}$ and hence the pushout $W\to W'$ is a monomorphism (not anymore split) in $\mathcal SSet_*^{\underline{A}}$.
\end{enumerate}
\end{proof}

\begin{theoreme}(Schwede \cite{SchSHAT})\label{ThSch}
Each well-pointed $\Gamma$-theory $A$ induces a Quillen equivalence between the category of $(A\circ)$-algebras in $\Gamma$-spaces and the category of $A^S$-modules in $\Gamma$-spaces for a functorially associated $\Gamma$-ring $A^S$.
\end{theoreme}

\begin{proof}
Consider the category of $\Gamma$-spaces $\mathcal{GS}$ with the symmetric monoidal structure of Lydakis. By Corollary \ref{propBFMMC}, $\Gamma$-spaces admit a stable monoidal model structure of Bousfield-Friedlander with cofibrant unit $\Gamma^{1}$. 
Moreover, the category of $\Gamma$-spaces has generating cofibrations with cofibrant domain, since they are of the form 
$$Y_{+}\wedge\partial\Gamma^n\cup X_{+}\wedge\Gamma^n\to Y_{+}\wedge\Gamma^n$$ where $X\to Y$ is a generating cofibration of $\mathcal SSet$.

In order to prove the existence of a transferred model structure for $(A\circ)$-algebras, we will use the remark $(a)$ after Theorem \ref{thMain}. 
Since the monad $A\circ$ preserves reflexive coequalizers (indeed, the functor $X\mapsto (-\mapsto X^-)$ preserves reflexive coequalizers) Proposition \ref{propalgcocom} shows that $\mathcal GS^{(A\circ)}$ is cocomplete. 
Moreover, the forgetful functor preserves filtered colimits and $\mathcal GS$ is locally finitely presentable; it therefore suffices by Theorem \ref{thcalgT} $(b)$ to construct a fibrant replacement functor for $(A\circ)$-algebras. 
Since the forgetful functor commutes with finite products it suffices to construct a finite product preserving fibrant replacement functor for the stable model structure on $\mathcal GS$. Bousfield-Friedlander construct such fibrant replacement functor in \cite{BouHT}. 

The hypothesis that a $\Gamma$-theory $A$ is well-pointed implies the hypothesis $(b)$ of Theorem \ref{thMain}, since $(\Gamma^1\to A)\circ B$ gives a cofibration $B \to A\circ B$ if $A$ is well-pointed and $B$ cofibrant.

By Proposition \ref{propsteq}, we have a stable equivalence:
$$\left(A\circ X\right) \wedge Y\rightarrow A\circ \left( X \wedge Y\right) $$ for X and Y cofibrant $\Gamma$-spaces.

By symmetry of the assembly map, we can do the following switch: 
$$Y \wedge \left(A\circ X\right)\rightarrow A\circ \left( Y \wedge X\right) $$
This coincides with the tensorial strength $\sigma_{Y,X}$ of the endofunctor $(A\circ)$, which yields hypothesis $(c)$.

It remains to verify that for a well-pointed $\Gamma$-theory $A$, the monad $(A\circ-)$ on $\Gamma$-spaces satisfies hypothesis $(d)$ of our 
homotopical Morita theorem. 
Since this monad preserves strong cofibrations (by well-pointedness of $A$), Lemma \ref{lemRealExc} and Proposition \ref{propForgfuncRE} of 
Section \ref{SRealMModCat} show that hypothesis $(d)$ reduces to Corollary \ref{corCofGS}.

$X \to Y$ is a cofibration and $X,Y,Z$ are cofibrant.
By Proposition \ref{PropStDay}, where we suppose that $\mathcal E=SSet_{*}$ and $\mathcal A=\Gamma^{op}$, the strong monad $A\circ-$ on $\mathcal{E}_*$
induces a strong monad $A\circ-$ on $\Gamma$-spaces and hence,
by Proposition \ref{propmonTI}, an endomorphism monoid $A\circ \Gamma^1$ in $\Gamma$-spaces. 
Furthermore, by Proposition \ref{propGring} the endomorphism monoid of the associated strong monad $A\circ-$ on $\Gamma$-spaces may be identified with the Gamma-ring $A^S$ induced by the assembly map.

Therefore our homotopical Morita theorem shows that the monad morphism $$\lambda: -\wedge A^{S}  \rightarrow A\circ- $$ induces a Quillen equivalence between the category of $A^S$-modules and the category of $(A\circ)$-algebras:
$$ Ho\left( Mod_{A^{S}}\right) \simeq Ho\left( \mathcal Alg_{(A\circ)}\right) $$
\end{proof}

\begin{rem}
Schwede's original statement is slightly more general than ours insofar as he imposes no restriction at all on the $\Gamma$-theory $A$.
He is able to do so by cleverly using the monoid axiom (cf. Definition \ref{defmonax}) at all places where we use the cofibrancy of the underlying $\Gamma$-space $A$.
It should however be noticed that our cofibrancy condition is not as restrictive as that since we are using Bousfield-Friedlander's cofibrations.
\end{rem}
\begin{rem}

Theorem \ref{ThSch} has its intrinsic limitation in the fact that $\Gamma$-spaces solely model connective spectra. Lydakis \cite{LydSFSHT} proves that we can embed the category of 
$\Gamma$-spaces ($\mathcal E=SSet$) into the category of strong endofunctors of $\mathcal E_{*}$ which are determined
(by enriched left Kan extension) by their values on the simplicial sets of finite presentation (i.e. having only a finite number of non-degenerated simplices or, equivalently, having a compact geometric realization).
This actually is a category (the hom-sets are small) and we can define a model structure which extends in a certain sense the one on $\Gamma$-spaces (one more time there is a strict version and a stable version). 
Lydakis \cite{LydSFSHT} (cf. also \cite{MMSSDSp}) proves that the stable version provides a model for all spectra.
Moreover, this category admits two monoidal structures: one (non-symmetric) corresponding to the composition of endofunctors and the other (symmetric) corresponding to the smash-product. The monoids for this smash-product are precisely the FSP's (Functor with Smash Product) of Bökstedt.

The homotopical Morita theorem then allows
(if one verifies that the category of strong endofunctors ``with compact support'' is stable monoidal model category and further that the axioms of the theorem are satisfied, essentially the verifications done in Section \ref{sAsmPpax}, but now in the case of strong endofunctors ``with compact support'') 
to associate to a strong monad ``with compact support'' a FSP of Bökstedt such that the category of algebras of the monad is Quillen equivalent to the category of modules for the corresponding FSP. 

\end{rem}

\bibliographystyle{alpha.bst}
\bibliography{TheseEng}
\nocite{MayGeoIterLoopSp, BerBoardVogtRes, BorHCA,EilCC,EKMRMASH, KelBCECT,KockSF,KockMS,MaclCWM,QuiHA,SchMTADSMC,StrFTM,SchSHAT,HovMC,LackFTMII,BungeRFCCA,SchAMMMC,GoubLRMT,LydSG, LydSFSHT, DwyHLFMCHC,ShiMTSHT,BouHT,KelRE2C,BMDCAO,BIWPSCILS,LCCA, SECCT, SStHomAlgGS, GDCA, GPCCA, QRatHTh,BMENRC, MMSSDSp, TabMRCLT, BarLRModC}

\begin{titlepage}
\vspace*{2cm}
\begin{center}
{\Large\bf R\'esum\'e}\\
\vspace*{0.5cm}
\end{center}

Nous développons une version homotopique de la théorie de Morita classique en utilisant la notion de monade forte. 
C'était Anders Kock qui a montré qu'une monade $T$ dans une catégorie monoidale $\mathcal{E}$ est forte si et seulement si la monade $T$ est enrichie.
Nous montrons que cette correspondance entre force et enrichissement se traduit par un 2-isomorphisme de 2-catégories. 
Sous certaines conditions sur la monade $T$,
nous montrons que la catégorie homotopique des $T$-algèbres est équivalente au sens de Quillen à la catégorie homotopique des modules sur le monoïde d'endomorphismes de la $T$-algèbre $T(I)$ librement engendré par
l'unité $I$ de $\mathcal{E}$.
Dans le cas particulier où $\mathcal{E}$ est la catégorie des $\Gamma$-espaces de Segal munie de la structure de modèle stable de Bousfield-Friedlander et $T$ est la monade forte associée à une $\Gamma$-théorie bien pointée, 
nous retrouvons un théorème de Stefan Schwede, comme corollaire du théorème homotopique de Morita.     

\textbf{Mots-clés}: Equivalence de Morita, Monade forte, Monade enrichie, Catégorie de modèles, Homotopie stable, Gamma espaces. 
\vspace*{2cm}
\begin{center}
{\Large\bf Abstract}\\
\vspace*{0.5cm}
\end{center}

We develop a homotopy theoretical version of classical Morita theory using the notion of a strong monad. 
It was Anders Kock who proved that a monad $T$ in a monoidal category $\mathcal{E}$ is strong if and only if $T$ is enriched in $\mathcal{E}$. 
We prove that this correspondence between strength and enrichment follows from a 2-isomorphism of 2-categories. 
Under certain conditions on $T$,
we prove that the category of $T$-algebras is Quillen equivalent to the category of modules over the endomorphism monoid of the $T$-algebra
$T(I)$ freely generated by the unit $I$ of $\mathcal{E}$.
In the special case where $\mathcal{E}$ is the category of $\Gamma$-spaces equipped with Bousfield-Friedlander's stable model structure and $T$ is the strong monad associated to a well-pointed $\Gamma$-theory, 
we recover a theorem of Stefan Schwede, as an instance of a general homotopical Morita theorem. 

\textbf{Key-words}: Morita equivalence, Strong monad, Enriched monad, Model category, Stable homotopy theory, Gamma spaces.    
\end{titlepage} 
\end{document}